\documentclass[12pt]{amsart}

\usepackage{amsmath}
\usepackage{amsthm}
\usepackage{amssymb}
\usepackage{mathrsfs}
\usepackage{ifthen}
\usepackage{graphicx}
\usepackage{hyperref}
\usepackage{float}
\usepackage{tikz}
\usepackage{fullpage}
\usepackage{enumerate}
\usepackage{xcolor}
\usepackage{todonotes}

 \nonstopmode \numberwithin{equation}{section}

\theoremstyle{plain}

\newtheorem{Thm}{Theorem}[section]
\newtheorem{cor}[Thm]{Corollary}
\newtheorem{lemma}[Thm]{Lemma}

\newtheorem{rem}[Thm]{Remark}

\theoremstyle{definition}

\newtheorem{defn}[Thm]{Definition}
\newtheorem{result}[Thm]{Result}

\newtheorem{matrixcondition}[Thm]{Condition}

\newcommand{\Tr}{\operatorname{Tr}}

\newcommand{\Cov}{\operatorname{Cov}}
\newcommand{\norm}[1]{\left\lVert#1\right\rVert}

\allowdisplaybreaks

\title[CLT for correlated Random matrices]{CLT for LES of correlated Non-Hermitian Random Matrices}

\author{Indrajit Jana}
\address{Indian Institute of Technology, Bhubaneswar}
\email{ijana@iitbbs.ac.in}
\thanks{Indrajit Jana - \textit{Email: ijana@iitbbs.ac.in;}\\ 
}

\author{Sunita Rani}
\address{Indian Institute of Technology, Bhubaneswar}
\email{s21ma09007@iitbbs.ac.in}
\thanks{Sunita Rani - \textit{Email: s21ma09007@iitbbs.ac.in;} (Corresponding author)\\
}

\begin{document}

\maketitle
\date{\today}

\begin{abstract} We consider two $n\times n$ non-Hermitian random matrices such that the $ij$th entry of one matrix is correlated with the $ij$th entry of the other matrix. However, the entries of any particular matrix are i.i.d. random variables. We study the asymptotic behavior of the combined spectrum, and the limit of the linear eigenvalue statistic defined on the combined spectrum. We show that if the random variables are centered with variance $1/n$ and having finite moments, then the centered \textit{Linear Eigenvalue Statistics} (LESs) converge jointly to a bivariate Gaussian distribution. We assumed that the test function used in the LES belongs to Sobolev $H^{2+\delta}$ space. The variance of the limiting Gaussian distribution depends on correlation structure of the matrix entries and the fourth order mixed cumulants of the matrix entries. This generalizes the previous results by Rider, Silverstein\cite{rider2006gaussian}, Cipolloni, Erd\H{o}s, Schr\"oder \cite{cipolloni2023central, cipolloni2021fluctuation}. In particular, we obtain the limiting LES of random centrosymmetric matrices.



\end{abstract}

\tableofcontents
\addtocontents{toc}{\protect\setcounter{tocdepth}{1}} 

\section{Introduction}


In this paper, we focus on analyzing the combined spectrum of two correlated non-Hermitian random matrices. While numerous studies analyzed random matrices with correlated entries \cite{che2017universality, adhikari2019edge, erdHos2019random, alt2020correlated, alt2021inhomogeneous, cipolloni2024out, banerjee2024edge, catalano2024random}, to the best of our knowledge, there are relatively few works specifically on the joint convergence of correlated random matrices. Joint convergence of sample cross-covariance matrices \cite{bhattacharjee2023joint}, independent random elliptic matrices \cite{adhikari2019brown}, independent random Toeplitz matrices \cite{adhikari2024convergence} were studied in the cited articles. In the last three articles, a correlation structure was assumed for intra-matrix elements; however, no correlation structure was considered for inter-matrix elements. The following papers represent more recent advancements in the spectrum of correlated random matrices. One day prior to the release of the first version of this article on arXiv, Kolupaiev \cite{kolupaiev2025spectral} published another article on arXiv that studied the local fluctuations of limiting spectrum of two correlated Hermitian random matrices. In the non-Hermitian context, Bourgade et al. \cite{bourgade2024fluctuations} established the joint convergence of the linear eigenvalue statistics for a time-indexed matrix dynamics at finitely many distinct times. 


 Before continuing any further, let us introduce the following two notions related to the spectrum of a random matrix. Let $\lambda_{1}, \lambda_{2}, \ldots, \lambda_{n}$ be the eigenvalues of an $n\times n$ random matrix $M.$ We define the \textit{Empirical Spectral Measure} (ESM) on the Borel sigma algebra of $\mathbb{C}$ as
\begin{align*}
    \mu_{n}(\cdot) = \frac{1}{n}\sum_{i=1}^{n}\delta_{\lambda_{i}}(\cdot),
\end{align*}
where $\delta_{x}(A)=\mathbf{1}_{\{x\in A\}}.$ This is a random probability measure. It is known that when $M$ is a random non-Hermitian matrix, $\mu_{n}$ converges almost surely in weak topology of measures to the uniform distribution on the unit disc $\{z\in \mathbb{C}:|z|\leq 1\}.$ In other words, if $f:\mathbb{C}\to\mathbb{C}$ is a bounded continuous function, then
\begin{align*}
    \frac{1}{n}\sum_{i=1}^{n}f(\lambda_{i})\stackrel{\text{a.s.}}{\to}\frac{1}{\pi}\int_{\mathbb{C}}f(z)\;\mathrm{d}^{2}z\;\;\text{as }n\to\infty,
\end{align*}
where $\mathrm{d}^{2}z=\mathrm{d}\Re z\;\mathrm{d}\Im z.$ This is known as \textit{Circular law} \cite{girko1985circular, bai1997circular, tao2008random, pan2010circular, gotze2010circular}.

In the above, $\sum_{i=1}^{n}f(\lambda_{i})=:\operatorname{L}_{n}(f)$ is referred as the \textit{Linear Eigenvalue Statistic} (LES) of the matrix $M.$ The fluctuations of the above convergence can be analyzed by studying the asymptotic behavior of $a_{n}(\operatorname{L}_{n}(f)-b_{n})$ for some sequences $\{a_{n}\}_{n},$ $\{b_{n}\}_{n}$ and some test function $f.$ Typically, $b_{n}=\mathbb{E}[\operatorname{L}_{n}(f)],$ but unlike the classical Central Limit Theorem (CLT), $a_{n}$ is not $\mathcal{O}(1/\sqrt{n}),$ primarily because the eigenvalues $\{\lambda_{i}\}_{1\leq i\leq n}$ are not independent.

In our case, we have two $n\times n$ random non-Hermitian matrices whose entries are correlated with each other. We study the LES of the joint spectrum via a linear combination of the individual LESs as mentioned in \eqref{eq:L_n(f)_First}. In \cite{bourgade2024fluctuations}, the LES of matrix dynamics of the form $X_{t}:=e^{-t/2}X + \sqrt{1-e^{-t}} \tilde{X}$ were studied. Here, $X$ is a non-Hermitian matrix with i.i.d. entries, and $\tilde{X}$ is an independent Ginibre ensemble. It was shown that the joint LES (at both macroscopic and mesoscopic scales) of $X_{t}$ and $X_{s}$ for distinct $s$ and $t$ converge to a bivariate Gaussian distribution. However, our result is not directly comparable. The main contribution of this paper is the Theorem \ref{Thm:CLT for LES}, which says that the limiting LES follows a Gaussian distribution. The variance of the limiting Gaussian distribution depends on the fourth mixed cumulant and the correlation structure of the matrix entries. Our proof relies on the methods developed by Erd\H{o}s et. al \cite{cipolloni2023central, cipolloni2021fluctuation}. However, since in this paper we are analyzing two correlated matrices instead of one, we had to consider the interaction between the matrices in several places such as in Section \ref{sec:Reductions}, where we needed to approximate product of resolvents of correlated matrices. The correlation structure starts to appear extensively in technical calculations from Subsection \ref{subsec: I_3^1} onward. Broadly, the LES is converted into an integral of trace of resolvent by using Girko's representation formula \eqref{eq:Girko's formula}. Then it is shown that asymptotically the product of the traces of the resolvents at different points decomposes as per the Isserlis' theorem, which in turn implies that the liming trace of the resolvent is a Gaussian process. 

As a consequence of Theorem \ref{Thm:CLT for LES}, we also obtain the limiting LES for random Centrosymmetric matrices, which was discussed in our previous work \cite{jana2024spectrum}. This article is organized as follows. The main result and the consequences are stated in Section \ref{sec:main result}. The proof of the CLT for LES is provided in Section \ref{sec: clt for les}. However, the key part of the proof relies on establishing the CLT for resolvents, which is presented in Section \ref{sec: clt for resolvents}. Let us begin with introducing the notations, which will be frequently used in this article.

\section{Notations}
Throughout the paper, we adopt the following notations. Let $\mathbf{H} := \{z \in \mathbb{C} \mid \Im{z} > 0\}$ denote the upper half-plane, and let $\mathbf{D} \subset \mathbb{C}$ be the open unit disk. The integration with respect to $\mathrm{d}^{2}z$ is equivalent to the integration with respect to $\mathrm{d}\Re z\;\mathrm{d}\Im z.$ For positive functions $f$ and $g$, we use the notations $f \lesssim g$ and $f\sim g$ to indicate that $f \leq Cg$ and $df\leq g\leq Df$ for some constants $C, d, D>0.$ For $\nu_{1}, \nu_{2} > 0,$ the notation $\nu_1 \ll \nu_2$ means that $\nu_1 \leq c \nu_2$ for sufficiently small constant $c > 0$.

In this paper, we are considering two correlated random matrices $X^{(1)},X^{(2)}.$ We shall use the notation $\mathcal{M}^{(t)}$ to represent a matrix which is a function of the matrix $X^{(t)}$ for $t=1, 2.$ However, $\mathcal{M}^t$ denotes the transpose of the matrix $\mathcal{M}.$ The symbol $\mathbb{I}$ denotes the identity matrix of an appropriate size, while $\mathbf{1}=(1, 1, \ldots, 1)$ denotes a column vector of an appropriate size. For an $n \times n$ matrix $\mathcal{M}$, we define $\langle \mathcal{M} \rangle := \frac{1}{n} \Tr{\mathcal{M}}$. For vectors $x, y \in \mathbb{C}$, the inner product is defined by $\langle x, y \rangle := \sum \bar{x_i} y_i$, and for matrices $\mathcal{M}_1, \mathcal{M}_2 \in \mathbb{C}^{2n \times 2n}$, we define $\langle \mathcal{M}_1, \mathcal{M}_2 \rangle := \langle \mathcal{M}_1^{*} \mathcal{M}_2 \rangle$, where $\mathcal{M}_1^*$ is the conjugate transpose of $\mathcal{M}_1.$ 

We use the notation $\kappa(\zeta_{i_{1}}, \zeta_{i_{2}}, \ldots, \zeta_{i_{l}})$ to denote a $l$-th order joint cumulant of the random variables $\zeta_{1}, \zeta_{2}, \ldots, \zeta_{p}.$ It is defined as follows;
\begin{align*}
    \log\Big(\mathbb{E}\big[\iota\sum_{k=1}^{p}s_{k}\zeta_{k}\big]\Big)=\sum_{l \geq 0} \iota^{l}\sum_{i_1, \ldots, i_l=1}^p \frac{\kappa(\zeta_{i_1}, \zeta_{i_2}\dots, \zeta_{i_l})}{g_{1}!g_{2}!\cdots g_{p}!}s_{1}^{g_{1}}s_{2}^{g_{2}}\cdots s_{p}^{g_{p}},
\end{align*}
where $g_{j}=\#\{i_{q},1\leq q\leq l|i_{q}=j\}.$ We note down a fact that if the random variables $\zeta_{1}, \zeta_{2}, \ldots, \zeta_{p}$ are independent, then $\kappa(\zeta_{i_1}, \zeta_{i_2}\dots, \zeta_{i_l})=0$ whenever $g_{j} < l$ for all $1\leq j\leq p$ \cite[Theorem 7.12]{speicher2023high}.

We say that a sequence of events $\{F_{n}\}_{n}$ occur with high probability if for any arbitrary $K>0,$ there exist $n_{0}=n_{0}(K)\in \mathbb{N}$ such that $\mathbb{P}(F_{n}^{c})<n^{-K}$ for all $n\geq n_{0}.$  
We also use the notion of stochastic domination, which is defined in Definition \ref{defn:Stochastic domination}.
\begin{defn}[Stochastic domination]\label{defn:Stochastic domination}
    Consider families of non-negative random variables,  
    $$ 
    X=\left\{X_{n}(\vartheta) \mid n \in \mathbb{N}, \vartheta \in \Theta^{(n)}\right\}, \quad
    Y=\left\{Y_{n}(\vartheta) \mid n \in \mathbb{N}, \vartheta \in \Theta^{(n)}\right\},
    $$  
    where $\Theta^{(n)}$ is a parameter set. If for any $\delta, K > 0$, there exists $n_{0}=n_0(\delta, K)\in \mathbb{N}$ such that
    $$ 
    \sup _{\vartheta \in \Theta^{(n)}} \mathbb{P}\left[X_{n}(\vartheta) > n^\delta Y_{n}(\vartheta)\right] \leq n^{-K},\;\;\forall\;n\geq n_{0}, 
    $$  
    then $X$ is said to be stochastically dominated by $Y,$ denoted by $X \prec Y$.
\end{defn}


\section{Main result}\label{sec:main result}
Let us consider a $2n\times 2n$ ordered block matrix $X$ defined as follows;
\begin{align}\label{eqn: main definition of block matrix X}
    X=\left(\begin{array}{ll}(X^{(1)})_{n\times n} & 0 \\ 0 & (X^{(2)})_{n\times n}\end{array}\right),
\end{align}
where $X^{(1)}$ and $X^{(2)}$ are i.i.d. random matrices with mean $0$ and variance $1/n$ for $i=1,2$ respectively. The $(i,j)^{th}$ entry of $X^{(1)}$ is correlated to the $(i,j)^{th}$ entry of $X^{(2)}$ as follows;

\begin{equation}\label{eq:cov}
    \begin{aligned}
        &\gamma := n\Cov\big((X^{(1)})_{(i,j)},(X^{(2)})_{(i,j)}\big) = n\Cov(X_{(i,j)},X_{(n+i,n+j)}),
    \end{aligned}
\end{equation}

and assume that $\{\sigma_i\}_{i=1}^{2n}$ are eigenvalues of $X.$ Without loss of generality, let us assume that $\{\sigma_i\}_{i=1}^{n}$ are eigenvalues of $X^{(1)}$ and $\{\sigma_i\}_{i=n+1}^{2n}$ are eigenvalues of $X^{(2)}$.
The limiting ESM of $X^{(1)}$ and $X^{(2)}$ follow the Circular law. Consequently, the limiting ESM of $X$ follows the Circular law as well.

For test functions $f$ defined on the spectrum of $X$, let us define the linear eigenvalue statistics as
\begin{align}\label{eq:L_n(f)_First}
    \mathcal{L}_n(f) := c \mathcal{L}^{(1)}_n(f)+ d \mathcal{L}^{(2)}_n(f),
\end{align}
where $c, d\in \mathbb{C}$ and
\begin{align*}
    \mathcal{L}^{(1)}(f)&:=\sum_{i=1}^n f\left(\sigma_i\right)-\mathbb{E} \sum_{i=1}^n f\left(\sigma_i\right),\\
     \mathcal{L}^{(2)}_n(f)&:=\sum_{i=n+1}^{2n} f\left(\sigma_i\right)-\mathbb{E} \sum_{i=n+1}^{2n} f\left(\sigma_i\right).
\end{align*}
From \cite[Theorem 2.2]{cipolloni2023central} and \cite[Theorem 2.2]{cipolloni2021fluctuation}, we have
\begin{align}
    \mathcal{L}^{(1)}(f)&\stackrel{d}{\to}\mathcal{N}(0, V^{(1)}_{f})\label{eq:LES_X_1}, \\
    \mathcal{L}^{(2)}(f)&\stackrel{d}{\to}\mathcal{N}(0, V^{(2)}_{f}),\label{eq:LES_X_2} 
\end{align}
as $n\to\infty,$ where $\stackrel{d}{\to}$ denotes convergence in distribution, and $V_{f}^{(1)}, V_{f}^{(2)}$ are the variances of the limiting normal distribution, which depend on the test function $f$ and cumulants of the matrix entries. However, $\mathcal{L}^{(1)}(f)$ and $\mathcal{L}^{(2)}(f)$ are not independent due to \eqref{eq:cov}. Therefore, the limiting behavior of $\mathcal{L}_{n}(f)$ can not be derived from the above limits. The main result of the current paper is the proof of
\begin{equation}\label{eq:L_n(f)_Second}
    \mathcal{L}_n(f) \xrightarrow{\  d  \ } \mathcal{N}\left(0, V_{f}\right),
\end{equation}
for non-Hermitian random matrices with correlated entries \eqref{eq:cov} and for general test functions $f$. This also shows that $(\mathcal{L}^{(1)}(f), \mathcal{L}^{(2)}(f))$ converges to a bivariate normal distribution. The main idea of this article makes use of the following Girko's formula \cite{10.1214/13-AOP876}
\begin{equation}\label{eq:Girko's formula}
    \sum_{\sigma \in \operatorname{Spec}(X^{(t)})} f(\sigma)=-\frac{1}{4 \pi} \int_{\mathbb{C}} \Delta f(z) \int_0^{\infty} \Im{\operatorname{Tr} (G^{(t)})^z(\iota \eta)} \mathrm{d} \eta \mathrm{d}^2 z,
\end{equation}
which expresses the linear eigenvalue statistics for any smooth compactly supported test functions in terms of the integral of the trace of the resolvent of  
\begin{equation}{\label{eq:(H)^z}}
    (H^{(t)})^z:=\left(\begin{array}{cc}
0 & X^{(t)}-z \\
(X^{(t)})^*-\bar{z} & 0
\end{array}\right), t=1, 2, 
\end{equation}
where $z \in \mathbb{C}$. The resolvent of $(H^{(t)})^{z}$ is denoted by $(G^{(t)})^{z}(\omega):=((H^{(t)})^{z}-\omega)^{-1}.$

Before proceeding any further, we list down our assumptions on the matrix $X$, which is defined in \eqref{eqn: main definition of block matrix X}. 

\begin{matrixcondition}\label{Condtion::}
    Let $X^{(t)}=[x^{(t)}_{ab}]_{1\leq a, b, \leq n}, t=1, 2,$ be two random matrices such that $x_{ab}^{(t)}\stackrel{\text{i.i.d.}}{\sim}n^{-1/2}\chi^{(t)},$ where $\chi^{(t)}$ satisfy the following conditions;
\begin{enumerate}[(i)]
  \item $\mathbb{E}[\chi^{(t)}] = 0,$ $\mathbb{E}[|\chi^{(t)}|^2] = 1,$ and $\mathbb{E}[(\chi^{(t)})^2] = \gamma_{1},$
  where $$
\gamma_{1} = 
\begin{cases} 
    1, & \text{if } \chi^{(t)} \text{ is a real valued,} \\
    0, & \text{if } \chi^{(t)} \text{ is a complex valued.}
\end{cases}
$$
\item $\mathbb{E}[\chi^{(1)}\chi^{(2)}]=\rho,$ $\mathbb{E}[\chi^{(1)}\overline{\chi^{(2)}}]=\gamma.$
\item There exist constants $C_p >0$, for any $p \in \mathbb{N}$, such that 
$$\mathbb{E}[|\chi^{(t)}|^p] \leq C_p.$$
\end{enumerate}
\end{matrixcondition}

Essentially, we start with an i.i.d. set of bivariate random variables $(\chi^{(1)}, \chi^{(2)})$. Then we create the matrix $X^{(1)}$ out of $n^{-1/2}\chi^{(1)}$s and $X^{(2)}$ out of $n^{-1/2}\chi^{(2)}$s. Individually, each matrix $X^{(1)}$ and $X^{(2)}$ are matrices with i.i.d. entries. Therefore, individually the limiting empirical distribution of the eigenvalues of $X^{(1)}$ and $X^{(2)}$ will follow the circular law. To analyze the fluctuations around the circular law, we consider the linear eigenvalue statistic defined in \eqref{eq:L_n(f)_First}. We take the test functions from the Sobolev space $H_{0}^{2+\delta}(\Omega),$ where $\Omega \supset \overline{\mathbf{D}}:=\{z\in \mathbb{C}:|z|\leq 1\}.$ Here the space $H_{0}^{2+\delta}(\Omega)$ is the closure of the space of compactly supported smooth functions having the following finite norm
\begin{align*}
    \|f\|_{H^{2+\delta}(\Omega)}:=\left\|(1+|\xi|)^{2+\delta} \hat{f}(\xi)\right\|_{L^2(\Omega)},
\end{align*}
where $\hat{f}$ is the Fourier transform of $f.$

\begin{Thm}[Central Limit Theorem for linear eigenvalue statistics]\label{Thm:CLT for LES}
    Let $X^{(1)}$ and $X^{(2)}$ be two $n\times n$ random matrices satisfying Condition \ref{Condtion::}. Let $f\in H^{2+\delta}(\Omega)$ be a test function, where $\delta > 0$ and $\overline{\mathbf{D}}\subset \Omega\subset\mathbb{C}.$ Then the centered linear statistics $\mathcal{L}_n{(f)}$, defined in \ref{eq:L_n(f)_First}, converges in distribution to a complex Gaussian random variable $\mathcal{L}(f),$
where $\mathbb{E}[\mathcal{L}(f)]=0,$ $\mathbb{E}[\mathcal{L}(f)^{2}]= C(\bar{f},f),$ and $\mathbb{E}[|\mathcal{L}(f)|^2]= C(f,f):=V_f$. The covariance kernel $C(g, f)$ is given by
\begin{equation}
    \begin{aligned}\label{eq:covariance kernel}
        C(g,f) := &-\frac{1}{8{\pi}^2} \int_{\mathbb{C}} \mathrm{d}^2 z_1 \int_{\mathbb{C}} \mathrm{d}^2 z_2 \Delta f\left(z_1\right) \Delta \overline{g\left(z_2\right)} \int_0^{\infty} \mathrm{d} \eta_1 \int_0^{\infty} \mathrm{d} \eta_2 L.
    \end{aligned}
\end{equation}
The above notations are defined as follows;
\begin{equation}\label{notation: limit notations}
    \begin{aligned}
        &L=c^2\widehat{V_{ij}^{(1)}}+d^2\widehat{V_{ij}^{(2)}}+8cd\gamma (L_{ji}^{12}+L_{ji}^{21})+2cd\rho  N_{ij} +c^2(\kappa_4)_{11}U_i^{(1)}U_j^{(1)}\\
        &\;\;\;\;+d^2(\kappa_4)_{22}U_i^{(2)}U_j^{(2)}+cd(\kappa_4)_{12} \big(U_i^{(1)}U_j^{(2)}+U_j^{(1)}U_i^{(2)}\big),\\
        &(\kappa_4)_{ts}= \kappa\left(\chi^{(t)},\bar{\chi}^{(t)}, \chi^{(s)}, \bar{\chi}^{(s)}\right),\\
        &\widehat{V_{ij}^{(t)}}\left(z_i, z_j, \eta_i, \eta_j\right)\\
        &=
    \begin{cases} 
        V_{ij}^{(t)}\left(z_i, z_j, \eta_i, \eta_j\right), \text{if $X^{(t)}$ is $\mathbb{C}$-valued}, \\ 
        V_{ij}^{(t)}\left(z_i, z_j, \eta_i, \eta_j\right) + V_{ij}^{(t)}\left(z_i, \overline{z_j}, \eta_i, \eta_j\right),\text{if $X^{(t)}$ is $\mathbb{R}$-valued},
    \end{cases}\\
    &V_{ij}^{(t)}\left(z_i,z_j,\eta_i,\eta_j\right)\\ 
    &=\frac{1}{2} \partial_{\eta_i} \partial_{\eta_j} \log \Big[1+\big(u_i^{(t)} u_j^{(t)}|z_i||z_j|\big)^2-(m_i^{(t)})^2 (m_j^{(t)})^2-2 u_i^{(t)} u_j^{(t)} \Re z_i \overline{z_j}\Big],\\
         &L_{ji}^{12} =\frac{ m_j^{(1)}m_i^{(2)} u_j^{(1)}u_i^{(2)}\Re{(\overline{z_i}z_j)}}{\beta_j^{(1)}\beta_i^{(2)}},
        L_{ji}^{21} =\frac{m_j^{(2)} m_i^{(1)}u_j^{(2)}u_i^{(1)} \Re{(\overline{z_i}z_j)}}{\beta_j^{(2)}\beta_i^{(1)}},\\
       & N_{ij}=\big( \partial_{\eta_i}(m_i^{(1)})\big)\big( \partial_{\eta_j}(m_j^{(2)})\big)+\big( \partial_{\eta_i}(m_i^{(2)})\big)\big( \partial_{\eta_j}(m_j^{(1)})\big),\\
        &U_i^{(t)} = \frac{i}{\sqrt{2}} \partial_{\eta_i} (m_i^{(t)})^2,\\ 
    \end{aligned}
\end{equation}
$m_i^{(t)}=(m^{z_i})^{(t)}(\iota\eta_i)$ is the solution of the equation \eqref{eq:MDS}, $u_{i}=u^{z_i}(\iota\eta_i)=-\frac{\iota \Im m_{i}}{\eta_{i}+\Im m_{i}},$ and $\beta_{i}^{(t)}=1-(m_{i}^{(t)})^{2} - (u_{i}^{(t)})^{2}|z_{i}|^{2}.$
\end{Thm}
The above theorem is proved using the resolvents of the matrix $(H^{(t)})^{z},$ which is $(G^{(t)})^{z}.$ Using \eqref{eq:Girko's formula}, the linear eigenvalue statistics $\mathcal{L}_{n}(f)$ can be expressed in terms of $(G^{(t)})^{z}$ as follows; 

\begin{align}\label{eq:J_T notations}
& \mathcal{L}^{(t)}_n(f)\\
=& \frac{1}{4 \pi} \int_{\mathbb{C}} \Delta f(z)\left[\log|\operatorname{det}\left((H^{(t)})^z-iT\right)|-\mathbb{E} \log |\operatorname{det}\left((H^{(t)})^z-iT\right)|\right] \mathrm{d}^2z\notag \\
&-\frac{n}{2 \pi \iota} \int_{\mathbb{C}} \Delta f\left[\left(\int_0^{\eta_0}+\int_{\eta_0}^{\eta_c}+\int_{\eta_c}^T\right)\left[\langle (G^{(t)})^z(\iota\eta)-\mathbb{E} (G^{(t)})^z(\iota\eta)\rangle \right] \mathrm{d} \eta\right] \mathrm{d}^2 z\notag \\
=&: J_T^{(t)}+(I_0^{\eta_0})^{(t)} +(I_{\eta_0}^{\eta_c})^{(t)}+(I_{\eta_c}^T)^{(t)}.\notag
\end{align}

Here, we have decomposed the whole integral into four parts, namely $(0, \eta_{0}),$ $[\eta_{0}, \eta_{c}),$ $[\eta_{c}, T),$ and $[T,\infty),$ where $\eta_{0} = n^{-1-\delta_{0}},$ $\eta_{c}=n^{-1+\delta_{1}},$ and $T=n^{C},$ for some small constants $\delta_{0}, \delta_{1} > 0$ and a large positive constant $C.$ In our context, the linear eigenvalue statistic $\mathcal{L}_{n}(f)=c\mathcal{L}_{n}^{(1)}(f)+d\mathcal{L}_{n}^{(2)}(f),$ which can similarly be expressed as
\begin{equation}\label{eq:L_n(f)}
   \begin{aligned}
& \mathcal{L}_n(f) \\
&= c \mathcal{L}^{(1)}_n(f)+d \mathcal{L}^{(2)}_n(f)\\
&=: J_T+I_0^{\eta_0}+I_{\eta_0}^{\eta_c}+I_{\eta_c}^T,\\
  \end{aligned} 
\end{equation}
where 
\begin{align*}
    J_T &= cJ_T^{(1)} + dJ_T^{(2)},&I_0^{\eta_0} &= c(I_0^{\eta_0})^{(1)} + d(I_0^{\eta_0})^{(2)},\\
    I_{\eta_0}^{\eta_c} &= c(I_{\eta_0}^{\eta_c})^{(1)} + d(I_{\eta_0}^{\eta_c})^{(2)},&I_{\eta_c}^T &= c(I_{\eta_c}^T)^{(1)} + d(I_{\eta_c}^T)^{(2)}.
\end{align*}
Now, we handle $J_T$, $I_0^{\eta_0}$, $I_{\eta_0}^{\eta_c}$, and $I_{\eta_c}^T$ using different techniques. In particular, from \eqref{eqn: bound on J_T, etc}, we observe that the contribution of $J_T$ is negligible. To estimate $I_0^{\eta_0}$, we use the fact that, with high probability, there are no eigenvalues in the interval $[0, \eta_0]$, as established in \cite[Theorem 3.2]{Tao2008SmoothAO}. The contribution of $I_{\eta_0}^{\eta_c}$ is argued to be small in Lemma \ref{Lemma:J_T^t_bound}. The main contributing factor in \eqref{eq:L_n(f)} is $I_{\eta_{c}}^{T},$ which constitutes the majority of this paper. Essentially, we need to analyze the limiting behaviors of $\langle (G^{(1)})^z(\iota\eta)-\mathbb{E} (G^{(1)})^z(\iota\eta)\rangle$, and $\langle (G^{(2)})^z(\iota\eta)-\mathbb{E} (G^{(2)})^z(\iota\eta)\rangle$ in the regime $[\eta_{c}, T].$ This is formulated in the following Theorem \ref{Thm:CLT for resolvents}.

\begin{Thm}[CLT for resolvents]\label{Thm:CLT for resolvents} Let $ G_i^{(t)}:=(G^{(t)})^{z_i}\left(\iota \eta_i\right),$ where $\eta_{i} > 0$ for $1\leq i\leq p.$ Let $z_{1}, z_{2}, \ldots, z_{p}\in \mathbb{C}$ be distinct complex numbers such that whenever $i\neq j,$ $\min\{\eta_i,\eta_j\}\geq n^{\epsilon-1}|z_i-z_j|^{-2}$ for some $\epsilon > 0.$ Then we have
    \begin{equation}\label{eq:CLT_.For_.Resolvent}
    \begin{aligned}
        &\mathbb{E}\Big[\prod_{i=1}^{p}c\langle G^{(1)}_i-\mathbb{E}(G^{(1)}_i)\rangle +d\langle G^{(2)}_i-\mathbb{E}(G^{(2)}_i)\rangle\Big]\\
        &=\sum_{\mathcal{P} \in \text {Parings}[p]} \prod_{\{i, j\} \in \mathcal{P}} \mathbb{E}\Big[\big\{c\langle G^{(1)}_i-\mathbb{E}(G^{(1)}_i)\rangle +d\langle G^{(2)}_i-\mathbb{E}(G^{(2)}_i)\rangle\big\}\\
        &\;\;\;\;\;\;\;\;\;\times\big\{c\langle G^{(1)}_j-\mathbb{E}(G^{(1)}_j)\rangle +d\langle G^{(2)}_j-\mathbb{E}(G^{(2)}_j)\rangle\big\}\Big] + \mathcal{O}_{\prec}(\Psi)\\
        &= \frac{1}{n^{p}}\sum_{\mathcal{P} \in \text {Parings}[p]} \prod_{\{i, j\} \in \mathcal{P}} \bigg\{\frac{c^2\widehat{V_{ij}^{(1)}}+d^2\widehat{V_{ij}^{(2)}} +8cd\gamma (L_{ji}^{12}+L_{ji}^{21})+2cd\rho  N_{ij} }{2}\\
        &\;\;\;\;\;+ \frac{c^2(\kappa_4)_{11}U_i^{(1)}U_j^{(1)}+d^2(\kappa_4)_{22}U_i^{(2)}U_j^{(2)}+cd(\kappa_4)_{12} (U_i^{(1)}U_j^{(2)}+U_j^{(1)}U_i^{(2)})}{2}\bigg\}\\
        &\;\;\;\;\;+\mathcal{O}_\prec(\Psi),
    \end{aligned}
\end{equation}
where
\begin{align}\label{eq:hat_V_{ij}^{(t)}}
  &\widehat{V_{ij}^{(t)}}\left(z_i, z_j, \eta_i, \eta_j\right)\\ &=
    \begin{cases} 
        V_{ij}^{(t)}\left(z_i, z_j, \eta_i, \eta_j\right), & \text{if $X^{(t)}$ is $\mathbb{C}$-valued}, \notag\\ 
        V_{ij}^{(t)}\left(z_i, z_j, \eta_i, \eta_j\right) + V_{ij}^{(t)}\left(z_i, \overline{z_j}, \eta_i, \eta_j\right), & \text{if $X^{(t)}$ is $\mathbb{R}$-valued},
    \end{cases}\\ 
  & \text{ and } V_{ij}^{(t)}\left(z_i,z_j,\eta_i,\eta_j\right)\\ &=\frac{1}{2} \partial_{\eta_i} \partial_{\eta_j} \log \Big[1+\big(u_i^{(t)} u_j^{(t)}|z_i||z_j|\big)^2-(m_i^{(t)})^2 (m_j^{(t)})^2-2 u_i^{(t)} u_j^{(t)} \Re z_i \overline{z_j}\Big]\notag,
\end{align}
\begin{align}\label{eq:Psi_error_term}
    &\Psi\\
    =&\bigg(\frac{1}{(n\eta_*)^2}+\frac{1}{n^2 \eta_*|z_i-z_j|^4} + \frac{1}{(n\eta_*)^3 |z_i-z_j|^4}+\frac{1}{\Im{z_i}(n\eta_*)^{3/2}} \bigg) \prod_{i\in [p]}\frac{(n\eta_i)^{-1/2}}{|1-|z_i||},\notag
\end{align}
where $\eta_{*}:=\displaystyle\min_i\{\eta_{i}\}.$
The other notations are same as defined in \eqref{notation: limit notations}. 
\end{Thm}

We would like to conclude this section by noting down the following consequences.
\begin{enumerate}[(A)]
    \item Recall from \eqref{eq:L_n(f)_First} that $\mathcal{L}_n(f) = c \mathcal{L}^{(1)}_n(f)+ d \mathcal{L}^{(2)}_n(f).$ Since Theorem \ref{Thm:CLT for LES} is true for all $c, d,$ it implies that $(\mathcal{L}^{(1)}_n(f), \mathcal{L}^{(2)}_n(f))$ converges to a bivariate Gaussian distribution.
    
    \item If the random matrices defined in the Condition \ref{Condtion::} are independent of each other, then $\gamma = 0 = \rho.$ Moreover, the joint cumulants $(\kappa_{4})_{12}=0.$ As a result the Theorem \ref{Thm:CLT for resolvents} becomes equivalent to \cite[Proposition 3.3]{cipolloni2023central} when the matrices are complex valued, and \cite[Proposition 3.3]{cipolloni2021fluctuation} when the matrices are real valued.

    \item If we begin with a $2n\times 2n$ random centrosymmetric matrix $X$, then $X$ can be represented as \eqref{eqn: main definition of block matrix X}, where the matrices $X^{(1)}$ and $X^{(2)}$ are $n\times n$ random matrices with i.i.d. entries. Moreover, there exist independent random variables $\zeta$ and $\theta$ such that the entries of $X^{(1)}$ have the same distribution as $\zeta+\theta,$ and the entries of $X^{(2)}$ have the same distribution as $\zeta-\theta$ (see \cite[Theorem 9]{weaver1985centrosymmetric}).  This impels that the $(ab)$th entry of $X^{(1)}$ and $(ab)$th entry of $X^{(2)}$ are uncorrelated but may not be independent. In other words, according to the notations used in Condition \ref{Condtion::}, $\gamma = 0=\rho.$ But $(\kappa_{4})_{12}$ is not necessarily zero.

    However, if in addition, we assume that the test function $f$ is analytic over $\overline{\mathbf{D}},$ then in the expression of the covariance kernel $C(f, g)$ (see \eqref{eq:covariance kernel}), only the integration against $c^2\widehat{V_{ij}^{(1)}}+d^2\widehat{V_{ij}^{(2)}}$ will survive. Following the calculations outlined in \cite[Lemma 4.8]{cipolloni2023central}, we obtain the expression of the limiting variance as
    \begin{align*}
        C(f, f)=\frac{2}{\pi}\int_{\mathbf{D}}|f'(z)|^{2}\;\mathrm{d}^{2}z.
    \end{align*}
    The above expression was proved independently in a previous work \cite[Theorem 3.4]{jana2024spectrum}.

    \item The main results in this paper are stated in macroscopic scale. However, we believe that the techniques used in \cite{cipolloni2024mesoscopic} can be used to extend those in mesoscopic scale as well.
\end{enumerate}

\section{Primary reductions}\label{sec:Reductions}

Proof of the main theorem requires a careful analysis of the resolvent of the matrix $(H^{(t)})^{z}$, which is defined in \eqref{eq:(H)^z}. The asymptotic behavior of the resolvent $(G^{(t)})^{z}(\omega)$ is given by the following matrix
\begin{equation}\label{eq:M^z matrix}
    \begin{aligned}
        (M^{(t)})^z(\omega):&=\left(\begin{array}{cc}(m^{(t)})^z(\omega)\mathbb{I} & -z (u^{(t)})^z(\omega)\mathbb{I} \\ -\bar{z} (u^{(t)})^z(\omega)\mathbb{I} & (m^{(t)})^z(\omega)\mathbb{I}\end{array}\right),\\
     (u^{(t)})^z(\omega):&=\frac{(m^{(t)})^z(\omega)}{\omega+(m^{(t)})^z(\omega)},
    \end{aligned}
\end{equation}
where $\mathbb{I}$ is the $n\times n$ identity matrix, and $(m^{(t)})^{z}(\omega)$ is the unique solution of the following equation
\begin{equation}\label{eq:MDS}
    -\frac{1}{(m^{(t)})^z}=\omega+(m^{(t)})^z-\frac{|z|^2}{\omega+(m^{(t)})^z},\;\;\; \eta \Im (m^{(t)})^z(\omega)>0, \;\;\; \eta=\Im \omega\neq0.
\end{equation}
The above follows from the analysis of \textit{matrix Dyson equation} (MDE) (see \cite{ajanki2019stability}). The exact quantification of the error $(G^{(t)})^{z}(\omega) - (M^{(t)})^{z}(\omega)$ is  given by the local law, which is mentioned in Theorem \ref{Thm: Local_Law}.

As it is mentioned in the statement of Theorem \ref{Thm:CLT for LES}, the derivative of $(m^{(t)})^{z}(\omega)$ with respect to $\omega$ and $\Im{\omega}=\eta$ appears frequently in the rest of the article. We note down the expression below;
\begin{equation}\label{eq:m^t_derivative}
    \partial_\omega(m^{(t)})=\frac{1-\beta^{(t)}}{\beta^{(t)}}= -\iota \partial_\eta(m^{(t)}), \quad \beta^{(t)}:=1-(m^{(t)})^2-(u^{(t)})^2|z|^2.
\end{equation}

Along the same line, we define yet another quantity $$\beta_{*}^{(t)}=1-|m^{(t)}|^{2} - |u^{(t)}|^{2}||z|^{2},$$ which satisfies
\begin{align}
    &\beta^{(t)}_* \Im m^{(t)} = (1 - \beta^{(t)}_*) \Im \omega,\\
    &(|m^{(t)}|^{2}-|u^{(t)}|^{2}|z|^{2}+1)\Re m^{(t)}(\iota \eta) = 0.\label{eqn: real part of MDE}\\
    &(\beta^{(t)})_* = \frac{\eta}{\Im m^{(t)}+\eta}, \quad \beta^{(t)}=(\beta^{(t)})_*+2(\Im m^{(t)})^2.\label{eq:beta_*} 
\end{align}
The above identities are obtained by taking the imaginary and real part of \eqref{eq:MDS} on the both sides. Now, since $\Im\omega$ and $\Im m^{(t)}$ have the same sign (from \eqref{eq:MDS}), $\beta_{*}^{(t)}$ and $(1-\beta_{*}^{(t)})$ must also have the same sign. However, as per the definition, $1-\beta_{*}^{(t)} > 0.$ Consequently, $\beta_{*}^{(t)} > 0,$ which is equivalent to  
\begin{equation}\label{eq: beta star is non-negative}
    |m^{(t)}|^2+|(u^{(t)})^z|^2|z|^2<1,
\end{equation}
which implies $|(u^{(t)})^z|^2|z|^2<1$ as well. Using this fact in \eqref{eqn: real part of MDE}, we obtain
\begin{align*}
    \Re m^{(t)}(\iota \eta) = 0.
\end{align*}
Moreover, from \eqref{eq:MDS}, we have $u^{(t)}=-(m^{(t)})^2+(u^{(t)})^2|z|^2,$ which in conjunction with \eqref{eq: beta star is non-negative} implies that $|u^{(t)}| < 1.$ Therefore, we obtain the following uniform bound in $z$ and $\omega$
\begin{equation}\label{eq:3.3}
\|(M^{(t)})^z(\omega)\|+|(m^{(t)})^z(\omega)|+|(u^{(t)})^z(\omega)| \lesssim 1.
\end{equation}

Furthermore, from \cite[eq. (3.13)]{article_1} we have
\begin{equation}\label{eq:Im{m^t}}
     \Im m^{(t)}(\iota\eta) \sim\left\{\begin{array}{ll}\eta^{1 / 3}+|1-|z|^2|^{1/2} & \text { if }|z|\leq1  \\ \frac{\eta}{|z|^2-1+\eta^{2/3}} & \text { if }|z|>1\end{array} \quad \eta\lesssim 1\right..
\end{equation}

Now, using \eqref{eq: beta star is non-negative}, \eqref{eq:Im{m^t}}, we get a lower bound on $|\beta^{(t)}|$ as follows;
\begin{align}\label{eqn: lower bound on beta}
    |\beta^{(t)}|
    \geq&1-\Re((m^{(t)})^{2})-\Re((u^{(t)})^{2})|z|^{2}\notag\\
    =&1-(\Re(m^{(t)}))^{2}+(\Im(m^{(t)}))^{2}-\Re((u^{(t)})^{2})|z|^{2}\notag\\
    \geq&(1-|m^{(t)}|^{2}-|u^{(t)}|^{2}|z|^{2})+(\Im(m^{(t)}))^{2}\notag\\
    \geq &(\Im(m^{(t)}))^{2}\notag\\
    \gtrsim &\eta^{2/3}+|1-|z|^{2}|\notag\\
    \geq &\eta^{2/3}+|1-|z||.
\end{align}
We conclude this subsection by stating the local law of the resolvent.

\begin{Thm}[Optimal local law for $G^{(t)}$,\cite{10.1214/21-EJP591}, Theorem 3.1]\label{Thm: Local_Law}
For any $\epsilon >0 $ and $z \in \mathbb{C}$ with $|1-|z|| \geq \epsilon$ the resolvent $(G^{(t)})^z$ at $\omega \in \mathbf{H}$ with $\eta = \Im{\omega}$ is very well approximated by the deterministic matrix $(M^{(t)})^z$ in the sense
    $$
   \begin{gathered}
    |\langle\big((G^{(t)})^z(\omega)-(M^{(t)})^z(\omega)\big) A\rangle| \leq \frac{C_{\epsilon}\|A\| n^{\xi}}{n \eta} \\
    |\langle \boldsymbol{x},\left((G^{(t)})^z(\omega)-(M^{(t)})^z(\omega)\right)\boldsymbol{y}\rangle| \leq C_{\epsilon} \|\boldsymbol{x}\|\|\boldsymbol{y}\| n^{\xi}\left(\frac{1}{\sqrt{ n \eta}} + \frac{1}{n \eta}  \right),
   \end{gathered}
$$
with very high probability for some $C_\epsilon \leq \epsilon^{-100}$, uniformly for $\eta> n^{-100}$, and for any deterministic matrices $A$ and any vectors $\boldsymbol{x}, \boldsymbol{y} \text{ and }\xi > 0$.
\end{Thm}

This local law will be used in various estimates. The main idea is to write $G^{(t)} = M^{(t)} + G^{(t)} - M^{(t)}.$ Then, any particular entry of $G^{(t)}$ can be written as $G_{ab}^{(t)} = M_{ab}^{(t)} + (G^{(t)} - M^{(t)})_{ab}.$ Now, the estimate of $(G^{(t)}-M^{(t)})_{ab}$ follows from the second part of the local law by taking $\boldsymbol{x} = \boldsymbol{e}_{a}$ and $\boldsymbol{y}=\boldsymbol{e}_{b}.$ Moreover, since the matrix $M^{(t)}$ is mostly a diagonal matrix, $M^{(t)}_{ab}=0$ for most of the off-diagonal $M^{(t)}_{ab}.$ A combination of these two facts yield the required estimates. Additionally, when we have $(G^{(t)}A)_{ab},$ we decompose $G^{(t)}$ as before $G^{(t)} = M^{(t)} + G^{(t)}-M^{(t)}.$ Then, we apply the first part of the local law to estimate $((G^{(t)}-M^{(t)})A)_{ab}.$ This idea is being repeatedly used in the estimates such as \eqref{eq:exp_local_law},\eqref{eq:GaGAbbGAab}, \eqref{eq:GaaGAbbGAWGab} etc..

\subsection{Local law for product of resolvents}

We begin with the shorthand notations
\begin{equation}\label{eqn: M_i notation}
    G_i^{(t)}:=(G^{(t)})^{z_i} (\omega_i),\;\;M_i^{(t)}:=(M^{(t)})^{z_i} (\omega_i).
\end{equation}
From \cite{article_1}, we note down the deviation of $G^{(t)}$ from $M^{(t)}$ in the equation \eqref{eq:deviation}, which requires the following notions
\begin{align*}
    &W^{(t)}:=\left(\begin{array}{ll}0 & X^{(t)} \\ (X^{(t)})^* & 0\end{array}\right).
\end{align*}
The linear covariance operator $\mathcal{S}:\mathbb{C}^{2 n \times 2 n} \rightarrow \mathbb{C}^{2 n \times 2 n}$ is defined by
\begin{align}\label{eq:S[.]}
    &\mathcal{S}\left[\left(\begin{array}{ll}A & B \\ C & D\end{array}\right)\right]:=\widetilde{\mathbb{E}}^{(t)}\widetilde{W}^{(t)}\left(\begin{array}{ll}A & B \\ C & D\end{array}\right) \widetilde{W}^{(t)}
    =\left(\begin{array}{cc}\langle D \rangle & 0 \\ 0 & \langle A \rangle\end{array}\right), \\
    \text{where }&\widetilde{W}^{(t)}=\left(\begin{array}{cc}0 & \widetilde{X}^{(t)} \\ (\widetilde{X}^*)^{(t)} & 0\end{array}\right),\notag
\end{align}
and $\widetilde{X}^{(t)} \sim$ Gin$_\mathbb{C}$. Here Gin$_\mathbb{C}$ stands for the standard complex Ginibre ensemble, $\widetilde{\mathbb{E}}^{(t)}$ denotes the expectation with respect to $\widetilde{X}^{(t)}.$

For any given function $f: \mathbb{C}^{2 n \times 2 n} \rightarrow \mathbb{C}^{2 n \times 2 n}$, we define the self renormalization as follows;
\begin{equation}\label{eq:underline}
    \underline{W^{(t)} f(W^{(t)})}:=W^{(t)} f(W^{(t)})-\widetilde{\mathbb{E}}^{(t)} \widetilde{W}^{(t)} \left(\partial_{\widetilde{W}^{(t)}} f\right)(W^{(t)}),
\end{equation}
where $\partial_{\widetilde{W}^{(t)}}$ indicates a directional derivative in the direction $\widetilde{W}^{(t)}$ and $\widetilde{W}^{(t)}$ denotes an independent random matrix as in \eqref{eq:S[.]} with $\widetilde{X}^{(t)}$ a complex Ginibre matrix with expectation $\widetilde{\mathbb{E}}^{(t)}$.

\subsubsection{Approximation of $G_{1}^{(t)}BG_{2}^{(t)}$ via stability operators}
According to the definition above, we have
\begin{align*}
\underline{W^{(t)}G^{(t)}}&=W^{(t)}G^{(t)}+\mathcal{S}[G^{(t)}]G^{(t)}.
\end{align*}

Now, using the above equation, \eqref{eq:MDS}, \eqref{eq:M^z matrix}, along with the observation $\mathcal{S}[M_{i}^{(t)}]=m_{i}^{(t)}\mathbb{I},$ we get the following identity
\begin{align}\label{eq:deviation}
    G_i^{(t)}&=M_i^{(t)}-M_i^{(t)}\underline{W^{(t)}G_i^{(t)}}+M_i\mathcal{S}[G_i^{(t)}-M_i^{(t)}]G_i^{(t)}.
\end{align}

Similarly, we can calculate the deviation of $G_{1}^{(t)}BG_{2}^{(s)}$ from $M_{1}^{(t)}BM_{2}^{(s)}$ for any deterministic matrix $B$ with bounded norm as follows;
\begin{align}\label{eq:G_1^{(t)} B G_2^{(s)}}
    &G_1^{(t)} B G_2^{(s)}\\
        =&\big(M_1^{(t)}-M_1^{(t)}\underline{W^{(t)}G_1^{(t)}}+M_1\mathcal{S}[G_1^{(t)}-M_1^{(t)}]G_1^{(t)}\big)BG_2^{(s)}\notag\\
        = & M_1^{(t)} B G_2^{(s)}-M_1^{(t)} \underline{W^{(t)}G_1^{(t)}} B G_2^{(s)}+M_1^{(t)} \mathcal{S}[G_1^{(t)}-M_1^{(t)}] G_1^{(t)} B G_2^{(s)} +M_1^{(t)}B M_2^{(s)}-M_1^{(t)}B M_2^{(s)}\notag\\
        =& M_1^{(t)} B M_2^{(s)}-M_1^{(t)} \underline{W^{(t)}G_1^{(t)} B G_2^{(s)}}+M_1^{(t)} \mathcal{S}[G_1^{(t)}-M_1^{(t)}] G_1^{(t)} B G_2^{(s)}+M_1^{(t)} B (G_2^{(s)}-M_2^{(s)})\notag.
\end{align}

Here in the last equality, we used
\begin{align}
    M_1^{(t)} \underline{W^{(t)}G_1^{(t)}} B G_2^{(s)}
    =&M_1^{(t)} W^{(t)}G_1^{(t)} B G_2^{(s)}-M_1^{(t)} \widetilde{\mathbb{E}^{(t)}} \widetilde{W^{(t)}}\partial_{\widetilde{W^{(t)}}}(G_1^{(t)})BG_2^{(s)}\\
        =&M_1^{(t)} W^{(t)}G_1^{(t)} B G_2^{(s)}-M_1^{(t)} \widetilde{\mathbb{E}^{(t)}} \widetilde{W^{(t)}}\partial_{\widetilde{W^{(t)}}}(G_1^{(t)}BG_2^{(s)})\notag\\
        =&M_1^{(t)}\underline{ W^{(t)}G_1^{(t)} B G_2^{(s)}}\notag.
\end{align}
It can be shown that the terms $G_{i}^{(t)}-M_{i}^{(t)}$ and the self renormalized terms are small, which is formalized in Theorem \ref{Thm: Local law combined}. As a result, we obtain the following approximation 
\begin{equation}\label{eq:approx}
    G_1^{(t)}B G_2^{(s)}\approx M_1^{(t)}B M_2^{(s)}.
\end{equation}

Similarly, for $t=s$ we have
\begin{align*}
    G_1^{(t)}B G_2^{(t)}\approx M_1^{(t)}B M_2^{(t)}=:(M^{(t)}_B)^{z_1,z_2}\left(\omega_1, \omega_2\right),
\end{align*}
which can also be expressed as 
\begin{align}\label{eq:M_B^t}
    &(M^{(t)}_B)^{z_1,z_2}\left(\omega_1, \omega_2\right)\notag\\
    &:=\left(\mathbb{I}-(M^{(t)})^{z_1}\left(\omega_1\right) \mathcal{S}\left[\cdot\right] (M^{(t)})^{z_2}(\omega_2)\right)^{-1}\left[(M^{(t)})^{z_1}\left(\omega_1\right) B (M^{(t)})^{z_2}(\omega_2)\right].
\end{align}

In this context, we define the 2-body stability operator $\hat{\mathcal{B}}^{(t)}_{12}\left(z_1, z_2, \omega_1, \omega_2\right):GL_{2n}(\mathbb{C})\to GL_{2n}(\mathbb{C})$ as follows;
 \begin{equation}\label{def:2-body stability operator}
    \hat{\mathcal{B}}^{(t)}_{12}\left(z_1, z_2, \omega_1, \omega_2\right):=\mathbb{I}-M_1^{(t)} \mathcal{S}[\cdot] M_2^{(t)}.
 \end{equation}
 Here, $GL_{2n}(\mathbb{C})$ denotes the space of $2n\times 2n$ matrices with the usual matrix norm. In the rest of the article, for convenience, we denote $\hat{\mathcal{B}}^{(t)}_{12}\left(z_1, z_2, \omega_1, \omega_2\right)$ as $\hat{\mathcal{B}}^{(t)}_{12}$ or $\hat{\mathcal{B}}^{(t)}.$

\begin{Thm}\cite[Theorem 3.5]{10.1214/21-EJP591}\label{Thm: Local law combined}
    Fix $z_1, z_2 \in \mathbb{C}$ with $|1-|z_i||\geq \epsilon,$ for some $\epsilon >0$ and $\omega_1, \omega_2 \in \mathbb{C}$ with $|\eta_i|:=|\Im w_i| \geq n^{-1}$ such that
$$
\eta_{*}:=\min \left\{|\eta_1|,|\eta_2|\right\} \geq n^{-1+\epsilon_{*}} |(\hat{\beta}^{(t)}_*)^{-1}|
$$
for some $\epsilon_*>0$, where $|\hat{\beta}^{(t)}_{*}|$ is the smallest eigen value of $\hat{\mathcal{B}}_{12}^{(t)}.$
Then, for any bounded deterministic matrix $B$, with $\|B\| \lesssim 1$, the product of resolvents
$$
(G^{(t)})^{z_1} B (G^{(t)})^{z_2}=(G^{(t)})^{z_1}(\omega_1) B (G^{(t)})^{z_2}(\omega_2)
$$
is approximated by $(M^{(t)})_B^{z_1, z_2}=(M^{(t)})_B^{z_1, z_2}(\omega_1, \omega_2)$ defined in \eqref{eq:M_B^t} in the sense that

\begin{align}
& |\langle A((G^{(t)})^{z_1} B (G^{(t)})^{z_2}-(M^{(t)})_B^{z_1, z_2})\rangle|\notag \\
& \lesssim  \frac{C_{\epsilon} \|A\|n^{\xi}}{n \eta_*(\eta_1 \eta_2)^{1 / 2} |\hat{\beta}^{(t)}_*|}\Bigg((\eta_*)^{1/12}+\frac{\eta_*^{1 / 4}}{|\hat{\beta}^{(t)}_*|}+\frac{1}{\sqrt{n \eta_*}}+\frac{1}{\big(|\hat{\beta}^{(t)}_*| n \eta_*\big)^{1 / 4}}\Bigg),\\
& \quad\left|\left\langle \boldsymbol{x},\left((G^{(t)})^{z_1} B (G^{(t)})^{z_2}-(M^{(t)})_B^{z_1, z_2}\right) \boldsymbol{y}\right\rangle\right| \lesssim \frac{C_{\epsilon}\|\boldsymbol{x}\|\|\boldsymbol{y}\|n^{\xi}}{\left(n \eta_*\right)^{1 / 2}(\eta_1 \eta_2)^{1 / 2}|\hat{\beta}^{(t)}_*|}
\end{align}

for some $C_\epsilon$ with very high probability for any deterministic $A, \boldsymbol{x}, \boldsymbol{y}$, and $\epsilon \geq 0$. If $\omega_1,\omega_2 \in i\mathbb{R}$ we may choose $C_{\epsilon}=1,$ otherwise we can choose $C_\epsilon \leq \epsilon^{-100}.$
\end{Thm}

In continuation with the two-body stability operator, we also note down the one-body stability operator as follows;
\begin{align}\label{eqn: one body stability definition}
    \mathcal{B}^{(t)} := \mathbb{I} - M^{(t)}\mathcal{S}[\cdot]M^{(t)}.
\end{align}
Recalling the action of $\mathcal{S}$ from \eqref{eq:S[.]}, we see that $\mathcal{B}^{(t)}$ acting on a block matrix $\left(\begin{array}{cc}
   a\mathbb{I}  &  b\mathbb{I}\\
   c\mathbb{I}  & d\mathbb{I}
\end{array}\right)$ is given by
\begin{align*}
    \mathcal{B}^{(t)}\left(\begin{array}{cc}
   a\mathbb{I}  &  b\mathbb{I}\\
   c\mathbb{I}  & d\mathbb{I}
\end{array}\right) 
= \left(\begin{array}{cc}
   (1-d(m^{(t)})^{2}-a|z|^{2}(u^{(t)})^{2})\mathbb{I}  &  (d+a)zu^{(t)}m^{(t)}\mathbb{I}\\
    (d+a)\bar{z}u^{(t)}m^{(t)}\mathbb{I} & (1-d|z|^{2}u^{(t)}-a(m^{(t)})^{2})\mathbb{I}
\end{array}
\right).
\end{align*}
From the above expression, we notice that any eigenvector of $\mathcal{B}^{(t)}$ will be of the form $\left(\begin{array}{cc}
   a\mathbb{I}  &  0\\
   0  & d\mathbb{I}
\end{array}\right).$ 

We now enlist the following lemmas, which estimate the bounds on the eigenvalues of the stability operators.
\begin{lemma}\cite[Lemma 5.2]{10.1214/21-EJP591} Let $\omega\in \mathbf{H}, z\in \mathbb{C}$ be bounded spectral parameters, $|\omega|+|z|\lesssim 1.$ Then the operator $\mathcal{B}^{(t)}$ has the trivial eigenvalue $1$ and two non-trivial eigenvalues $1+(m^{(t)})^{2}-|z|^{2}(u^{(t)})^{2}$ and $1-(m^{(t)})^{2}-(u^{(t)})^{2}|z|^{2},$ which is same as $\beta^{(t)}$ defined in \eqref{eq:m^t_derivative}. Moreover, $\bar{\beta}^{(t)}$ is an eigenvalue of $(\mathcal{B}^{(t)})^{*}$ corresponding to the eigenvector $\left(\begin{array}{cc}
   \mathbb{I}  &  0\\
   0  & \mathbb{I}
\end{array}\right).$ In addition, we have the bound
\begin{align*}
    |\beta^{(t)}|\gtrsim |1-|z|| + \eta^{2/3}.
\end{align*}
    \end{lemma}
\begin{lemma}\cite[Lemma 6.1]{10.1214/21-EJP591}\label{Lemma:bound on beta hat}
    For $z_1, z_2 \in \mathbb{C},\omega_1,\omega_2 \in \mathbb{C} \backslash \mathbb{R}$ such that $|z_i|,|\omega_i| \lesssim 1,$ the two non-trivial eigenvalues $\hat{\beta}^{(t)}, \hat{\beta}^{(t)}$ of $\hat{\mathcal{B}}^{(t)}$ satisfy
$$
\min \big\{\Re \hat{\beta}^{(t)}, \Re \hat{\beta}_{*}^{(t)}\big\} \gtrsim|z_1-z_2|^2+\min \big\{|\omega_1+\overline{\omega_2}|,|\omega_1-\overline{\omega_2}|\big\}^2+|\Im \omega_1|+|\Im \omega_2|.
$$
\end{lemma}


\section{Central limit theorem for linear eigenvalue statistic}\label{sec: clt for les}

In this section, we prove our main result, Theorem \ref{Thm:CLT for LES}. Before proceeding, let us state the following result, which asserts that the main contribution comes from the regime $I_{\eta_c}^{T}.$
    \begin{lemma}[\cite{cipolloni2023central}, Lemma 4.3, 4.4]\label{Lemma:J_T^t_bound}
    For some bounded open $\bar{\mathbf{D}} \subset\Omega\subset\mathbb{C},$ let $f \in H_0^{2+\delta}(\Omega).$ Then, for any $\xi>0$ the following bound holds with high probability;
    \begin{align*}
        & |J_T^{(t)}| \lesssim \frac{n^{1+\xi}\norm{\Delta f}_{L^1(\Omega)}}{T^2},\\
        & |(I_0^{\eta_0})^{(t)}|+|(I_{\eta_0}^{\eta_c})^{(t)}|+|(I_{\eta_c}^{T})^{(t)}| \lesssim n^{\xi}\norm{\Delta f}_{L^2(\Omega)}|\Omega|^{1/2},
    \end{align*}
    where $J_{T}^{(t)}, (I_{0}^{\eta_{0}})^{(t)}, (I_{\eta_{0}}^{\eta_{c}})^{(t)}, (I_{\eta_{c}}^{T})^{(t)}$ are same as defined in \eqref{eq:J_T notations}. Moreover, there exists $\delta'>0$ such that
        \begin{align*}
        \mathbb{E}[|(I_0^{\eta_0})^{(t)}|] +\mathbb{E}[|(I_{\eta_0}^{\eta_c})^{(t)}|]\lesssim n^{-\delta^{'}}\norm{\Delta f}_{L^2(\Omega)}.
    \end{align*}
\end{lemma}
Applying the above lemma for $t=1, 2$, we have the following estimates
    \begin{align}\label{eqn: bound on J_T, etc}
        & |J_T| \lesssim \frac{n^{1+\xi}\norm{\Delta f}_{L^1(\Omega)}}{T^2},
        \quad |I_0^{\eta_0}|+|I_{\eta_0}^{\eta_c}|+|I_{\eta_c}^{T}| \lesssim n^{\xi}\norm{\Delta f}_{L^2(\Omega)}|\Omega|^{1/2},
    \end{align}
    with high probability, and
    \begin{align}\label{eqn: bound on expectation of eta_0 and eta_c}
        \mathbb{E}[|I_0^{\eta_0}|] +\mathbb{E}[|I_{\eta_0}^{\eta_c}|]\lesssim n^{-\delta^{'}}\norm{\Delta f}_{L^2(\Omega)},
    \end{align}
    where $J_{T}, I_{0}^{\eta_{0}}, I_{\eta_{0}}^{\eta_{c}}, I_{\eta_{c}}^{T}$ are defined in \eqref{eq:L_n(f)}.

The above estimates \eqref{eqn: bound on J_T, etc}, and \eqref{eqn: bound on expectation of eta_0 and eta_c} imply that the main contribution in $\mathcal{L}_{n}(f)$ comes from $I_{\eta_c}^{T},$ which leads to the following corollary.

\begin{cor}\label{cor:reduced_to_I_eta_c} Let $f_{i}\in H^{2+\delta}(\Omega),$ $1\leq i\leq p$ for some bounded open $\overline{\mathbf{D}}\subset \Omega\subset \mathbb{C}.$ Then
  \begin{align*}
      &\mathbb{E}\Big[\prod_{i \in [p]} \mathcal{L}_n\left(f_{i}\right)\Big]=\mathbb{E}\Big[\prod_{i \in [p]} I_{\eta_c}^T(f_{i})\Big]+\mathcal{O}_{\prec}(n^{-c(p)}),\notag
  \end{align*}
  where $\mathcal{L}_n(f_{i})$ and $I_{\eta_c}^T(f_{i})$ defined in \eqref{eq:J_T notations} and $c(p) >0.$
\end{cor}

The above corollary leads to the following;
\begin{align}\label{eq:Sum_Prod_L_n(f_{i}}
    &\mathbb{E}\Big[\prod_{i \in [p]} \mathcal{L}_n(f_{i})\Big]\notag\\
    =&\mathbb{E}\Big[\prod_{i \in [p]} I_{\eta_c}^T(f_{i})\Big]+\mathcal{O}_{\prec}(n^{-c(p)})\notag\\
    =&\mathbb{E} \bigg[\prod_{i \in[p]}\bigg(-\frac{n}{2 \pi \iota} \int_{\mathbb{C}} \Delta f_{i}(z) 
    \int_{\eta_c}^{T} \big(c\langle(G^{(1)})^z(\iota \eta)-\mathbb{E} (G^{(1)})^z(i \eta)\rangle \\
    &+d\langle(G^{(2)})^z(\iota \eta)-\mathbb{E} (G^{(2)})^z(\iota \eta)\rangle \big) \mathrm{d} \eta  \mathrm{d}^2 z\bigg)\bigg]
    +\mathcal{O}_{\prec}(n^{-c(p)})\notag.
\end{align}

\begin{lemma}
    Let $f_{i}$ be as in Corollary \ref{cor:reduced_to_I_eta_c}. Then
    \begin{align}\label{eq:CLT}
        &\mathbb{E}\Big[\prod_{i \in [p]} \mathcal{L}_n(f_{i})\Big]\\ 
        =&\mathbb{E}\bigg[ \prod_{i \in[p]}\Big(-\frac{n}{2 \pi \iota} \int_{\mathbb{C}} \Delta f_{i}(z) \int_{\eta_c}^T \big(c\langle(G^{(1)})^z(\iota \eta)-\mathbb{E} (G^{(1)})^z(\iota \eta)\rangle\notag\\
        &+d\langle(G^{(2)})^z(\iota \eta)-\mathbb{E} (G^{(2)})^z(\iota \eta)\rangle \big)\mathrm{d} \eta  \mathrm{d}^2 z\Big)\bigg]\notag\\
           =&\frac{1}{8\pi^2}\sum_{p\in \prod_p}\prod_{\{i,j\}\in P}\bigg[\int_{\mathbb{C}}-\mathrm{d}^2 z_i \Delta f_{i} \int_{\mathbb{C}} \mathrm{d}^2 z_j \Delta f^{(j)}\notag\\
          &\times \int_0^{\infty} \mathrm{d} \eta_i \int_0^{\infty} \mathrm{d} \eta_j \bigg( c^2\widehat{V_{ij}^{(1)}}+d^2\widehat{V_{ij}^{(2)}} +8cd\gamma (L_{ji}^{12}+L_{ji}^{21})+2cd\rho  N_{ij} \notag\\
          &+c^2(\kappa_4)_{11}U_i^{(1)}U_j^{(1)}+ d^2(\kappa_4)_{22}U_i^{(2)}U_j^{(2)}+cd(\kappa_4)_{12} \big(U_i^{(1)}U_j^{(2)}+U_j^{(1)}U_i^{(2)}\big)\bigg)\Bigg]\notag\\
          &+\mathcal{O}_{\prec}(n^{-c(p)}),\notag
   \end{align}
\end{lemma}
for some small $c(p)>0$, where $\widehat{V_{ij}^{(1)}},\widehat{V_{ij}^{(2)}},L_{ji}^{12},L_{ji}^{21}, N_{ij}, \text{ and }U_i^{(t)}$ are defined in \eqref{eq:hat_V_{ij}^{(t)}} and  \eqref{notation: limit notations}.

\begin{proof}
This is proved by applying the CLT for resolvents (Theorem \ref{Thm:CLT for resolvents}) on \eqref{eq:Sum_Prod_L_n(f_{i}}. However, it should be noted that while applying \eqref{eq:CLT_.For_.Resolvent}, we need to estimate the error term $\Psi,$ which is given in \eqref{eq:Psi_error_term}. In addition, it should be noted that the domain of the integration in the last equality of \eqref{eq:CLT} is changed from $[\eta_{c}, T]$ to $[0, \infty).$ Therefore, we need to argue that the value of the integration is negligible on the domains $[0, \eta_{c})$ and $(T, \infty).$ Note that it is only sufficient to consider even $p.$ Because the terms corresponding to odd $p,$ are of lower order by  \eqref{eq:CLT_.For_.Resolvent}. We now proceed to estimate the terms $V_{ij}^{(t)}$, $U_{i}^{(t)},$ $N_{ij},$ $L_{ij}.$

Borrowing the shorthand notations from \eqref{eqn: M_i notation}, let us denote $m_{i}^{(t)}:=(m^{(t)})^{z_{i}}(\iota \eta_{i}).$ Since $\Re (m^{(t)})^{z}(\omega)=0$ on $\omega = \iota \eta,$ we may write $m_{i}=\iota \Im m_{i}.$ As a result, from \eqref{eq:M^z matrix}, we may denote
\begin{align}\label{eqn: definition of u_i on imaginary omega}
    u_{i}=-\frac{\iota \Im m_{i}}{\eta_{i}+\Im m_{i}}.
\end{align}

We now list down the estimates as follows;

\begin{equation}\label{eq: estimates for integrals}
    \begin{aligned}
        &\big|\partial_{\eta_i}(m_i^{(t)})\big|\lesssim \frac{1}{\big(\eta_i+(\Im{m_i^{(t)}})^2\big)^2},\\
        &|\partial_{\eta_{i}}u_{i}^{(t)}|\lesssim [\Im m_{i}^{(t)}+\eta_{i}]^{-2},\\
        &|V_{ij}^{(t)}|  \lesssim \frac{\left[\left(\Im{(m^{(t)})^{z_i}(\iota\eta_i)}+\eta_i \right)\left(\Im{(m^{(t)})^{z_j}(\iota\eta_j)}+\eta_j \right)\right]^{-2}}{\left[|z_i-z_j|^2+\left(\eta_i^{(t)}+\eta_j^{(t)}\right)\left(\min \left\{\Im{m_i^{(t)}},  \Im{m_j^{(t)}}\right\}^2\right)\right]^2},\\
        &|U_i^{(t)}| \lesssim \frac{1}{\big(\eta_i+(\Im{m_i^{(t)}})^2\big)^2},\\
        &|N_{1i}| 
        \lesssim \frac{1}{\big(\eta_i+(\Im{m_i^{(1)}})^2\big)^2\big(\eta_j+(\Im{m_j^{(2)}})^2\big)^2},\\
        &|L_{ji}^{12}|
        \lesssim \frac{1}{\big(\eta_i+(\Im{m_i^{(2)}})^2\big)^2\big(\eta_j+(\Im{m_j^{(1)}})^2\big)^2}\big|\Re{(\overline{z_i}z_j)}\Big|,\\
        &|L_{ji}^{21}|
        \lesssim \frac{1}{\big(\eta_i+(\Im{m_i^{(1)}})^2\big)^2\big(\eta_j+(\Im{m_j^{(2)}})^2\big)^2}\big|\Re{(\overline{z_i}z_j)}\Big|.
    \end{aligned}
\end{equation}

\textbf{Estimate of $\partial_{\eta_{i}}m_i^{(t)}:$}
By using the estimate \eqref{eq: beta star is non-negative} and asymptotes from \eqref{eq:beta_*}, we obtain the following estimate;
\begin{align*}
    \big|\partial_{\eta_i}(m_i^{(t)})\big| &= \Big|\frac{1-\beta_i^{(t)}}{\beta_i^{(t)}}\Big|\leq \frac{1}{|\beta_i^{(t)}|}= \frac{1}{|(\beta_*^{(t)})_i+2(\Im{m_i^{(t)}})^2|}\\
    &= \frac{1}{|\frac{\eta_i}{\Im{m_i^{(t)}}+\eta_i}+2(\Im{m_i^{(t)}})^2|}\lesssim \frac{1}{\big(\eta_i+(\Im{m_i^{(t)}})^2\big)^2},
\end{align*}
where the last inequality follows from the fact that $\eta_i/(\eta_i+\Im{m_i^{(t)}})\geq \eta_i/(1+\eta_{i})\gtrsim \eta_{i}.$

\textbf{Estimate of $\partial_{\eta_{i}}u_{i}^{(t)}:$}
Differentiating both sides of \eqref{eqn: definition of u_i on imaginary omega} with respect to $\eta_{i},$ we obtain
\begin{align*}
    \partial_{\eta_{i}}(u_{i}^{(t)})=-\iota\frac{\eta_{i}\partial_{\eta_{i}}(\Im m_{i}^{(t)})-\Im m_{i}^{(t)}}{(\eta_{i}+\Im m_{i}^{(t)})^{2}}.
\end{align*}
Using \eqref{eq: beta star is non-negative}, we have $|\Im m_{i}^{(t)}| < 1.$ Moreover, using \eqref{eq:m^t_derivative} in conjunction with the estimates \eqref{eq:Im{m^t}} and \eqref{eq:beta_*} we obtain $|\eta_{i}\partial_{\eta_{i}}\Im m_{i}^{(t)}|<1,$ which gives us
\begin{align*}
    |\partial_{\eta_{i}}u_{i}^{(t)}|\lesssim [\Im m_{i}^{(t)}+\eta_{i}]^{-2}.
\end{align*}

\textbf{Estimate of $V_{ij}^{(t)}:$} Using the definition of $u_i^{(t)}$ from \eqref{eq:M^z matrix} and the equation \eqref{eq:MDS}, we notice that $-(m_i^{(t)})^2+|z_i|^2 (u_i^{(t)})^2=u_i^{(t)}.$ Consequently, we may rewrite $V_{ij}^{(t)},$ whose expression is given in \eqref{eq:hat_V_{ij}^{(t)}}, as follows;
\begin{align*}
    &V_{ij}^{(t)}\\
    &=\frac{1}{2} \partial_{\eta_i} \partial_{\eta_j} \log \Big(1-u_i^{(t)} u_j^{(t)}\big(1-|z_i-z_{j}|^2+(1-u_i^{(t)})|z_i|^2+(1-u_j^{(t)} )|z_j|^2\big)\Big)\\
    &=\frac{1}{2}\frac{(\partial_{\eta_{i}}u_{i}^{(t)})(\partial_{\eta_{j}} u_{j}^{(t)})C(u_{i}^{(t)}, u_{j}^{(t)}, z_{i}, z_{j})}{\Big(1-u_i^{(t)} u_j^{(t)}\big(1-|z_i-z_{j}|^2+(1-u_i^{(t)})|z_i|^2+(1-u_j^{(t)} )|z_j|^2\big)\Big)^{2}},
\end{align*}
where $C(u_{i}^{(t)}, u_{j}^{(t)}, z_{i}, z_{j})$ is a function of $u_{i}^{(t)}, u_{j}^{(t)}, z_{i}, z_{j},$ which is uniformly bounded by a constant for all $u_{i}^{(t)}, u_{j}^{(t)}, z_{i}, z_{j}.$ 

We now just need to calculate a lower bound of the denominator. Using $-(m_i^{(t)})^2+|z_i|^2 (u_i^{(t)})^2=u_i^{(t)}$ once again, we shall obtain \eqref{eqn: alternative expression of V_ij in section 5}, where the main factor to be analyzed is $1-u_{i}^{(t)}.$ Now, on the imaginary axis i.e., $\omega_{i} = \iota \eta_{i},$ we have $m^{(t)}(\iota\eta_i) = \iota\Im{m^{(t)}}(\iota\eta_i).$ It should be noted that we are working on the regime where $\eta_{i}\in [\eta_{c}, T]$ and $\eta_{c}= n^{-1+\delta_{1}} > 0.$ Since $\eta_{i}\Im m_{i}^{(t)} > 0,$ we have $\Im m_{i}^{(t)} > 0.$ Moreover, we also have $\Im m_{i}^{(t)} < 1,$ which was noted down in the estimate of $\partial_{\eta_{i}} u_{i}^{(t)}$ as well.  Consequently, $1-u_i^{(t)}=\eta_i/(\eta_i+\Im{m_i^{(t)}})\geq \eta_i/(1+\eta_{i}),$ which implies that $\eta_{i}\lesssim 1- u_{i}^{(t)} < 1.$ Combining all, we obtain
\begin{align}\label{eqn: alternative expression of V_ij in section 5}
    &\Big|1-u_i^{(t)} u_j^{(t)}\big(1-|z_i-z_{j}|^2+(1-u_i^{(t)})|z_i|^2+(1-u_j^{(t)} )|z_j|^2\big)\big|\\
    &\quad =\big|u_i^{(t)}u_j^{(t)}|z_i-z_j|^2+(1-u_i^{(t)})(1-u_j^{(t)})\notag\\
    &\quad\quad-(m_i^{(t)})^2u_j^{(t)}\Big(\frac{1}{u_i^{(t)}}-1\Big)-(m_j^{(t)})^2u_i^{(t)}\Big(\frac{1}{u_j^{(t)}}-1\Big)\Big|\notag\\
    &\quad \geq u_{i}^{(t)}u_{j}^{(t)}|z_i-z_j|^2+ \min\{(\Im{m_i^{(t)}})^2,(\Im{m_j^{(t)}})^2\}(2-u_i^{(t)}-u_j^{(t)})\notag\\
    &\quad \gtrsim u_{i}^{(t)}u_{j}^{(t)}|z_i-z_j|^2+ \min\{(\Im{m_i^{(t)}})^2,(\Im{m_j^{(t)}})^2\}(\eta_i+\eta_j).\notag
\end{align}
If both $u_{i}^{(t)}, u_{j}^{(t)}$ are bounded away from zero i.e., there exists a $\delta$ such that $u_{i}^{(t)}, u_{j}^{(t)} > \delta,$ then $u_{i}^{(t)}u_{j}^{(t)}|z_i-z_j|^2\gtrsim |z_i-z_j|^2.$ Otherwise, $u_{i}^{(t)}u_{j}^{(t)}|z_i-z_j|^2\geq 0,$ but $2-u_i^{(t)}-u_j^{(t)}\geq 1-\delta.$ So, in the worst case, we have the following bound
\begin{align*}
&\big|1-u_i^{(t)} u_j^{(t)}\big(1-|z_i-z_{j}|^2+(1-u_i^{(t)})|z_i|^2+(1-u_j^{(t)} )|z_j|^2\big)\big|\\
    &\quad\gtrsim|z_i-z_j|^2+ \min\{(\Im{m_i^{(t)}})^2,(\Im{m_j^{(t)}})^2\}(\eta_i+\eta_j).
\end{align*}

\textbf{Estimate of $U_{i}^{(t)}:$} From \eqref{eq: estimates for integrals} using the estimate of $\big|\partial_{\eta_i}(m_i^{(t)})\big|$, we obtain the following estimate for $U_i^{(t)};$
\begin{align*}
    |U_i^{(t)}|
    &= \Big|\frac{i}{\sqrt{2}} \partial_{\eta_i} (m_i^{(t)})^2\Big|
    \lesssim \big|m_i^{(t)}(m_i^{(t)} )^{'} \big|
    \lesssim \frac{1}{\big(\eta_i+(\Im{m_i^{(t)}})^2\big)^2}.
\end{align*}

\textbf{Estimate of $N_{ij}:$} 
\begin{align*}
    |N_{ij}|&=|\partial_{\eta_i}(m_i^{(1)}) \partial_{\eta_j}(m_j^{(2)})+ \partial_{\eta_i}(m_i^{(2)}) \partial_{\eta_j}(m_j^{(1)})|\\
    &\lesssim \frac{1}{\big(\eta_i+(\Im{m_i^{(1)}})^2\big)^2\big(\eta_j+(\Im{m_j^{(2)}})^2\big)^2}.
\end{align*}

\textbf{Estimates of $L_{ji}^{12}$ and $L_{ji}^{21}:$}
\begin{align*}
    |L_{ji}^{12}| &=\Big|\frac{ m_j^{(1)}m_i^{(2)} \big(1-2(m_j^{(1)})^2-\beta_j^{(1)}\big)\big(1-2(m_i^{(2)})^2-\beta_i^{(2)}\big)\Re{(\overline{z_i}z_j)}}{\beta_j^{(1)}\beta_i^{(2)}}\Big|\\
    &\lesssim \Big|\frac{ m_j^{(1)}m_i^{(2)} \big(1-\beta_j^{(1)}\big)\big(1-\beta_i^{(2)}\big)\Re{(\overline{z_i}z_j)}}{\beta_j^{(1)}\beta_i^{(2)}}\Big|\\
    &\lesssim \Big| m_j^{(1)}m_i^{(2)} (m_j^{(1)} )^{'}(m_i^{(2)} )^{'}\big)\Re{(\overline{z_i}z_j)}\Big|\\
     &\lesssim \frac{1}{\big(\eta_i+(\Im{m_i^{(2)}})^2\big)^2\big(\eta_j+(\Im{m_j^{(1)}})^2\big)^2}\big|\Re{(\overline{z_i}z_j)}\Big|.
\end{align*}
Similarly, 
\begin{align*}
    |L_{ji}^{21}|
     &\lesssim \frac{1}{\big(\eta_i+(\Im{m_i^{(1)}})^2\big)^2\big(\eta_j+(\Im{m_j^{(2)}})^2\big)^2}\big|\Re{(\overline{z_i}z_j)}\Big|.
\end{align*}

We are now ready to prove the lemma. We first remove the region of singularity 
\begin{align*}
    Z_{i}:=\{|1-| z_i|^2| \leq n^{-2 \nu}\}\bigcup_{j<i}\{z_j:|z_i-z_j| \leq n^{-2 \nu}\}
\end{align*}
from $\mathbb{C},$ for some small $\nu > 0.$ We may now rewrite the integral of \eqref{eq:CLT} as follows;
\begin{align}\label{eq:Removed_.bad_.region}
    & \frac{(-n)^p}{(2 \pi \iota)^p} \prod_{i \in[p]} \int_{Z_i^c} \mathrm{~d}^2 z_i \Delta f_{i}\left(z_i\right) \mathbb{E} \prod_{i \in[p]} \int_{\eta_c}^T\big(c\langle(G^{(1)})^z(\iota \eta_i)-\mathbb{E} (G^{(1)})^z(\iota \eta_i)\rangle  \\
& \quad+d\langle(G^{(2)})^z(\iota \eta_i)-\mathbb{E} (G^{(2)})^z(\iota \eta_i)\rangle \big) \mathrm{d} \eta_i+\mathcal{O}_{\prec}\left(\frac{n^{p \xi}}{n^\nu}\right).\notag
\end{align}
Since the Lebesgue measure of $Z_{i}$ is $\mathcal{O}(n^{-2\nu}),$ the removal of the region $Z_{i}$ introduced an error of $\mathcal{O}_{\prec}\left(n^{p \xi}/n^\nu\right),$ for any very small $\xi>0,$ which is a consequence of \eqref{eqn: bound on J_T, etc}.

Now, by using \eqref{eq: estimates for integrals} we have the following estimates over $Z_{i}^{c}$ in the regime $ \eta_i \in\left[0, \eta_c\right]$, where $\eta_c=n^{-1+\delta_1};$
\begin{align}\label{eq: estimates over Z_icap regimes}
    &|V_{ij}^{(t)}|\lesssim \frac{n^{12\nu}}{(1+\eta_i^2)(1+\eta_j^2)},\notag\\
    & |U_i^{(t)}|\lesssim \frac{n^{4\nu}}{1+\eta_i^2},\notag\\
    &|N_{ij}|\lesssim \frac{n^{8\nu}}{(1+\eta_i^2)(1+\eta_j^2)},\\
    &|L_{ji}^{12}| \lesssim \frac{n^{8\nu}}{(1+\eta_i^2)(1+\eta_j^2)},\notag\\
    &|L_{ji}^{21}| \lesssim \frac{n^{8\nu}}{(1+\eta_i^2)(1+\eta_j^2)}.\notag
\end{align}

Finally, using \eqref{eq:CLT_.For_.Resolvent} in \eqref{eq:Removed_.bad_.region}, we obtain the integrand displayed in the last equality of \eqref{eq:CLT}. However, the domain of the integration in \eqref{eq:Removed_.bad_.region} is not same as \eqref{eq:CLT}. We can add back the domains $(0, \eta_{c})$ and $(T, \infty)$ using the estimates \eqref{eq: estimates over Z_icap regimes} and the first estimate in \eqref{eqn: bound on J_T, etc}. Lastly, the error $\Psi$ from \eqref{eq:Psi_error_term} is bounded by $\mathcal{O}(n^{\xi p+2p\nu-\delta_{1}/2})$ on $\cap _{i}Z_{i}^{c}.$ This concludes the lemma and yields the following equation

\begin{equation}\label{eq:(4.24)}
    \begin{aligned}
        &\frac{1}{(2 \pi \iota)^p} \prod_{i \in[p]} \int_{Z_i^c \cap \widehat{Z_i^c}} \mathrm{~d}^2 z_i \Delta f_{i}\\
& \times \sum_{P \in \Pi_p\{i, j) \in P} \int_0^{\infty}  \int_0^{\infty}-\Bigg(\frac{c^2\widehat{V_{ij}^{(1)}}+d^2\widehat{V_{ij}^{(2)}} +8cd\gamma (L_{ji}^{12}+L_{ji}^{21})+2cd\rho  N_{ij} }{8\pi^2}\\
&+ \frac{c^2(\kappa_4)_{11}U_i^{(1)}U_j^{(1)}+d^2(\kappa_4)_{22} U_i^{(2)}U_j^{(2)}+cd(\kappa_4)_{12} \big(U_i^{(1)}U_j^{(2)}+U_j^{(1)}U_i^{(2)}\big)}{8\pi^2}  \Bigg)\mathrm{~d} \eta_i \mathrm{~d} \eta_j\\
&+\mathcal{O}_{\prec}\left(\frac{n^{p \xi}}{n^\nu}+\frac{n^{12 \nu p+\delta_1}}{n}+\frac{n^{\xi p+2 p \nu}}{n^{\delta_1 / 2}}\right).
    \end{aligned}
\end{equation}
Based on $\delta_{1},$ we can choose $\nu$ in such a way that the above error vanishes asymptotically.
\end{proof}


\section{Central limit theorem for resolvents}\label{sec: clt for resolvents}

In this section, we shall present the the proof of the CLT for the resolvents, which is stated in the Theorem \ref{Thm:CLT for resolvents}. Recall that we are considering the resolvents $G_{i}^{(t)}=(G^{(t)})^{z_{i}}(\iota\eta_{i}),$ $1\leq i\leq p$ at $p$ different points on the imaginary axis. In addition, recall that $\Upsilon_i:=c\langle G^{(1)}_i-\mathbb{E}(G^{(1)}_i)\rangle +d\langle G^{(2)}_i-\mathbb{E}(G^{(2)}_i)\rangle.$ Before proceeding to the computation of $\mathbb{E}\left[\prod_{i=1}^{p}\Upsilon_{i}\right],$ we give an estimate of $\mathbb{E}\big[\big(c(G^{(1)})^z+d(G^{(2)})^z\big)(\iota\eta)\big]$. From the local law, we get the first order approximation $\mathbb{E}\big[\langle cG^{(1)}+dG^{(2)}\rangle\big]\sim  \langle cM^{(1)}+dM^{(2)}\rangle.$ The second order correction of $\mathbb{E}\big[\langle cG^{(1)}+dG^{(2)}\rangle\big]$ is of the order of $1/n.$ This is formalized in the following lemma.

\begin{lemma}\label{Lemma:Exp or resolvent}
   For $(\kappa_{4})_{tt} = \kappa\big(\chi^{(t)},\bar{\chi}^{(t)},\chi^{(t)},\bar{\chi}^{(t)}\big) \neq 0,$ where $\chi^{(t)}$ is same as defined in Condition \ref{Condtion::}, we have
    
    \begin{align}
        \mathbb{E}\langle cG^{(1)}+dG^{(2)}\rangle = \langle cM^{(1)} +dM^{(2)} \rangle +c\xi^{(1)}+d\xi^{(2)} + Error_{exp},
    \end{align}
where 
\begin{align*}\notag
    &\xi^{(t)}\\
    =&  \left\{\begin{array}{l} -\frac{i (\kappa_4)_{tt}}{4n} \partial_{\eta}\big((m^{(t)})^4\big),\;  \text { when $X^{(t)}$ in Condition \ref{Condtion::} is $\mathbb{C}$-valued,}  \\ 
    \frac{i}{4n}\partial_{\eta}\Big(1-(u^{(t)})^2+2(u^{(t)})^3|z|^2-(u^{(t)})^2(z^2+\bar{z}^2)\Big) \\
   -\frac{i (\kappa_4)_{tt}}{4n} \partial_{\eta}\big((m^{(t)})^4\big),\;  \text { when $X^{(t)}$ in Condition \ref{Condtion::} is $\mathbb{R}$-valued,}\end{array} \right.
\end{align*}
and
\begin{align*}\notag
    &Error_{exp}\\
    =&  \left\{\begin{array}{l} \mathcal{O}_\prec\Big(\frac{1}{|1-|z||}\big(\frac{1}{n^{3/2}\eta}+\frac{1}{(n\eta)^2}\big)\Big),\;  \text { when $X^{(t)}$ is $\mathbb{C}$-valued,}  \\ 
    \mathcal{O}_\prec\Big(\big(\frac{1}{|1-|z||}+\frac{1}{\Im{z}^2}\big)\big(\frac{1}{n^{3/2}\eta}+\frac{1}{(n\eta)^2}\big)\Big), \text {when $X^{(t)}$ is $\mathbb{R}$-valued.}\end{array} \right.
\end{align*}
\end{lemma}

Before proving the lemma, we note down the cumulant expansion as follows \cite[Proposition 7.5]{speicher2023high}.
\begin{result}[Cumulant Expansion]\label{Result:Cumulant Expansion}
\begin{enumerate}
    \item For a general scalar random variable $ X $, the cumulant expansion is given as follows;
$$
\mathbb{E}[Xf(X)] = \sum_{l}\frac{\kappa_{l+1}(X)}{l!}\mathbb{E}[f^{(l)}(X)],
$$
where $\kappa_{l+1}(X)$ represents the $(l+1)$-th cumulant of $ X $, and $ f^{(l)}$ denotes the $ l $-th derivative of a smooth function $ f $.
\item In the multidimensional case, the cumulant expansion is defined as follows. Consider a collection of random variables $ X_1, X_2, \dots, X_p $, and let $\kappa$ denote their joint cumulants. For a smooth function $ f: \mathbb{R}^p \rightarrow \mathbb{R} $, the cumulant expansions are given by
\begin{align*}
    &\mathbb{E}[X_i f(X_1, X_2, \dots, X_p)]\\
    &\;\;\;= \sum_{l \geq 0} \sum_{i_1, \dots, i_l=1}^p \frac{\kappa(X_i, X_{i_1}, \dots, X_{i_l})}{l!} \mathbb{E}[\partial_{i_1}\partial_{i_2} \cdots \partial_{i_l} f(X_1, X_2, \dots, X_p)],
\end{align*}
where $\partial_{i_1}\partial_{i_2} \cdots \partial_{i_l}$ represents the partial derivatives with respect to $ X_{i_1}, X_{i_2}, \dots, X_{i_l} $.
\end{enumerate}
\end{result}


\begin{proof}[Proof of Lemma \ref{Lemma:Exp or resolvent}]
    Using the stability operator $\mathcal{B}^{(t)}$ from \eqref{eqn: one body stability definition}, let us define $A^{(t)}:=\left(((\mathcal{B}^{(t)})^*)^{-1}[\mathbb{I}]\right)^* M^{(t)}.$ Now using \eqref{eq:deviation}, we have

\begin{align}\label{eq:G^t-M^t}
    &\langle cG^{(1)}+dG^{(2)}-cM^{(1)}-dM^{(2)}\rangle\\
    &=c\langle G^{(1)}-M^{(1)}\rangle + d\langle G^{(2)}-M^{(2)}\rangle\notag\\
    &=c\langle \mathbb{I},(\mathcal{B}^{(1)})^{-1}\mathcal{B}^{(1)}[G^{(1)}- M^{(1)}]\rangle+d\langle \mathbb{I},(\mathcal{B}^{(2)})^{-1}\mathcal{B}^{(2)}[G^{(2)}- M^{(2)}]\rangle\notag\\
    &=-c\langle (M^{(1)})^*\big((\mathcal{B}^{(1)})^*\big)^{-1}[\mathbb{I}], \underline{W^{(1)}G^{(1)}}\rangle\notag\\
    &\;\;\;+c\langle (M^{(1)})^*\big((\mathcal{B}^{(1)})^*\big)^{-1}[\mathbb{I}], \mathcal{S}[G^{(1)}-M^{(1)}](G^{(1)}-M^{(1)})\rangle\notag\\
    &\;\;\;-d\langle (M^{(2)})^*\big((\mathcal{B}^{(2)})^*\big)^{-1}[\mathbb{I}], \underline{W^{(2)}G^{(2)}}\rangle\notag\\
    &\;\;\;+d\langle (M^{(2)})^*\big((\mathcal{B}^{(2)})^*\big)^{-1}[\mathbb{I}], \mathcal{S}[G^{(2)}-M^{(2)}](G^{(2)}-M^{(2)})\rangle\notag\\
    &=-c\langle (M^{(1)})^*\big((\mathcal{B}^{(1)})^*\big)^{-1}[\mathbb{I}], \underline{W^{(1)}G^{(1)}}\rangle + \mathcal{O}_\prec\bigg(\frac{\norm{((\mathcal{B}^{(1)})^*)^{-1}[\mathbb{I}]}}{(n\eta)^2}\bigg)\notag\\
    &\;\;\;-d\langle (M^{(2)})^*\big((\mathcal{B}^{(2)})^*\big)^{-1}[\mathbb{I}], \underline{W^{(2)}G^{(2)}}\rangle + \mathcal{O}_\prec\bigg(\frac{\norm{((\mathcal{B}^{(2)})^*)^{-1}[\mathbb{I}]}}{(n\eta)^2}\bigg)\notag\\
    &=-\langle c\underline{W^{(1)}G^{(1)}}A^{(1)}+d\underline{W^{(2)}G^{(2)}}A^{(2)} \rangle + \mathcal{O}_\prec\bigg(\frac{\norm{((\mathcal{B}^{(1)})^*)^{-1}[\mathbb{I}]}}{(n\eta)^2}\bigg)\\
    &\quad+\mathcal{O}_\prec\bigg(\frac{\norm{((\mathcal{B}^{(2)})^*)^{-1}[\mathbb{I}]}}{(n\eta)^2}\bigg).\notag
\end{align}

Now, we use the cumulant expansion from the Result \ref{Result:Cumulant Expansion} to evaluate $\mathbb{E}\left[\langle \underline{W^{(t)}G^{(t)}}A^{(t)} \rangle\right].$ This will lead us to calculate the derivatives of $G^{(t)}A^{(t)}$ with respect to the entries of $W^{(t)}.$ This calculation is demonstrated below, where for brevity we use the notation $\Delta^{ba}:=\boldsymbol{e}_{b}^{t}\boldsymbol{e}_{a},$ which is essentially a $2n\times 2n$ matrix with all zero entries except the $(ba)$th entry being equal to $1.$ Here, we introduce a shorthand notation for joint cumulants of the entries of the matrices $W^{(1)}$ and $W^{(2)}$. We shall denote the joint cumulant $\kappa(w_{ab}^{(t)}, w_{ij}^{(s)}, \ldots)$ as $\kappa((ab)_{t}, (ij)_{s}, \ldots),$ where $t, s\in \{1, 2\}.$ 

For $\alpha = (\alpha_1, \ldots, \alpha_k)$, we note that $\kappa((ba)_t, \alpha)$ is non-zero only for $\alpha\in \{{(ab)_t, (ba)_t}\}^k$. The derivative $\partial_\alpha$ denotes the derivative with respect to $w_{\alpha_1}^{(t)}, w_{\alpha_2}^{(t)}, \ldots, w_{\alpha_k}^{(t)}$.
\begin{align}\label{eqn: WGA in terms of S_1, S_2}
    &\mathbb{E}\left[\langle \underline{W^{(t)}G^{(t)}}A^{(t)} \rangle\right]\\
    &=\mathbb{E}\big[\langle W^{(t)}G^{(t)}A^{(t)}\rangle\big]+ \mathbb{E}\big[\langle \mathcal{S}[G^{(t)}]G^{(t)}A^{(t)}\rangle\big]\notag\\
    &=\mathbb{E}\Big[ \sum_{ab}w^{(t)}_{ba}\langle\Delta^{ba}G^{(t)}A^{(t)}\rangle\Big]+ \mathbb{E}\big[\langle \mathcal{S}[G^{(t)}]G^{(t)}A^{(t)}\rangle\big]\notag\\
    &=\sum_{k\geq 0}\sum_{ab} \sum_{\alpha\in \{(ab)_t,(ba)_t\}^k} \frac{\kappa((ba)_t, \alpha)}{k!} \mathbb{E}\left[\partial_{\alpha}\langle \Delta^{ba}G^{(t)}A^{(t)}\rangle\right]+ \mathbb{E}\big[\langle \mathcal{S}[G^{(t)}]G^{(t)}A^{(t)} \rangle\big]\notag\\
    &=\sum_{ab}\frac{\kappa((ba)_t, (ab)_t)}{1!} \mathbb{E}\left[\partial_{(ab)_t}\langle \Delta^{ba}G^{(t)}A^{(t)}\rangle\right]\notag\\
    &\quad+\sum_{ab} \frac{\kappa((ba)_t, (ba)_t)}{1!}\mathbb{E}\left[\partial_{(ba)_t}\langle \Delta^{ba}G^{(t)} A^{(t)}\rangle\right]+ \mathbb{E}\big[\langle \mathcal{S}[G^{(t)}]G^{(t)}A^{(t)} \rangle\big]\notag\\
    &\;\;\;\;\;+\sum_{k\geq 2} \sum_{ab} \sum_{\alpha\in \{(ab)_t,(ba)_t\}^k} \frac{\kappa((ba)_t, \alpha)}{k!} \mathbb{E}\left[\partial_{\alpha}\langle \Delta^{ba}G^{(t)}A^{(t)}\rangle\right]\notag\\
    &=-\frac{1}{2n^{2}}{\sum_{ab}}' \mathbb{E}\left[G^{(t)}_{aa}(G^{(t)}A^{(t)})_{bb}\right]
   +\sum_{ab} \frac{\kappa((ba)_t, (ba)_t)}{1!} \mathbb{E}\left[\partial_{(ba)_t}\langle \Delta^{ba}G^{(t)} A^{(t)}\rangle\right]\notag\\
    &\;\;\;\;\;+\sum_{k\geq 2} \sum_{ab} \sum_{\alpha\in \{(ab)_t,(ba)_t\}^k} \frac{\kappa((ba)_t, \alpha)}{k!} \mathbb{E}\left[\partial_{\alpha}\langle \Delta^{ba}G^{(t)}A^{(t)}\rangle\right]\notag\\
    &\;\;\;\;\;+ \frac{1}{2n^{2}}{\sum_{ab}}' \mathbb{E}\left[G^{(t)}_{aa}(G^{(t)}A^{(t)})_{bb}\right]\notag\\
    &=\sum_{ab} \frac{\kappa((ba)_t, (ba)_t)}{1!} \mathbb{E}\left[\partial_{(ba)_t}\langle \Delta^{ba}G^{(t)} A^{(t)}\rangle\right]\notag\\
    &\;\;\;\;\;+\sum_{k\geq 2} \sum_{ab} \sum_{\alpha\in \{(ab)_t,(ba)_t\}^k} \frac{\kappa((ba)_t, \alpha)}{k!} \mathbb{E}\left[\partial_{\alpha}\langle \Delta^{ba}G^{(t)}A^{(t)}\rangle\right]\notag\\
    &=-\frac{1}{n}{\sum_{ab}}' \mathbb{E}\left[\langle \Delta^{ba}G^{(t)}\Delta^{ba}G^{(t)}A^{(t)}\rangle\right]\notag\\
    &\;\;\;\;\;+\sum_{k\geq 2} \sum_{ab} \sum_{\alpha\in \{(ab)_t,(ba)_t\}^k} \frac{\kappa((ba)_t, \alpha)}{k!} \mathbb{E}\left[\partial_{\alpha}\langle \Delta^{ba}G^{(t)}A^{(t)}\rangle\right]\notag\\
    &=:-\frac{1}{n}F_1^{(t)}+F_{k\geq2}^{(t)},\notag
\end{align}
where, the summation notation ${\sum'_{ab}}$ is defined as 
\begin{align*}
    {\sum_{ab}}' = \sum_{a\leq n} \sum_{b> n} +\sum_{a >n} \sum_{b \leq n},
\end{align*}
which is not leaving out any terms of the double sum $\sum_{ab},$ though it seems so. This is a simple consequence of the structure of $W^{(t)}=\left[\begin{array}{cc}
    0 & X^{(t)} \\
    X^{(t)*} & 0
\end{array}\right]_{2n\times 2n}.$

Moreover, according the scaling and the definition of the cumulants, we notice that 
\begin{align*}
    \kappa((ba)_t, (ba)_t)=\bigg\{
\begin{array}{cc}
    0 & \text{if $X^{(t)}$ is $\mathbb{C}$-valued} \\
   \frac{1}{n}  & \text{if $X^{(t)}$ is $\mathbb{R}$-valued,}
\end{array}
\end{align*}
and
\begin{align}\label{eqn: cumulant estimate}
    |\kappa(\alpha_1,\alpha_2,\dots \alpha_k)| \lesssim \frac{1}{n^{k/2}}.
\end{align}

Before evaluating the terms $F_1^{(t)}, F_{k\geq2}^{(t)}$, let us define $2n \times 2n$ ordered block matrices $E_1$ and $E_2$ as follows;

\begin{align}\label{eqn: dfn of E_1 and E_2}
    E_1 =\left(\begin{array}{cc}
  \mathbb{I} &  0\\
  0 & 0 
\end{array} \right),
\text{ and }
E_2 =\left(\begin{array}{cc}
  0 &  0\\
  0 & \mathbb{I}
\end{array} \right),
\end{align}
where $\mathbb{I}$ is $n \times n$ identity matrix. We shall follow the convention $$\langle AE_1BE_2+AE_2BE_1 \rangle = :\langle AEBE^{\dag} \rangle,$$ for any $2n \times 2n$ matrices $A$ and $B$. 
Now, let us move to evaluate the terms $F_1^{(t)}, F_{k\geq2}^{(t)}.$

We shall use the convention $((G^{(t)})^z)^t = (G^{(t)})^{\bar{z}}$, along with Theorem \ref{Thm: Local law combined} (the local law for the product of resolvents), and the bound on $|(\hat{\beta}^{(t)})_*|$ from Lemma \ref{Lemma:bound on beta hat} to obtain the following estimate;
\begin{align}\label{eq:estimate_for_S_1^t}
     &\Big|\langle (G^{(t)})^{z}A^{(t)}E(G^{(t)})^{\bar{z}} E^{\dag}- (M^{(t)})_{A^{(t)}E}^{z,\bar{z}} E^{\dag} \rangle\Big|\notag\\
     &\;\;\;\;\;\lesssim \frac{C_{\epsilon} \|A^{(t)}E\|\|E^{\dag}\|n^{\xi}}{n \eta^2  |z-\bar{z}|^2} \bigg(\eta^{1/12} +\frac{\eta^{1 / 4}}{|z-\bar{z}|^2}+\frac{1}{\sqrt{n \eta}}+\frac{1}{\left(|z-\bar{z}|^2| n \eta\right)^{1 / 4}}\bigg)\notag\\
     &=\mathcal{O}_\prec\bigg(\frac{1}{|z-\bar{z}|^2|\beta^{(t)}|(n\eta)^2}\bigg).
\end{align}

\textbf{Estimate of $F_{1}^{(t)}$:} Using \eqref{eq:estimate_for_S_1^t} in the fourth equality and \eqref{eq:M_B^t} in the the fifth equality in the following calculation, we conclude 
\begin{align}\label{eq:F^{(t)}}
    &F_1^{(t)}\\
    &={\sum_{ab}}' \mathbb{E}\left[\langle \Delta^{ba}G^{(t)}\Delta^{ba}G^{(t)}A^{(t)}\rangle\right]\notag\\
        &= \mathbb{E}\left[\langle G^{(t)}A^{(t)}E(G^{(t)})^tE^{\dag}\rangle  \right]\notag\\
        &= \mathbb{E}\left[\langle (G^{(t)})^{z}A^{(t)}E(G^{(t)})^{\bar{z}}E^{\dag}\rangle  \right]\notag\\
        &= \langle (M^{(t)})_{A^{(t)}E}^{z,\bar{z}}E^{\dag}\rangle + \mathcal{O}_\prec\bigg(\frac{1}{|z-\bar{z}|^2|\beta^{(t)}|(n\eta)^2}\bigg)\notag\\
        &= \langle \big(\mathbb{I}-(M^{(t)})^{z}\mathcal{S}[\cdot](M^{(t)})^{(\bar{z})}\big)^{-1}[(M^{(t)})^{z}A^{(t)} E(M^{(t)})^{\bar{z}}] E^\dag\rangle\notag\\
        &\;\;\;\;\;+\mathcal{O}_\prec\bigg(\frac{1}{|z-\bar{z}|^2|\beta^{(t)}|(n\eta)^2}\bigg)\notag\\
        &=m^{(t)}\frac{(m^{(t)})^4+(m^{(t)})^2(u^{(t)})^2|z|^2-2(u^{(t)})^4|z|^4+2(u^{(t)})^2(x^2-y^2)}{(1-(m^{(t)})^2-(u^{(t)})^2|z|^2)(1+(u^{(t)})^4|z|^4-(m^{(t)})^4-2(u^{(t)})^2(x^2-y^2))}\notag\\
        &\;\;\;\;\;+ \mathcal{O}_\prec\bigg(\frac{1}{|z-\bar{z}|^2|\beta^{(t)}|(n\eta)^2}\bigg)\notag\\
        &=-\frac{i(u^{(t)})^{\prime}}{2}\frac{u^{(t)}-3|z|^2(u^{(t)})^2+2u^{(t)}(x^2-y^2)}{1-(u^{(t)})^2+2(u^{(t)})^3|z|^2-2(u^{(t)})^2(x^2-y^2)}\notag\\
        &\;\;\;\;\;+ \mathcal{O}_\prec\bigg(\frac{1}{|z-\bar{z}|^2|\beta^{(t)}|(n\eta)^2}\bigg)\notag\\
        &=\frac{i}{4}\partial_{\eta}\log\Big(1-(u^{(t)})^2+(u^{(t)})^3|z|^2-(u^{(t)})^2(z^2+\bar{z}^2)\Big)\notag\\
        &\;\;\;\;\;+ \mathcal{O}_\prec\bigg(\frac{1}{|z-\bar{z}|^2|\beta^{(t)}|(n\eta)^2}\bigg),\notag
    \end{align}
 where $z= x+iy.$ The sixth equality follows from the direct calculation of the inverse operator from the previous step. The second last step follows from the facts $(m^{(t)})^z = (m^{(t)})^{\bar{z}},$ $(u^{(t)})^z = (u^{(t)})^{\bar{z}},$ and the following identities,
 \begin{align*}
      &|z|^2(u^{(t)})^2-(m^{(t)})^2 = u^{(t)},\\
 &(u^{(t)})^\prime=\partial_{\eta}u^{(t)}= \frac{2iu^{(t)}m^{(t)}}{1+u^{(t)}-2|z|^2(u^{(t)})^2}. 
 \end{align*}

\textbf{Estimate of $F_{k\geq 2}^{(t)}:$} Let us first consider the case when $k=2,$ and let $\alpha =(\alpha_1,\alpha_2).$ Then
\begin{align*}
    &\sum_{ab} \sum_{\alpha} \frac{\kappa((ba)_t, \alpha_1, \alpha_2)}{2!} 
    \mathbb{E}\left[\partial_{\alpha_1, \alpha_2} 
    \langle \Delta^{ba} G^{(t)} A^{(t)} \rangle\right] \\
    &\quad = \sum_{\alpha} \frac{\kappa((ba)_t, \alpha_1, \alpha_2)}{2!} 
    {\sum_{ab}}'  \mathbb{E}
    \big[\langle \Delta^{ba} G^{(t)} \Delta^{\alpha_1} G^{(t)} \Delta^{\alpha_2} G^{(t)} A^{(t)} \rangle\big].
\end{align*}
To obtain a significant contribution, we need more diagonal terms than off-diagonal ones (as explained in Remark \ref{rem:leading_contri_discussion}). Here, by parity, at least one $G^{(t)}$ factor in $\Delta^{ba} G^{(t)} \Delta^{\alpha_1} G^{(t)} \Delta^{\alpha_2} G^{(t)} A^{(t)}$ is off-diagonal. For example, taking $(\alpha_{1}, \alpha_{2})= ((ba)_{t}, (ab)_{t})$ and using $G^{(t)}=M^{(t)}+G^{(t)}-M^{(t)}$ along with the local law bound of $G^{(t)}-M^{(t)}$ from Theorem \ref{Thm: Local_Law}, we obtain
\begin{align}\label{eq:exp_local_law}
    &\frac{1}{n^{1+3/2}}{\sum_{ab}}' \mathbb{E}\big[(G^{(t)})_{ab}(G^{(t)})_{aa}(G^{(t)}A^{(t)})_{bb}\big]\\
    \quad=& \frac{1}{n^{5/2}}{\sum_{ab}}' \mathbb{E} \bigg[(G^{(t)})_{ab} \Big(m^{(t)}(M^{(t)}A^{(t)})_{bb}+m^{(t)}\big((G^{(t)}-M^{(t)})A^{(t)}\big)_{bb}\notag\\
    \;\;\;\;\;\;\;\;\;\;&+(M^{(t)}A^{(t)})_{bb} (G^{(t)}-M^{(t)})_{aa}+(G^{(t)}-M^{(t)})_{aa}\big((G^{(t)}-M^{(t)})A^{(t)}\big)_{bb}\Big) \bigg]\notag\\
    \quad=& \frac{1}{n^{5/2}}{\sum_{ab}}' \mathbb{E} \bigg[m^{(t)}(M^{(t)})_{ab}  (M^{(t)}A^{(t)})_{bb}+m^{(t)}(G^{(t)}-M^{(t)})_{ab} (M^{(t)}A^{(t)})_{bb}\notag\\
    \;\;\;\;\;\;\;\;\;\;&+m^{(t)}(M^{(t)})_{ab}\big((G^{(t)}-M^{(t)})A^{(t)}\big)_{bb}+m^{(t)}(G^{(t)}-M^{(t)})_{ab}\big((G^{(t)}-M^{(t)})A^{(t)}\big)_{bb}\notag\\
    \;\;\;\;\;\;\;\;\;\;&+(M^{(t)})_{ab}(M^{(t)}A^{(t)})_{bb} (G^{(t)}-M^{(t)})_{aa}+(G^{(t)}-M^{(t)})_{ab}(M^{(t)}A^{(t)})_{bb} (G^{(t)}-M^{(t)})_{aa}\notag\\
    \;\;\;\;\;\;\;\;\;\;&+(M^{(t)})_{ab}(G^{(t)}-M^{(t)})_{aa}\big((G^{(t)}-M^{(t)})A^{(t)}\big)_{bb}\notag\\
    \;\;\;\;\;\;\;\;\;\;&+(G^{(t)}-M^{(t)})_{ab}(G^{(t)}-M^{(t)})_{aa}\big((G^{(t)}-M^{(t)})A^{(t)}\big)_{bb}\bigg]\notag\\
     \quad= &\frac{1}{n^{5/2}} \mathbb{E} \bigg[ m^{(t)}\big(M^{(t)}A^{(t)}\big)_{(n+1,n+1)}\langle E \boldsymbol{1},M^{(t)}E^{\dag}\boldsymbol{1}\rangle +m^{(t)}(M^{(t)}A^{(t)})_{bb}\langle \boldsymbol{e_a},(G^{(t)}-M^{(t)})\boldsymbol{e_b} \rangle\notag\\
    \;\;\;\;\;\;\;\;\;\;&+m^{(t)}(M^{(t)})_{ab}\big((G^{(t)}-M^{(t)})A^{(t)}\big)_{bb}+m^{(t)}(G^{(t)}-M^{(t)})_{ab}\big((G^{(t)}-M^{(t)})A^{(t)}\big)_{bb}\notag\\
    \;\;\;\;\;\;\;\;\;\;&+(M^{(t)})_{ab}(M^{(t)}A^{(t)})_{bb} (G^{(t)}-M^{(t)})_{aa}+(G^{(t)}-M^{(t)})_{ab}(M^{(t)}A^{(t)})_{bb} (G^{(t)}-M^{(t)})_{aa}\notag\\
    \;\;\;\;\;\;\;\;\;\;&+(M^{(t)})_{ab}(G^{(t)}-M^{(t)})_{aa}\big((G^{(t)}-M^{(t)})A^{(t)}\big)_{bb}\notag\\
    \;\;\;\;\;\;\;\;\;\;&+(G^{(t)}-M^{(t)})_{ab}(G^{(t)}-M^{(t)})_{aa}\big((G^{(t)}-M^{(t)})A^{(t)}\big)_{bb} \bigg]\notag\\
        \quad=&\mathcal{O}_\prec\Bigg(\frac{1}{n^{5/2}}\bigg(\frac{\sqrt{2n}\sqrt{2n}}{|\beta^{(t)}|(1+\eta)}+\frac{1}{|\beta^{(t)}|}\frac{n^2}{\sqrt{n\eta}}+\frac{n}{\eta|\beta^{(t)}|}+\frac{1}{\sqrt{n\eta}}\frac{n}{\eta|\beta^{(t)}|}\notag\\
    \;\;\;\;\;\;\;\;\;\;&+\frac{n}{(1+\eta)}\frac{1}{|\beta^{(t)}|}\frac{1}{\sqrt{n\eta}}+\frac{1}{\sqrt{n\eta}}\frac{1}{|\beta^{(t)}|}\frac{n}{\eta}+\frac{1}{\sqrt{n\eta}}\frac{n}{\eta|\beta^{(t)}|}+\frac{1}{\sqrt{n\eta}}\frac{1}{\sqrt{n\eta}}\frac{n}{\eta|\beta^{(t)}|}\bigg)\Bigg)\notag\\
     \quad= & \mathcal{O}_\prec\bigg(\frac{1}{|\beta^{(t)}|}\Big(\frac{1}{n^{3/2}(1+\eta)}+\frac{1}{n^2\eta^{3/2}}\Big)\bigg).\notag
 \end{align}

Thus, we can bound all $k=2$ terms in $F_{k\geq 2}^{(t)}$ by $$|\beta^{(t)}|^{-1}\times\left(n^{-3/2}(1+\eta)^{-1}+n^{-2}\eta^{-3/2}\right).$$

In the following paragraph, we argue that the terms corresponding to $k\geq 4$ have insignificant contribution to $F_{k\geq 2}^{(t)},$ and the only significant contribution comes from the terms corresponding to $k=3.$ When $k=4,$ proceeding as \eqref{eq:exp_local_law}, we shall get terms of the form $n^{-7/2}\sum'_{ab}\mathbb{E}[(G^{(t)})_{ab}(G^{(t)})_{aa}(G^{(t)})_{bb}(G^{(t)})_{aa}(G^{(t)}A^{(t)})_{bb}].$ Now, we can apply individual naive bounds on the each $G^{(t)}$ term, which will yield the bound of $\mathcal{O}_{\prec}(|\beta^{(t)}|^{-1}n^{-7/2}\cdot n^{2})=\mathcal{O}_{\prec}(|\beta^{(t)}|^{-1}n^{-3/2}).$ For $k\geq 4,$ the pre-factor will be even smaller than $n^{-7/2}.$ Thus, the terms do not contribute to $F_{k\geq 2}^{(t)}.$

Now, we estimate the contribution of the terms in $F_{k\geq 2}^{(t)}$ when $k=3.$ We begin with the observation from the Condition \ref{Condtion::} that $x_{ab}^{(t)}\stackrel{\text{i.i.d.}}{\sim}n^{-1/2}\chi^{(t)},$ implying $w_{(ab)}^{(t)}\sim n^{-1/2}\chi^{(t)}$ for $a\leq n, b > n$ or $a > n,b\leq n.$ Therefore 
$$
\kappa\left((ba)_t,(ab)_t, (ab)_s, (ba)_s\right)=\frac{1}{(n^{1/2})^{4}}\kappa\left(\chi^{(t)},\bar{\chi}^{(t)}, \chi^{(s)}, \bar{\chi}^{(s)}\right)=\frac{1}{n^2}(\kappa_4)_{ts}.
$$

When $k=3,$ we get a product of three $G^{(t)}$ terms and one $G^{(t)}A^{(t)}$ term in the expansion. In that scenario, all the choices of $\alpha$s except $\alpha = ((ab)_t, (ba)_{t}, (ba)_{t})$ will have a off-diagonal entry of $G^{(t)}$ in the product. When $\alpha = ((ab)_t, (ba)_{t}, (ba)_{t}),$ we get the term $$\mathbb{E}[(G^{(t)})_{aa}(G^{(t)})_{bb}(G^{(t)})_{aa}(G^{(t)}A^{(t)})_{bb}],$$ which results in the following contribution
\begin{align}\label{eq:DeltaG_DeltaG_DeltaG_DeltaGA}
    &-\frac{(\kappa_4)_{tt}}{2n^3}{\sum_{ab}}'(M^{(t)})_{aa}(M^{(t)})_{bb}(M^{(t)})_{aa}(M^{(t)}A^{(t)})_{bb}\\
    &\quad\quad=-\frac{(\kappa_4)_{tt}}{n}\langle M^{(t)}\rangle^3 \langle M^{(t)}A^{(t)}\rangle\notag\\
    &\quad\quad=\frac{i (\kappa_4)_{tt}}{4n} \partial_{\eta}\big((m^{(t)})^4\big)\notag,
\end{align}
where we have used 
$$
\langle M^{(t)}A^{(t)} \rangle = \frac{1-\beta^{(t)}}{\beta^{(t)}}= -i\partial_{\eta}m^{(t)}.
$$
Thus, combining \eqref{eqn: WGA in terms of S_1, S_2} \eqref{eq:F^{(t)}}, and \eqref{eq:DeltaG_DeltaG_DeltaG_DeltaGA}
 we have
\begin{align*}
     &\mathbb{E}\left[\langle \underline{W^{(t)}G^{(t)}}A^{(t)} \rangle\right]\\
     &=-\frac{i}{4n}\partial_{\eta}\log\big(1-(u^{(t)})^2+2(u^{(t)})^3|z|^2-(u^{(t)})^2(z^2+\bar{z}^2)\big)\\
     &\;\;\;\;\;\;\;\;+ \mathcal{O}_\prec\Big(\frac{1}{|z-\bar{z}|^2|\beta^{(t)}|(n\eta)^2}\Big)+\frac{i (\kappa_4)_{tt}}{4n} \partial_{\eta}\big((m^{(t)})^4\big)\\
     &\;\;\;\;\;\;\;\;+\mathcal{O}_\prec\Big(\frac{1}{|\beta^{(t)}|n^{3/2}(1+\eta)}+\frac{1}{|\beta^{(t)}|n^2\eta^{3/2}}\Big).
\end{align*}
Finally,
\begin{align*}
    &\mathbb{E}[\langle G^{(t)}-M^{(t)}\rangle]\\
    &=\mathbb{E}[\langle \underline{-W^{(t)}G^{(t)}}A^{(t)}\rangle]+\mathcal{O}_\prec\Big(\frac{\norm{((\mathcal{B}^{(t)})^*)^{-1}}[\mathbb{I}]}{(n\eta)^2}\Big)\\
    &=-\frac{i (\kappa_4)_{tt}}{4n} \partial_{\eta}\big((m^{(t)})^4\big)+\frac{i}{4n}\partial_{\eta}\log\big(1-(u^{(t)})^2+2(u^{(t)})^3|z|^2-(u^{(t)})^2(z^2+\bar{z}^2)\big)\\
    &\quad+\mathcal{O}_\prec\Big(\frac{1}{|\beta^{(t)}|n^{3/2}(1+\eta)}+\frac{1}{|\beta^{(t)}|n^2\eta^{3/2}}\Big)\\
    &\quad+ \mathcal{O}_\prec \Big(\frac{1}{|z-\bar{z}|^2}\frac{1}{|\beta^{(t)}|(n\eta)^2}\Big)+\mathcal{O}_\prec\Big(\frac{1}{|\beta^{(t)}|(n\eta)^2}\Big)\\
    &=:\xi^{(t)}+\mathcal{O}_\prec\bigg(\Big
   (\frac{1}{|1-|z||}+\frac{1}{|\Im{z}|^2}\Big)\Big(\frac{1}{n^{3/2}(1+\eta)}+\frac{1}{(n\eta)^2}\Big)\bigg)\\
    &=:\xi^{(t)}+ Error_{exp}.
\end{align*}
Therefore,
\begin{align}
    \mathbb{E}\left[\langle cG^{(1)}+dG^{(2)}-cM^{(1)}-dM^{(2)}\rangle\right]
    =c\xi^{(1)}+d\xi^{(2)} + Error_{exp},
\end{align}
which completes the proof of Lemma \ref{Lemma:Exp or resolvent}.
\end{proof}
Now, we are ready to present the proof of the Theorem \ref{Thm:CLT for resolvents}. Using Lemma \ref{Lemma:Exp or resolvent} along with the equation \eqref{eq:G^t-M^t}, $\mathbb{E}\Big[\prod_{i=1}^{p}\big\{c\langle G^{(1)}_i-\mathbb{E}(G^{(1)}_i)\rangle +d\langle G^{(2)}_i-\mathbb{E}(G^{(2)}_i)\rangle\big\}\Big]$ is equivalent to computing
\begin{align*}
    &\mathbb{E}\Big[\prod_{i\in [p]}\langle-c\underline{W^{(1)}G^{(1)}_i}A^{(1)}_i - c\xi_i^{(1)}-d\underline{W^{(2)}G^{(2)}_i}A^{(2)}_i - d\xi_i^{(2)}\rangle\Big]\\
    &\quad+ \mathcal{O}_{\prec}\bigg(\frac{\psi^{(1)}}{n\eta_i}\bigg)+\mathcal{O}_{\prec}\bigg(\frac{\psi^{(2)}} {n\eta_i}\bigg),
\end{align*}
where
\begin{align}\label{definition fo psi_t}
    \psi^{(t)} = \prod_{i\in [p]}\bigg(\frac{1}{|\beta^{(t)}_i|}+\frac{1}{(\Im{z_i})^2}\bigg)\frac{1}{n\eta_i} \leq \prod_{i\in [p]}\bigg(\frac{1}{|1-|z_i||}+\frac{1}{(\Im{z_i})^2}\bigg)\frac{1}{n\eta_i},
\end{align}
and $\psi^{(1)}+\psi^{(2)} \leq \psi$, where $\psi$ is defined in \eqref{eqn: definition of psi error}.

Recall that 
$$\Upsilon_{i}=c\langle -\underline{W^{(1)}G_i^{(1)}}A_i^{(1)}-\xi_i^{(1)}\rangle +d\langle -\underline{W^{(2)}G_i^{(2)}}A_i^{(2)}-\xi_i^{(2)}\rangle.$$
We now begin with the expansion of $\mathbb{E}\big[\prod_{i\in [p]}\Upsilon_{i}\big]$ as follows;

\begin{align}\label{eq:I^{(1)}+I^{(2)}}
        &\mathbb{E}\bigg[\prod_{i\in[p]}\Upsilon_{i}\bigg]\\
        &=\mathbb{E}\bigg[\Upsilon_{1}\prod_{i\neq1}\Upsilon_{i} \bigg]\notag\\
        &=\mathbb{E}\bigg[c\langle-\underline{W^{(1)}G_1^{(1)}}A_1^{(1)}-\xi_1^{(1)}\rangle\prod_{i\neq1}\Upsilon_{i}\bigg]
        +\mathbb{E}\bigg[d\langle-\underline{W^{(2)}G_1^{(2)}}A_1^2-\xi_1^{(2)}\rangle\prod_{i\neq1}\Upsilon_{i}\bigg]\notag\\
        &=: I^{(1)}+I^{(2)}.\notag
\end{align}

Now, we shall evaluate $I^{(1)}.$
\begin{align}\label{eq:I^{(1)}}
     &I^{(1)}\\
        &=\mathbb{E}\Big[c\langle-\underline{W^{(1)}G_1^{(1)}}A_1^{(1)}-\xi_1^{(1)}\rangle\prod_{i\neq1}\Upsilon_{i}\Big]\notag\\
        &=\mathbb{E}\Big[c\langle-W^{(1)}G_1^{(1)}A_1^{(1)}-\mathcal{S}[G_1^{(1)}]G_1^{(1)}A_1^{(1)}-\xi_1^{(1)}\rangle\prod_{i\neq1}\Upsilon_{i}\Big]\notag\\
        &=-c\langle\xi_1^{(1)}\rangle\mathbb{E}\Big[\prod_{i\neq1}\Upsilon_{i}\Big]+c\mathbb{E}\Big[\langle-\mathcal{S}[G_1^{(1)}]G_1^{(1)}A_1^{(1)}\rangle\prod_{i\neq1}\Upsilon_{i}\Big]\notag\\
        &\quad+c\mathbb{E}\Big[\langle-W^{(1)}G_1^{(1)}A_1^{(1)}\rangle\prod_{i\neq1}\Upsilon_{i}\Big]\notag\\
        &=-c\langle\xi_1^{(1)}\rangle\mathbb{E}\Big[\prod_{i\neq1}\Upsilon_{i}\Big]-c\frac{1}{2n^2}\mathbb{E}\Big[\sum_{ab}(G^{(1)}_1)_{aa}(G_1^{(1)}A_1^{(1)})_{bb}\prod_{i\neq1}\Upsilon_{i}\Big]\notag\\
        &\quad+c\mathbb{E}\Big[\langle-W^{(1)}G_1^{(1)}A_1^{(1)} \rangle\prod_{i\neq1}\Upsilon_{i}\Big]\notag\\
        &=: c(I_1^{(1)}+I_2^{(1)}+I_3^{(1)})\notag.
\end{align}
 We shall now estimate $I_{1}^{(1)}, I_{2}^{(1)},$ and $I_{3}^{(1)}$ individually. From a glimpse of the upcoming calculations, we shall see from \eqref{eq:T_1} that $I_{2}^{(1)}$ actually gets canceled by a part of $I_{3}^{(1)}.$ The primary ingredient in these estimations is the cumulant expansion as mentioned in the Result \ref{Result:Cumulant Expansion}.

\subsection{Evaluation of $I_{3}^{(1)}$}\label{subsec: I_3^1}

\begin{align}\label{eq:cummulant_exp_I_1^{(3)}}
     &I_3^{(1)}\\
         &=\mathbb{E}\Big[\langle-W^{(1)}G_1^{(1)}A_1^{(1)}\rangle\prod_{i\neq1}\Upsilon_{i}\Big]\notag\\
         &=\mathbb{E}\Big[\sum_{ab}w^{(1)}_{ba}\langle-\Delta^{ba}G_1^{(1)}A_1^{(1)}\rangle\prod_{i\neq1}\Upsilon_{i}\Big]
         \notag\\
         &=\sum_{ab}\sum_{k\geq 0}\sum_{\alpha\in\{(ab)_1,(ba)_1,(ab)_2,(ba)_2\}^k}\frac{\kappa((ba)_1,\alpha)}{k!}\mathbb{E}\Big[\partial_\alpha\big(\langle-\Delta^{ba}G_1^{(1)}A_1^{(1)}\rangle\prod_{i\neq1}\Upsilon_{i}\big)\Big].
         \notag
 \end{align}
Notice that in \eqref{eq:cummulant_exp_I_1^{(3)}}, the term with $k = 0$ shall not contribute since according to the Condition \ref{Condtion::}, $\mathbb{E}[w_{ba}^{(1)}] = 0$. Thus, from \eqref{eq:cummulant_exp_I_1^{(3)}} we have

\begin{align}\label{eq:I_1^{(3)}_final}
    &I_3^{(1)}\\
        &=\sum_{ab}\kappa((ba)_1,(ab)_1)\mathbb{E}\Big[\partial_{(ab)_1}[\langle-\Delta^{ba}G_1^{(1)}A_1^{(1)}\rangle\prod_{i\neq1}\Upsilon_{i}]\Big]
        +\sum_{ab}\kappa((ba)_1,(ba)_1)\notag\\
        &\;\;\;\;\;\mathbb{E}\Big[\partial_{(ba)_1}[\langle-\Delta^{ba}G_1^{(1)}A_1^{(1)}\rangle
        \prod_{i\neq1}\Upsilon_{i}]\Big]+\sum_{ab}\kappa((ba)_1,(ab)_2)\mathbb{E}\Big[\partial_{(ab)_2}[\langle-\Delta^{ba}\notag\\
        &\;\;\;\;\;G_1^{(1)}A_1^{(1)}\rangle\prod_{i\neq1}\Upsilon_{i}]\Big]+\sum_{ab}\kappa((ba)_1,(ba)_2)\mathbb{E}\Big[\partial_{(ba)_2}[\langle-\Delta^{ba}G_1^{(1)}A_1^{(1)}\rangle\prod_{i\neq1}\Upsilon_{i}]\Big]\notag\\
        &\;\;\;\;
        +\sum_{ab}\sum_{k\geq 2}\sum_{\alpha\in\{(ab)_1,(ba)_1,(ab)_2,(ba)_2\}^k}\frac{\kappa((ba)_1,\alpha)}{k!}\mathbb{E}\Big[\partial_\alpha\big(\langle-\Delta^{ba}G_1^{(1)}A_1^{(1)}\rangle\prod_{i\neq1}\Upsilon_{i}\big)\Big]\notag\\
        &=\frac{1}{n}\sum_{ab}\mathbb{E}\Big[\partial_{(ab)_1}[\langle-\Delta^{ba}G_1^{(1)}A_1^{(1)}\rangle\prod_{i\neq1}\Upsilon_{i}]\Big]+\frac{\gamma_1}{n}\sum_{ab}\mathbb{E}\Big[\partial_{(ba)_1}[\langle-\Delta^{ba}G_1^{(1)}A_1^{(1)}\rangle
        \notag\\
        &\;\;\;\;\prod_{i\neq1}\Upsilon_{i}]\Big]+\frac{\gamma}{n} \sum_{ab}\mathbb{E}\Big[\partial_{(ab)_2}[\langle-\Delta^{ba}G_1^{(1)}A_1^{(1)}\rangle\prod_{i\neq1}\Upsilon_{i}]\Big]\notag\\
        &\;\;\;\;+\frac{\rho}{n} \sum_{ab}\mathbb{E}\Big[\partial_{(ba)_2}[\langle-\Delta^{ba}G_1^{(1)}A_1^{(1)}\rangle\prod_{i\neq1}\Upsilon_{i}]\Big]\notag\\
        &\;\;\;\;+\sum_{ab}\sum_{k\geq 2}\sum_{\alpha\in\{(ab)_1,(ba)_1,(ab)_2,(ba)_2\}^k}\frac{\kappa((ba)_1,\alpha)}{k!}\mathbb{E}\Big[\partial_\alpha\big(\langle-\Delta^{ba}G_1^{(1)}A_1^{(1)}\rangle\prod_{i\neq1}\Upsilon_{i}\big)\Big]\notag\\
        &=:\frac{1}{n}T_1^{(1)}+\frac{1}{n}\gamma_1 T_2^{(1)} +\frac{1}{n}\gamma T_3^{(1)} +\frac{1}{n}\rho T_4^{(1)}+T_5^{(1)}\notag.
\end{align}

The symbols $\gamma_{1}, \gamma, \rho$ are same as defined in the Condition \ref{Condtion::}. For convenience, we list those down below,
$$\kappa((ba)_1, (ab)_1)=\mathbb{E}[(ba)_1,\bar{(ba)_1}]=\mathbb{E}[|(ba)_1|^2]=\frac{1}{n},$$
    $$\kappa((ba)_1, (ba)_2)=\mathbb{E}[(ba)_1,(ba)_2]=:\frac{1}{n}\rho,$$
    $$\kappa((ba)_1, (ab)_2)=\mathbb{E}[(ba)_1,(ab)_2]=: \frac{1}{n}\gamma,$$
and 

\begin{equation}\notag
    \kappa((ba)_1, (ba)_1) = \frac{1}{n} \gamma_1 =  \left\{\begin{array}{ll} 0 & \text {  when $X^{(t)}$ in condition \ref{Condtion::} is $\mathbb{C}$-valued,}  \\ 1/n & \text {  when $X^{(t)}$ in condition \ref{Condtion::} is $\mathbb{R}$-valued.}\end{array} \right.
\end{equation}

We delegate the calculations of the factors $T_{1}^{(1)}, T_{2}^{(1)}, T_{3}^{(1)}, T_{4}^{(1)}, T_{5}^{(1)}$ in the following subsections.
\subsubsection{Evaluation of $T_{1}^{(1)}, T_{2}^{(1)}, T_{3}^{(1)}, T_{4}^{(1)}$}
Let us first evaluate $T_1^{(1)}$ as follows;
\begin{align}\label{eq:T_1}
     &T_1^{(1)}\\
        &=\sum_{ab}\mathbb{E}\Big[\partial_{(ab)_1}\big(\langle-\Delta^{ba}G_1^{(1)}A_1^{(1)} \rangle \prod_{i\neq1}\Upsilon_{i}\big)\Big]\notag\\
         &=\frac{1}{2n}{\sum_{ab}}'\mathbb{E}\Big[(G_1^{(1)})_{aa}(G_1^{(1)}A_1^{(1)})_{bb}\prod_{i\neq1}\Upsilon_{i}\Big]\notag\\
        &\;\;\;\;+\sum_{ab}\mathbb{E}\Big[\langle-\Delta^{ba}G_1^{(1)}A_1^{(1)} \rangle \sum_{i\neq1}\partial_{(ab)_1} \big(c\langle -\underline{W^{(1)}G_i^{(1)}}A_i^{(1)}\rangle\big) \prod_{j\neq1,i}\Upsilon_{j}\Big]\notag\\
        &=-nI_2^{(1)}+{\sum_{ab}}'\mathbb{E}\Big[\langle-\Delta^{ba}G_1^{(1)}A_1^{(1)} \rangle \sum_{i\neq1}c\langle- \Delta^{ab}G_i^{(1)}A_i^{(1)}\notag\\
        &\;\;\;\;+\underline{W^{(1)}G_i^{(1)}\Delta^{ab}G_i^{(1)}}A_i^{(1)}\rangle \prod_{j\neq1,i}\Upsilon_{j}\Big]\notag\\
        &= -nI_2^{(1)}+c\sum_{i\neq1}\mathbb{E} \Big[{\sum_{ab}}'\big(\langle-\Delta^{ba}G_1^{(1)}A_1^{(1)} \rangle \langle- \Delta^{ab}G_i^{(1)}A_i^{(1)}\rangle\notag\\
        &\;\;\;\;+\langle-\Delta^{ba}G_1^{(1)}A_1^{(1)} \rangle\langle\underline{W^{(1)}G_i^{(1)}\Delta^{ab}G_i^{(1)}}A_i^{(1)}\rangle\big) \prod_{j\neq1,i}\Upsilon_{j}\Big]\notag\\
         &=-nI_2^{(1)}+\frac{c}{2n}\sum_{i\neq1}\mathbb{E} \Big[\langle G_1^{(1)}A_1^{(1)}E_1G_i^{(1)}A_i^{(1)}E_2+G_1^{(1)}A_1^{(1)}E_2G_i^{(1)}A_i^{(1)}E_1\rangle \notag\\
         &\;\;\;\;\;\prod_{j\neq1,i}\Upsilon_{j}\Big]+\frac{c}{2n}\sum_{i\neq1}\mathbb{E} \Big[\langle-G_1^{(1)}A_1^{(1)}E_1 \underline{G_i^{(1)}A_i^{(1)}W^{(1)}G_i^{(1)}}E_2\notag\\
         &\;\;\;\;\; -G_1^{(1)}A_1^{(1)}E_2 \underline{G_i^{(1)} A_i^{(1)}W^{(1)}G_i^{(1)}}E_1\rangle \prod_{j\neq1,i}\Upsilon_{j}\Big]\notag\\
         &=-nI_2^{(1)}+\frac{c}{2n}\sum_{i\neq1}\mathbb{E}\Big[ \langle G_1^{(1)}A_1^{(1)}EG_i^{(1)}A_i^{(1)}E^{\dag}\notag\\
         &\;\;\;\;\;-G_1^{(1)}A_1^{(1)}E\underline{G_i^{(1)}A_i^{(1)}W^{(1)}G_i^{(1)}}E^{\dag}\rangle\prod_{j\neq1,i}\Upsilon_{j}\Big].\notag
\end{align}

Here, $I_{2}^{(1)}$ is same as defined in \eqref{eq:I^{(1)}}, and $E_{1}, E_{2}, E, E^{\dag}$ are same as defined in \eqref{eqn: dfn of E_1 and E_2}. Similarly, we can calculate
\begin{align}
    &T_2^{(1)}\\
    =&\frac{1}{2n}{\sum_{ab}}'\mathbb{E}\Big[(G_1^{(1)})_{ab}(G_1^{(1)}A_1^{(1)})_{ab}\prod_{i\neq1} \Upsilon_{i} \Big]+\frac{c}{2n}\sum_{i\neq1}\mathbb{E}\Big[ \langle G_1^{(1)}A_1^{(1)}E(A_i^{(1)})^t(G_i^{(1)})^tE^{\dag}\notag\\
         &\;\;\;\;\;-G_1^{(1)}A_1^{(1)}E\underline{(G_i^{(1)})^tW^{(1)}(A_i^{(1)})^t(G_i^{(1)})^t}E^{\dag}\rangle \prod_{j\neq1,i}\Upsilon_{j}\Big],\notag\\
    &T_3^{(1)}\label{eq:T_3^{(1)}_final}\\
    =&\frac{d}{2n}\sum_{i\neq1}\mathbb{E}\Big[\langle G_1^{(1)}A_1^{(1)}EG_i^{(2)}A_i^{(2)}E^{\dag}+ \underline{G_1^{(1)}A_1^{(1)}EG_i^{(2)}A_i^{(2)}W^{(2)}G_i^{(2)}}E^{\dag}\rangle \prod_{j\neq1,i}\Upsilon_{j}\Big],\notag\\
    &T_4^{(1)}\label{eq:T_4^{(1)}_final}\\
    =&\frac{d}{2n}\sum_{i\neq1}\mathbb{E}\bigg[\Big(\langle G_1^{(1)}A_1^{(1)}E_1G_i^{(2)}A_i^{(2)}E_1\rangle+\langle G_1^{(1)}A_1^{(1)}E_2G_i^{(2)}A_i^{(2)}E_2\rangle\notag\\
    &+\langle \underline{G_1^{(1)}A_1^{(1)}E_1G_i^{(2)}A_i^{(2)}W^{(2)}G_i^{(2)}}E_1\rangle+\langle \underline{G_1^{(1)}A_1^{(1)}E_2G_i^{(2)}A_i^{(2)}W^{(2)}G_i^{(2)}}E_2 \rangle \Big)\notag\\
    &\prod_{j\neq1,i}\Upsilon_{j}\bigg],\notag
\end{align}
where $A^t$ denotes the transpose of matrix $A.$

\subsection{Reduced form of $I^{(1)}+I^{(2)}$} Now, we use the above expressions of $T_{1}^{(1)}, T_{2}^{(1)}, T_{3}^{(1)}, T_{4}^{(1)}$ and $T_{5}^{(1)}$ in conjunction with \eqref{eq:I_1^{(3)}_final} to evaluate $I^{(1)}+I^{(2)}.$ First, we notice that the expression of $I_{3}^{(1)}$ in equation \eqref{eq:I_1^{(3)}_final} contains the factor $T_{1}^{(1)}+\gamma_{1}T_{2}^{(1)}.$ But, as per the definition of $\gamma_{1},$ it can take two possible values; $\gamma_{1}=0$ (complex value case), and $\gamma_{1}=1$ (real value case). In the first case, as per the equation \eqref{eq:I_1^{(3)}_final}, we observe that $T_{2}^{(1)}$ will not have any contribution to the $I_{3}^{(1)}.$ In the rest of the section, we present the calculation for the second case i.e., $\gamma_{1}=1.$ The other case when $\gamma_{1}=0$ is discussed in Remark \ref{remark: real vs complex case}.

\begin{align}\label{eq:T_1^{(1)}+T_2^{(1)}_final}
     &T_1^{(1)}+\gamma_1 T_2^{(1)}\\
    &= T_1^{(1)}+T_2^{(1)}\notag\\
    &=-nI_2^{(1)}+\frac{1}{2n}{\sum_{ab}}'\mathbb{E}\Big[(G_1^{(1)})_{ab}(G_1^{(1)}A_1^{(1)})_{ab}\prod_{i\neq1} \Upsilon_{i} \Big]\notag\\
    &\;\;\;\;+\frac{c}{2n}\sum_{i\neq1}\mathbb{E}\Big[ \langle G_1^{(1)}A_1^{(1)}E\big(G_i^{(1)}A_i^{(1)}+ (A_i^{(1)})^t(G_i^{(1)})^t \big)E^{\dag}\notag\\
    &\;\;\;\;-G_1^{(1)}A_1^{(1)}E\big(\underline{G_i^{(1)}A_i^{(1)}W^{(1)}G_i^{(1)}}+\underline{(G_i^{(1)})^tW^{(1)}(A_i^{(1)})^t(G_i^{(1)})^t}\big)E^{\dag}\rangle\prod_{j\neq1,i}\Upsilon_{j}\Big].\notag
\end{align}

Thus, from \eqref{eq:I^{(1)}}, \eqref{eq:I_1^{(3)}_final}, \eqref{eq:T_3^{(1)}_final}, \eqref{eq:T_4^{(1)}_final}, and \eqref{eq:T_1^{(1)}+T_2^{(1)}_final} we have 
\begin{align}\label{eq:I^{(1)}_final}
    &I^{(1)}\\
    =&c(I_1^{(1)}+I_2^{(1)}+I_3^{(1)})\notag\\
    =&c\Big(I_1^{(1)}+I_2^{(1)}+\frac{1}{n}T_1^{(1)}+\frac{1}{n}T_2^{(1)}+\frac{\gamma}{n} T_3^{(1)}+\frac{\rho}{n} T_4^{(1)}+T_5^{(1)}\Big)\notag\\
        =&cT_5^{(1)}-c\langle\xi_1^{(1)}\rangle\mathbb{E}\Big[\prod_{i\neq1}\Upsilon_{i}\Big]+\frac{c}{2n^2}{\sum_{ab}}'\mathbb{E}\Big[(G_1^{(1)})_{ab}(G_1^{(1)}A_1^{(1)})_{ab}\prod_{i\neq1} \Upsilon_{i} \Big]\notag\\
        &+\frac{c^2}{2n^2}\sum_{i\neq1}\mathbb{E}\Big[\langle G_1^{(1)}A_1^{(1)}E\big(G_i^{(1)}A_i^{(1)}+ (A_i^{(1)})^t(G_i^{(1)})^t \big)E^{\dag}\notag\\
        &-G_1^{(1)}A_1^{(1)}E\big(\underline{G_i^{(1)}A_i^{(1)}W^{(1)}G_i^{(1)}}+\underline{(G_i^{(1)})^tW^{(1)}(A_i^{(1)})^t(G_i^{(1)})^t}\big)E^{\dag}\rangle \prod_{j\neq1,i}\Upsilon_{j}\Big]\notag\\
        &+ \frac{cd\gamma}{2n^2}\sum_{i\neq1}\mathbb{E}\Big[\big(\langle G_1^{(1)}A_1^{(1)}EG_i^{(2)}A_i^{(2)}E^{\dag}\rangle+\langle \underline{G_1^{(1)}A_1^{(1)}EG_i^{(2)}A_i^{(2)}W^{(2)}G_i^{(2)}}E^{\dag}\rangle \big)\notag\\
        &\quad\prod_{j\neq1,i}\Upsilon_{j}\Big]+\frac{cd\rho}{2n^2}\sum_{i\neq1}\mathbb{E}\Big[\big(\langle G_1^{(1)}A_1^{(1)}E_1G_i^{(2)}A_i^{(2)}E_1\rangle+\langle G_1^{(1)}A_1^{(1)}E_2G_i^{(2)}A_i^{(2)}E_2\rangle\big)\notag\\
        &\quad\prod_{j\neq1,i}\Upsilon_{j}\Big]+\frac{cd\rho}{2n^2}\sum_{i\neq1}\mathbb{E}\Big[\big(\langle \underline{G_1^{(1)}A_1^{(1)}E_1G_i^{(2)}A_i^{(2)}W^{(2)}G_i^{(2)}}E_1\rangle\notag\\
        &+\langle \underline{G_1^{(1)}A_1^{(1)}E_2G_i^{(2)}A_i^{(2)}W^{(2)}G_i^{(2)}}E_2 \rangle \big)\prod_{j\neq1,i}\Upsilon_{j}\Big].\notag
\end{align}

Similarly, an equivalent expression of $I^{(2)}$ can be obtained as well. Now, by using \eqref{eq:I^{(1)}+I^{(2)}}, the expressions of $I^{(1)}$ from \eqref{eq:I^{(1)}_final}, and an equivalent expression for $I^{(2)},$ we have

\begin{align}\label{eq:I^{(1)}+I^{(2)}Noataions}
    &\mathbb{E}\Big[\prod_{i\in[p]}\Upsilon_{i}\Big]\notag\\
        =&cT_5^{(1)}+dT_5^{(2)}-c\langle\xi_1^{(1)}\rangle\mathbb{E}\Big[\prod_{i\neq1}\Upsilon_{i}\Big]-d\langle\xi_1^{(2)}\rangle\mathbb{E}\Big[ \prod_{i\neq1} \Upsilon_{i}\Big]\notag\\
        &+\frac{c}{2n^2}{\sum_{ab}}'\mathbb{E}\Big[(G_1^{(1)})_{ab}(G_1^{(1)}A_1^{(1)})_{ab}\prod_{i\neq1} \Upsilon_{i}\Big] +\frac{d}{2n^2}{\sum_{ab}}'\mathbb{E}\Big[(G_1^{(2)})_{ab}(G_1^{(2)}A_1^2)_{ab}\prod_{i\neq1} \Upsilon_{i} \Big]\notag\\
        &+\frac{c^2}{2n^2}\sum_{i\neq1}\mathbb{E} \Big[\langle G_1^{(1)}A_1^{(1)}E\big(G_i^{(1)}A_i^{(1)}+ (A_i^{(1)})^t(G_i^{(1)})^t \big)E^{\dag}-G_1^{(1)}A_1^{(1)}E\big(\underline{G_i^{(1)}A_i^{(1)}W^{(1)}G_i^{(1)}}\notag\\
        &+\underline{(G_i^{(1)})^tW^{(1)}(A_i^{(1)})^t(G_i^{(1)})^t}\big)E^{\dag}\rangle \prod_{j\neq1,i}\Upsilon_{j}\Big]+\frac{d^2}{2n^2}\sum_{i\neq1}\mathbb{E} \Big[\langle G_1^{(2)}A_1^{(2)}E\big(G_i^{(2)}A_i^{(2)}\notag\\
        &+ (A_i^{(2)})^t(G_i^{(2)})^t\big)E^{\dag}-G_1^{(2)}A_1^{(2)}E\big(\underline{G_i^{(2)}A_i^{(2)}W^{(2)}G_i^{(2)}}+\underline{(G_i^{(2)})^tW^{(2)}(A_i^{(2)})^t(G_i^{(2)})^t}\big)E^{\dag}\rangle \notag\\
        &\;\;\prod_{j\neq1,i}\Upsilon_{j}\Big]+ \frac{cd\gamma}{2n^2}\sum_{i\neq1}\mathbb{E}\Big[\langle G_1^{(1)}A_1^{(1)}EG_i^{(2)}A_i^{(2)}E^{\dag}\rangle \prod_{j\neq1,i}\Upsilon_{j}\Big]\notag\\
        &+\frac{cd\gamma}{2n^2}\sum_{i\neq1}\mathbb{E}\Big[\langle G_1^{(2)}A_1^{(2)}EG_i^{(1)}A_i^{(1)}E^{\dag}\rangle\prod_{j\neq1,i}\Upsilon_{j}\Big]+\frac{cd\rho}{2n^2}\sum_{i\neq1}\mathbb{E}\Big[\big(\langle G_1^{(1)}A_1^{(1)}E_1G_i^{(2)}A_i^{(2)}E_1\rangle\notag\\
        &+\langle G_1^{(1)}A_1^{(1)}E_2G_i^{(2)}A_i^{(2)}E_2\rangle\big)\prod_{j\neq1,i}\Upsilon_{j}\Big]+\frac{cd\rho}{2n^2}\sum_{i\neq1}\mathbb{E}\Big[\big(\langle G_1^{(2)}A_1^{(2)}E_1G_i^{(1)}A_i^{(1)}E_1\rangle\notag\\
        &+\langle G_1^{(2)}A_1^{(2)}E_2G_i^{(1)}A_i^{(1)}E_2\rangle \big)\prod_{j\neq1,i}\Upsilon_{j}\Big]+\frac{cd\gamma}{2n^2}\bigg(\sum_{i\neq1}\mathbb{E}\Big[\langle\underline{G_1^{(1)}A_1^{(1)}EG_i^{(2)}A_i^{(2)}W^{(2)}G_i^{(2)}}E^{\dag}\rangle \notag\\
         &\;\;\prod_{j\neq1,i}\Upsilon_{j}\Big]
        +\sum_{i\neq1}\mathbb{E}\Big[\langle \underline{G_1^{(2)}A_1^{(2)}EG_i^{(1)}A_i^{(1)}W^{(1)}G_i^{(1)}}E^{\dag}\rangle \prod_{j\neq1,i}\Upsilon_{j}\Big]\bigg)\notag\\
        &+\frac{cd\rho}{2n^2}\sum_{i\neq1}\mathbb{E}\Big[\big(\langle \underline{G_1^{(1)}A_1^{(1)}E_1G_i^{(2)}A_i^{(2)}W^{(2)}G_i^{(2)}}E_1\rangle+\langle \underline{G_1^{(1)}A_1^{(1)}E_2G_i^{(2)}A_i^{(2)}W^{(2)}G_i^{(2)}}E_2 \rangle \big)\notag\\
        &\;\;\prod_{j\neq1,i}\Upsilon_{j}\Big]+\frac{cd\rho}{2n^2}\sum_{i\neq1}\mathbb{E}\Big[\Big(\langle \underline{G_1^{(2)}A_1^{(2)}E_1G_i^{(1)}A_i^{(1)}W^{(1)}G_i^{(1)}}E_1\rangle\notag\\
        &+\langle \underline{G_1^{(2)}A_1^{(2)}E_2G_i^{(1)}A_i^{(1)}W^{(1)}G_i^{(1)}}E_2 \rangle \Big)\prod_{j\neq1,i}\Upsilon_{j}\Big]\notag\\
         =&:cT_5^{(1)}+dT_5^{(2)}-c\langle\xi_1^{(1)}\rangle\mathbb{E}\Big[\prod_{i\neq1}\Upsilon_{i}\Big]-d\langle\xi_1^{(2)}\rangle\mathbb{E}\Big[\prod_{i\neq1} \Upsilon_{i}\Big]
        +\frac{c}{2n^2}S^{(1)}_1 \\
        &+\frac{d}{2n^2}S^{(2)}_1+\frac{c^2}{2n^2}S^{(1)}_2
        +\frac{d^2}{2n^2}S^{(2)}_2+ \frac{cd\gamma}{2n^2}S^{(1)}_3+\frac{cd\gamma}{2n^2}S^{(2)}_3
        +\frac{cd\rho}{2n^2}S^{(1)}_4
        \notag\\
        &+\frac{cd\rho}{2n^2}S^{(2)}_4 +\frac{cd\gamma}{2n^2}\big(S^{(1)}_5
        +S^{(2)}_5\big)
        +\frac{cd\rho}{2n^2}S^{(1)}_6
        +\frac{cd\rho}{2n^2}S^{(2)}_6.\notag
\end{align}

We note that the $S_{i}^{(t)}$ terms actually originated from the cumulant expansion \eqref{eq:I_1^{(3)}_final}, when $k=1,$ and the $T_{i}^{(t)}$ terms originated from the terms  corresponding to $k\geq 2.$

\subsection{Evaluation of $S_{i}^{(t)}$ terms in \eqref{eq:I^{(1)}+I^{(2)}Noataions}}

Now, we proceed to the evaluations of $S_{i}^{(t)}.$ The evaluation of $S_{2}^{(t)}$s are explained in the Subsection \ref{subsec: S_2 evaluation}, and rest are calculated in the Subsection \ref{subsec: Evaluation of other S_i}.

\subsubsection{Evaluation of $S_{2}^{(t)}$s}\label{subsec: S_2 evaluation} We shall use the notation $G^{(t)}_{\bar{i}} = (G^{(t)})^{\bar{z_i}}.$ We first identify the main contributing terms in $S_{2}^{(1)}$ as follows;
\begin{align}\label{eq:S^{(1)}_2_ver1}
    &S^{(1)}_2\\
    =&\sum_{i\neq1}\mathbb{E} \bigg[\langle G_1^{(1)}A_1^{(1)}E\big(G_i^{(1)}A_i^{(1)}+ (A_i^{(1)})^t(G_i^{(1)})^t\big)E^{\dag}\notag\\
        &-G_1^{(1)}A_1^{(1)}E\big(\underline{G_i^{(1)}A_i^{(1)}W^{(1)}G_i^{(1)}}+\underline{(G_i^{(1)})^tW^{(1)}(A_i^{(1)})^t(G_i^{(1)})^t}\big)E^{\dag}\rangle \prod_{j\neq1,i}\Upsilon_{j}\bigg]\notag\\
        =&\sum_{i\neq1}\mathbb{E} \bigg[\Big\{\langle G_1^{(1)}A_1^{(1)}E\big(G_i^{(1)}A_i^{(1)}+ (A_i^{(1)})^t(G_i^{(1)})^t \big)E^{\dag}\rangle E\notag\\
        &+\langle G_1^{(1)}\mathcal{S}[G_1^{(1)}A_1^{(1)}G_i^{(1)}A_i^{(1)}]G_i^{(1)}E^\dag \rangle+G_1^{(1)}\mathcal{S}[G_1^{(1)}A_1^{(1)}EG_{\bar{i}}^{(1)}](A_i^{(1)})^tG_{\bar{i}}^{(1)}E^\dag \notag\\
        &-\langle \underline{G_1^{(1)}A_1^{(1)}EG_i^{(1)}A_i^{(1)}W^{(1)}G_i^{(1)}E^\dag}- \underline{G_1^{(1)}A_1^{(1)}EG_{\bar{i}}^{(1)}W^{(1)}(A_i^{(1)})^tG_{\bar{i}}^{(1)}E^\dag}\rangle\Big\} \notag\\
        &\;\;\prod_{j\neq1,i}\Upsilon_{j}\bigg]\notag\\
         =&\sum_{i\neq1}\mathbb{E} \bigg[\Big\{\langle G_1^{(1)}A_1^{(1)}E\big(G_i^{(1)}A_i^{(1)}+ (A_i^{(1)})^t(G_i^{(1)})^t \big)E^{\dag}\rangle+\langle G_1^{(1)}\mathcal{S}[G_1^{(1)}A_1^{(1)}E\notag\\
        &\;\;G_i^{(1)}A_i^{(1)}]G_i^{(1)}E^\dag \rangle+G_1^{(1)}\mathcal{S}[G_1^{(1)}A_1^{(1)}EG_{\bar{i}}^{(1)}](A_i^{(1)})^tG_{\bar{i}}^{(1)}E^\dag\rangle \Big\}\prod_{j\neq1,i}\Upsilon_{j}\bigg]\notag\\
        &-\sum_{i\neq1}\mathbb{E}\bigg[\Big\{\langle \underline{G_1^{(1)}A_1^{(1)}EG_i^{(1)}A_i^{(1)}W^{(1)}G_i^{(1)}E^\dag}\notag\\
        &+\underline{G_1^{(1)}A_1^{(1)}EG_{\bar{i}}^{(1)}W^{(1)}(A_i^{(1)})^tG_{\bar{i}}^{(1)}E^\dag}\rangle \Big\}\prod_{j\neq1,i}\Upsilon_{j}\bigg].\notag
\end{align}
In the following discussion, we show that the main contribution comes only from the terms of the form $G_1^{(1)}A_1^{(1)}EG_i^{(1)}A_i^{(1)}E^{\dag}$ and $G_1^{(1)}\mathcal{S}[G_1^{(1)}A_1^{(1)}EG_{\bar{i}}^{(1)}](A_i^{(1)})^tG_{\bar{i}}^{(1)}E^\dag.$ The main ingredient here is the local law as stated in the Theorem \ref{Thm: Local law combined}, which gives us the following estimates; 
\begin{align}\label{eq:S_2^1_Local_law_cal_1}
     &\big|\langle G_1^{(1)}A_1^{(1)}EG_i^{(1)}A_i^{(1)}E^{\dag}- (M^{(1)})_{A_1^{(1)}E}^{z_1,z_i}A_i^{(1)}E^{\dag}\rangle\big|\\
     &\prec\frac{C_{\epsilon} \|A_i^{(1)}E^{\dag}\|\|A_1^{(1)}E\|}{n \eta_*^{1i}(\eta_1 \eta_i)^{1/2} |(\hat{\beta}^*_{1i})^{(1)}|} \bigg((\eta_*^{1i})^{1/12} +\frac{(\eta_*^{1i})^{1 / 4}}{|(\hat{\beta}^*_{1i})^{(1)}|}+\frac{1}{\sqrt{n \eta_*^{1i}}}+\frac{1}{\big(|\hat{\beta}^{(1)}_*| n \eta_*^{1i}\big)^{1 / 4}}\bigg)\notag\\
     &=\mathcal{O}_\prec\bigg(\frac{1}{n |z_1-z_i|^4\eta_*^{1i}(\eta_1\eta_i)^{1/2}|\beta_i^{(1)}||\beta_1^{(1)}|}\bigg),\notag
\end{align}
where $|(\hat{\beta}^*_{1i})^{(1)}| \gtrsim |z_1-z_i|^2$  is from Lemma \ref{Lemma:bound on beta hat} and $\eta_*^{1i} = \min\{\eta_1, \eta_i\}.$
\begin{align}\label{eq:S_2^1_Local_law_cal_2}
     &\big|\langle G_1^{(1)}\mathcal{S}[G_1^{(1)}A_1^{(1)}EG_{\bar{i}}^{(1)}](A_i^{(1)})^tG_{\bar{i}}^{(1)}E^\dag- \mathcal{S}[ (M^{(1)})_{A_1^{(1)}E}^{z_1,\bar{z_i}}](A_i^{(1)})^t(M^{(1)})^{\bar{z_i},z_1}_{E^{\dag}} \rangle\big|\notag\\
     =&\big|\langle\mathcal{S}[G_1^{(1)}A_1^{(1)}EG_{\bar{i}}^{(1)}](A_i^{(1)})^tG_{\bar{i}}^{(1)}E^\dag G_1^{(1)}- \mathcal{S}[ (M^{(1)})_{A_1^{(1)}E}^{z_1,\bar{z_i}}](A_i^{(1)})^t(M^{(1)})^{\bar{z_i},z_1}_{E^{\dag}} \rangle\big|\notag\\
     \prec&\bigg(\frac{C_{\epsilon} \|A_1^{(1)}E\| \|A_i^{(1)}E^{\dag}\|}{n \eta_*^{1i}(\eta_1 \eta_i)^{1/2} |(\hat{\beta}^*_{1i})^{(1)}|} \Big((\eta_*^{1i})^{1/2} +\frac{(\eta_*^{1i})^{1 / 4}}{|(\hat{\beta}^*_{1i})^{(1)}|}+\frac{1}{\sqrt{n \eta_*^{1i}}}+\frac{1}{\big(|(\hat{\beta}^*_{1i})^{(1)}| n \eta_*^{1i}\big)^{1 / 4}}\Big)\bigg)^2\notag\\
     &\;\;=\mathcal{O}_\prec\bigg(\frac{1}{n^2 |z_{1}-z_{i}|^{4}(\eta_*^{1i})^2\eta_1\eta_i|\beta_i^{(1)}||\beta_1^{(1)}|}\bigg).
\end{align}
Now, for the error terms in \eqref{eq:S^{(1)}_2}, we shall use the bound 
\begin{align}\label{eq:error_bound}
    & \mathbb{E}|\langle\underline{ G_1^{(1)}A_1^{(1)}EG_i^{(1)} A_i^{(1)}W^{(1)}G_i^{(1)}E^\dag}\rangle|^2+\mathbb{E}|\langle\underline{G_1^{(1)}A_1^{(1)}EG_{\bar{i}}^{(1)}W^{(1)}(A_i^{(1)})^tG_{\bar{i}}^{(1)}E^\dag}\rangle|^2\\
        &\;\;\;\; \lesssim \Big(\frac{1}{n \eta_1\eta_i\eta_{*}^{1i}|\beta_1^{(1)}||\beta_i^{(1)}|}\Big)^2,\notag
\end{align}
which we shall prove at the end of this section in Corollary  \ref{cor:overline_cor}. In addition, using the \eqref{eq:G^t-M^t}, $|\xi_i^{(t)}| \lesssim \frac{1}{n},$ and the local law Theorem \ref{Thm: Local_Law}, we obtain that
\begin{align}\label{eqn: bound on WGA underline}
    &|\Upsilon_{i}|\notag\\
    &=|c\langle -\underline{W^{(1)}G_i^{(1)}}A_i^{(1)}-\xi_i^{(1)}\rangle +d\langle -\underline{W^{(2)}G_i^{(2)}}A_i^{(2)}-\xi_i^{(2)}\rangle|\notag\\
    &=|c\langle G_{i}^{(1)} - M_{i}^{(1)}-\xi_i^{(1)}\rangle +d\langle G_{i}^{(2)} - M_{i}^{(2)}-\xi_i^{(2)}\rangle| + \mathcal{O}_{\prec}\Big(\frac{1}{|\beta_{i}^{*}|(n\eta_{i})^{2}}\Big)\notag\\
    &=\mathcal{O}_{\prec}\Big(\frac{1}{\sqrt{n\eta_{i}}}\Big)+\mathcal{O}_{\prec}\Big(\frac{1}{|\beta_{i}^{*}|(n\eta_{i})^{2}}\Big)=\mathcal{O}_{\prec}\Big(\frac{1}{|\beta_{i}^{*}|\sqrt{n\eta_{i}}}\Big),
\end{align}
where $\beta_{i}^{*}:=\min\{\beta_{i}^{(1)}, \beta_{i}^{(2)}\}.$

Thus, by using \eqref{eq:error_bound} and\eqref{eqn: bound on WGA underline}, we have 
\begin{align}\label{eq:error_bound_final}
    &\Big|n^{-2}\mathbb{E}\big[\langle \underline{ G_1^{(1)}A_1^{(1)}EG_i^{(1)} A_i^{(1)}W^{(1)}G_i^{(1)}E^\dag} \rangle+\langle\underline{G_1^{(1)}A_1^{(1)}EG_{\bar{i}}^{(1)}W^{(1)}(A_i^{(1)})^tG_{\bar{i}}^{(1)}E^\dag}\rangle\big]
    \prod_{j\neq1,i}\Upsilon_{j}\Big|\\
        & \prec \frac{1}{n^2}\bigg[\prod_{j\neq1,i}\frac{1}{|\beta_{j}^{*}|\sqrt{n\eta_{j}}}\bigg]\Big(\big( \mathbb{E}|\langle\underline{ G_1^{(1)}A_1^{(1)}EG_i^{(1)} A_i^{(1)}W^{(1)}G_i^{(1)}E^\dag}\rangle|^2\big)^{1/2}\notag\\
        &\quad\quad+\big(\mathbb{E}| \langle\underline{G_1^{(1)}A_1^{(1)}EG_{\bar{i}}^{(1)}W^{(1)}(A_i^{(1)})^tG_{\bar{i}}^{(1)}E^\dag}\rangle|^2\big)^{1/2}\Big)\notag\\
        & \prec \frac{1}{n^2}\bigg[\prod_{j\neq1,i}\frac{1}{|\beta_{j}^{*}|\sqrt{n\eta_{j}}}\bigg]\Big(\frac{1}{n\eta_1\eta_i\eta_*^{1i}|\beta_{1}^{*}||\beta_{i}^{*}|}\Big)\notag\\
        &= \frac{1}{n^2}\prod_{j} \frac{|\beta_{1}^{*}|\sqrt{n\eta_{1}}|\beta_{i}^{*}|\sqrt{n\eta_{i}}}{|\beta_{j}^{*}|\sqrt{n\eta_{j}}}\Big(\frac{1}{n\eta_1\eta_i\eta_*^{1i}|\beta_{1}^{*}||\beta_{i}^{*}|}\Big)\notag\\
        &= \frac{1}{n^2\eta_{*}^{1i}\sqrt{\eta_1\eta_i}}\prod_{j}\frac{1}{|\beta_{j}^{*}|\sqrt{n\eta_{j}}}\notag\\
         &\prec \frac{1}{(n\eta_{*}^{1i})^2}\prod_{j}\frac{1}{|\beta_{j}^{*}|\sqrt{n\eta_{j}}}\notag\\
        & = \frac{\psi}{(n \eta_{*}^{1i})^2},\notag
\end{align}
where $\psi$ is defined as follows;
\begin{align}\label{eqn: definition of psi error}
    \psi:=\prod_{j}\frac{1}{|\beta_{j}^{*}|\sqrt{n\eta_{j}}}.
\end{align}

Now, in \eqref{eq:S^{(1)}_2_ver1}, using the estimates \eqref{eq:S_2^1_Local_law_cal_1}, \eqref{eq:S_2^1_Local_law_cal_2}, \eqref{eq:error_bound_final} along with the notations defined in \eqref{eq:V_1i hat}, we obtain the expression of $n^{-2}S_{2}^{(1)}$ as follows;
\begin{align}\label{eq:S^{(1)}_2}
    &\frac{1}{n^{2}}S^{(1)}_2\\
        =&\frac{1}{n^{2}}\sum_{i\neq1}\mathbb{E} \bigg[\Big\{V_{1i}^{(1)}(z_i,z_j,\eta_i,\eta_j)+V_{1i}^{(1)}(z_i,\bar{z_j},\eta_i,\eta_j)\notag\\
        &+ \mathcal{O}_\prec \Big(\frac{1}{n |z_1-z_i|^4\eta_*^{1i}(\eta_1\eta_i)^{1/2}|\beta_i^{(1)}||\beta_1^{(1)}|}+\frac{1}{n^2 |z_1-z_i|^4(\eta_*^{1i})^2\eta_1\eta_i|\beta_i^{(1)}||\beta_1^{(1)}|}\Big)\Big\}\notag\\
        &\prod_{j\neq1,i}\Upsilon_{j}\bigg]- \frac{1}{n^{2}}\sum_{i\neq1}\mathbb{E}\bigg[\Big\{\langle \underline{G_1^{(1)}A_1^{(1)}EG_i^{(1)}A_i^{(1)}W^{(1)}G_i^{(1)}E^\dag}\notag\\
        &+\underline{G_1^{(1)}A_1^{(1)}EG_{\bar{i}}^{(1)}W^{(1)}(A_i^{(1)})^tG_{\bar{i}}^{(1)}E^\dag}\rangle \Big\}\prod_{j\neq1,i}\Upsilon_{j}\bigg]\notag\\
         =&\frac{1}{n^{2}}\sum_{i \neq 1}\Big\{\widehat{V_{1i}^{(1)}}+ \mathcal{O}_\prec \Big(\frac{1}{n |z_1-z_i|^4\eta_*^{1i}(\eta_1\eta_i)^{1/2}|\beta_i^{(1)}||\beta_1^{(1)}|}\notag\\
         &+ \frac{1}{n^2 |z_1-z_i|^4(\eta_*^{1i})^2\eta_1\eta_i|\beta_i^{(1)}||\beta_1^{(1)}|} \Big) \Big\}\mathbb{E}\Big[\prod_{j\neq1,i}\Upsilon_{j}\Big]+\mathcal{O}_\prec \bigg(\frac{\psi}{(n\eta_*)^2}\bigg)\notag\\ 
       =&\sum_{i \neq 1}\bigg\{\frac{\widehat{V_{1i}^{(1)}}}{n^{2}}\mathbb{E} \Big[\prod_{j\neq1,i} \Upsilon_{j}\Big]+ \mathcal{O}_\prec \bigg(\psi\Big(\frac{1}{(n\eta_*)^2}+\frac{1}{n^2 \eta_*^{1i}|z_1-z_i|^4}
       + \frac{1}{(n\eta_*^{1i})^3 |z_1-z_i|^4} \Big)\bigg)\bigg\}\notag\\    
       =&:\frac{1}{n^2}\sum_{i \neq 1}\widehat{V_{1i}^{(1)}}\mathbb{E}\Big[\prod_{j\neq1,i}\Upsilon_{j}\Big]+error_1\notag,
\end{align}
where $\widehat{V_{1i}^{(1)}}$ is defined below. Similarly, we can evaluate $n^{-2}S_2^{(2)}$ as follows;

\begin{align}\label{eq:S^{(2)}_2}
        &\frac{1}{n^{2}}\sum_{i\neq1}\mathbb{E} \bigg[\langle G_1^{(2)}A_1^{(2)}E\big(G_i^{(2)}A_i^{(2)}+ (A_i^{(2)})^t(G_i^{(2)})^t \big)E^{\dag}\\
        &\quad-G_1^{(2)}A_1^{(2)}E\big(\underline{G_i^{(2)}A_i^{(2)}W^{(2)}G_i^{(2)}}+\underline{(G_i^{(2)})^tW^{(2)}(A_i^{(2)})^t(G_i^{(2)})^t}\big)E^{\dag}\rangle \prod_{j\neq1,i}\Upsilon_{j}\bigg]\notag\\
        &=\sum_{i \neq 1}\bigg\{\frac{\widehat{V_{1i}^{(2)}}}{n^{2}}\mathbb{E} \Big[\prod_{j\neq1,i} \Upsilon_{j}\Big]+ \mathcal{O}_\prec \bigg(\psi\Big(\frac{1}{(n\eta_*)^2}+\frac{1}{n^2 \eta_*^{1i}|z_1-z_i|^4}
        + \frac{1}{(n\eta_*^{1i})^3 |z_1-z_i|^4} \Big)\bigg)\bigg\}\notag\\    
        &=\frac{1}{n^2}\sum_{i \neq 1}\widehat{V_{1i}^{(2)}}\mathbb{E}\Big[\prod_{j\neq1,i}\Upsilon_{j}\Big]+error_1\notag,
\end{align}
where $\widehat{V_{1i}^{(t)}}(z_i,z_j,\eta_i,\eta_j)=V_{1i}^{(t)}(z_i,z_j,\eta_i,\eta_j)+V_{1i}^{(t)}(z_i,\bar{z_j},\eta_i,\eta_j)$ with

    \begin{align}\label{eq:V_1i hat}
        V_{1i}^{(t)}&:=\langle (M^{(t)})_{A_1^{(t)}E}^{z_1,z_i}A_i^{(t)}E^{\dag} \rangle +\langle (M^{(t)})_{A_1^{(t)}E}^{z_1,\bar{z_i}}(A_i^{(t)})^tE^{\dag} \rangle\\
        &\;\;\;\;\;+\langle \mathcal{S} [(M^{(t)})_{A_1^{(t)}E}^{z_1,z_i}A_i^{(t)}](M^{(t)})^{z_i,z_1}_{E^{\dag}}\rangle+\langle \mathcal{S}[ (M^{(t)})_{A_1^{(t)}E}^{z_1,\bar{z_i}}(A_i^{(t)})^t](M^{(t)})^{\bar{z_i},z_1}_{E^{\dag}}\rangle\notag\\
        &=\frac{2m_1^{(t)}m_i^{(t)}\big(2u_1^{(t)}u_i^{(t)}\Re{z_1\bar{z_i}}+(u_1^{(t)}u_i^{(t)}|z_1||z_i|)^2(r_1^{(t)}r_i^{(t)}-4)\big)}{\beta_1^{(t)}\beta_i^{(t)}[1+(u_1^{(t)}u_i^{(t)}||z_1|z_i|)^2-(m_1^{(t)}m_i^{(t)})^2-2u_1^{(t)}u_i^{(t)}\Re{z_1\bar{z_i}}]^2}\notag\\
        &\;\;\;\;\;+\frac{2m_1^{(t)}m_i^{(t)}((m_1^{(t)})^2+(u_1^{(t)})^2|z_1|^2)((m_i^{(t)})^2+(u_i^{(t)})^2|z_i|^2)}{\beta_1^{(t)}\beta_i^{(t)}[1+(u_1^{(t)}u_i^{(t)}||z_1|z_i|)^2-(m_1^{(t)}m_i^{(t)})^2-2u_1^{(t)}u_i^{(t)}\Re{z_1\bar{z_i}}]^2},\notag
    \end{align}

$\text{where } r_i^{(t)} = (m_i^{(t)})^2-(u_i^{(t)})^2|z_i|^2,$ $\beta_i^{(t)}=1-(m_i^{(t)})^2-(u_i^{(t)})^2|z_i|^2.$
\begin{rem}\label{remark: real vs complex case}
    If the matrix $X^{(t)}$ from Condition \ref{Condtion::} is considered to be complex-valued, then $\gamma_1=0$. In this case, the term $T_2^{(1)}$ does not contribute to $I_3^{(1)}$ in \eqref{eq:I_1^{(3)}_final}. Consequently, the terms $(A_i^{(1)})^t(G_i^{(1)})^t$ and $\underline{(G_i^{(1)})^tW^{(1)}(A_i^{(1)})^t(G_i^{(1)})^t}$ won't appear in the expression of $S_{2}^{(1)}$ and $S_{2}^{(2)}$ in \eqref{eq:S^{(1)}_2_ver1}. As a result, $V_{1i}^{(1)}(z_i,\bar{z_j},\eta_i,\eta_j)$ and $V_{1i}^{(2)}(z_i,\bar{z_j},\eta_i,\eta_j)$ do not appear in the expressions for $n^{-2}S_2^{(1)}$ and $n^{-2}S_2^{(2)}$ in \eqref{eq:S^{(1)}_2} and \eqref{eq:S^{(2)}_2}, respectively. Therefore, in this case, we have $\widehat{V_{1i}^{(t)}}(z_i,z_j,\eta_i,\eta_j)=V_{1i}^{(t)}(z_i,z_j,\eta_i,\eta_j)$.
\end{rem}
\subsubsection{Evaluation of $S_{i}^{(t)}$s, when $i=1, 3, 4, 5, 6$}\label{subsec: Evaluation of other S_i}
Let us first start with $S_3^{(1)}$ and $S_3^{(2)}.$ We shall follow the local law stated in Theorem \ref{Thm: Local_Law}

\begin{align}
    |\langle (G_1^{(t)}-M_1^{(t)})A \rangle| \prec \frac{C_{\epsilon}\|A\|}{n \eta_1} ,
\end{align}
which gives us the following estimate;

\begin{align}\label{eq:mixed_terms_approx}
      &\langle G_1^{(t)}A_1^{(t)}EG_i^{(s)}A_i^{(s)}E^{\dag}\rangle\\
        &=M_1^{(t)}A_1^{(t)}EM_i^{(s)}A_i^{(s)}E^{\dag} 
        + M_1^{(t)}A_1^{(t)}E (G_i^{(s)} - M_i^{(s)})A_i^{(s)}E^{\dag}\notag\\
        &\;\;+ (G_1^{(t)} - M_1^{(t)})A_1^{(t)}E M_i^{(s)}A_i^{(s)}E^{\dag}
        + (G_1^{(t)} - M_1^{(t)})A_1^{(t)}E (G_i^{(s)} - M_i^{(s)})A_i^{(s)}E^{\dag}\notag\\
        &=M_1^{(t)}A_1^{(t)}EM_i^{(s)}A_i^{(s)}E^{\dag} +\mathcal{O}_\prec\Big(\frac{1}{|\beta_1^*||\beta_i^*|n \eta_*^{1i}}\Big).\notag
\end{align}

Now, we shall use this estimate to evaluate $S_3^{(1)}$ and $S_3^{(2)}$ as follows;
\begin{align}\label{S_3^{(1)}+S_3^{(2)}}
    &\frac{1}{n^{2}}(S_3^{(1)}+S_3^{(2)})\\
         =& \frac{1}{n^{2}}\sum_{i\neq1}\mathbb{E}\bigg[\big(\langle G_1^{(1)}A_1^{(1)}EG_i^{(2)}A_i^{(2)}E^{\dag}\rangle +\langle G_1^{(2)}A_1^{(2)}EG_i^{(1)}A_i^{(1)}E^{\dag}\rangle \big)\prod_{j\neq1,i}\Upsilon_{j}\bigg]\notag\\
          =&\frac{1}{n^{2}}\sum_{i\neq1}\mathbb{E}\bigg[\bigg(\langle M_1^{(1)}A_1^{(1)}EM_i^{(2)}A_i^{(2)}E^{\dag}\rangle +\langle M_1^{(2)}A_1^{(2)}EM_i^{(1)}A_i^{(1)}E^{\dag}\rangle\notag\\
          &+\mathcal{O}_\prec\Big(\frac{1}{|\beta_1^*||\beta_i^*|n \eta_*^{1i}}\Big)\bigg)\prod_{j\neq1,i}\Upsilon_{j}\bigg]\notag\\
            =&\sum_{i\neq1}\Bigg\{\frac{1}{n^{2}}\mathbb{E}\bigg[\Big(\langle M_1^{(1)}A_1^{(1)}E_1M_i^{(2)}A_i^{(2)}E_2\rangle+\langle M_1^{(1)}A_1^{(1)}E_2M_i^{(2)}A_i^{(2)}E_1\rangle \notag\\
           &+\langle M_1^{(2)}A_1^{(2)}E_1M_i^{(1)}A_i^{(1)}E_2\rangle+\langle M_1^{(2)}A_1^{(2)}E_2M_i^{(1)}A_i^{(1)}E_1\rangle\Big)\prod_{j\neq1,i}\Upsilon_{j}\bigg]\notag\\
           &+\mathcal{O}_\prec\bigg(\frac{\psi |\beta_1^*||\beta_i^*| \sqrt{\eta_1\eta_i}}{n^2 \eta_*^{1i}|\beta_1^*||\beta_i^*|}\bigg)\Bigg\}\notag\\
           =&\sum_{i\neq1}\Bigg\{\frac{1}{n^{2}}\bigg(\frac{1}{\beta_1^{(1)} \beta_i^{(2)}} \big(4m_i^{(2)}m_1^{(1)} u_i^{(2)}u_1^{(1)}z_i\overline{z_1}\big)+\frac{1}{\beta_1^{(1)} \beta_i^{(2)}}\big(4m_i^{(2)}m_1^{(1)}u_i^{(2)}u_1^{(1)}\overline{z_i}z_1\big)\notag\\
           &+\frac{1}{\beta_1^{(2)} \beta_i^{(1)}} \big(4m_i^{(1)}m_1^{(2)} u_i^{(1)}u_1^{(2)}z_i\overline{z_1}\big)+\frac{1}{\beta_1^{(2)} \beta_i^{(1)}}\big(4m_i^{(1)}m_1^{(2)}u_i^{(1)}u_1^{(2)}\overline{z_i}z_1\big)\bigg)\notag\\
           &\;\;\mathbb{E}\Big[\prod_{j\neq1,i}\Upsilon_{j}\Big]+\mathcal{O}_\prec\bigg(\frac{\psi \sqrt{\eta_1\eta_i}}{n^2 \eta_*^{1i}}\bigg)\Bigg\}\notag\\
           =&:\frac{1}{n^{2}}\sum_{i\neq 1}\Bigg\{\bigg(\frac{8m_i^{(2)}m_1^{(1)} u_1^{(1)}u_i^{(2)}\Re{(\overline{z_1}z_i)}}{\beta_i^{(2)}\beta_1^{(1)}}+\frac{8m_i^{(1)}m_1^{(2)} u_i^{(1)}u_1^{(2)}\Re{(\overline{z_1}z_i)}}{\beta_i^{(1)}\beta_1^{(2)}}\bigg)\notag\\
           &\;\;\mathbb{E}\Big[\prod_{j\neq1,i}\Upsilon_{j}\Big]\Bigg\}\notag+error_2.\notag
\end{align}
Similarly, we can evaluate $S_4^{(1)}$ and $S_4^{(2)}$ as follows; 
\begin{align}\label{S_4^{(1)}+S_4^{(2)}}
    & \frac{1}{n^{2}}(S_4^{(1)}+S_4^{(2)})\\
         =& \sum_{i\neq1}\Bigg\{\frac{1}{n^{2}}\bigg(\frac{(1-\beta_1^{(1)})}{\beta_1^{(1)} \beta_i^{(2)}}\big((m_i^{(2)})^2+(u_i^{(2)})^2|z_i|^2\big) \notag\\
         &+\frac{(1-\beta_1^{(1)})}{\beta_i^{(2)}\beta_1^{(1)}}\big((m_i^{(2)})^2+(u_i^{(2)})^2|z_i|^2\big)+\frac{(1-\beta_1^{(2)})}{\beta_1^{(2)} \beta_i^{(1)}}\big((m_i^{(1)})^2+(u_i^{(1)})^2|z_i|^2\big)
         \notag\\
         &+\frac{(1-\beta_1^{(2)})}{\beta_i^{(1)}\beta_1^{(2)}}\big((m_i^{(1)})^2+(u_i^{(1)})^2|z_i|^2 \big)\bigg)\mathbb{E}\Big[\prod_{j\neq1,i}\Upsilon_{j}\Big]+\mathcal{O}_\prec\bigg(\frac{\psi \sqrt{\eta_1\eta_i}}{n^2 \eta_*^{1i}}\bigg)\Bigg\}\notag\\
         =& \sum_{i\neq1}\Bigg\{\frac{1}{n^{2}}\bigg(\frac{2(1-\beta_1^{(1)})(1-\beta_i^{(2)})}{\beta_1^{(1)} \beta_i^{(2)}}
         +\frac{2(1-\beta_1^{(2)})(1-\beta_i^{(1)})}{\beta_1^{(2)} \beta_i^{(1)}}\bigg)
        \mathbb{E}\Big[\prod_{j\neq1,i}\Upsilon_{j}\Big]\notag\\
         &+\mathcal{O}_\prec\bigg(\frac{\psi \sqrt{\eta_1\eta_i}}{n^2 \eta_*^{1i}}\bigg)\Bigg\}\notag\\
         =& \frac{1}{n^{2}}\sum_{i\neq1}(-2)\big\{ \partial_{\eta_1}(m_1^{(1)}) \partial_{\eta_i}(m_i^{(2)})
         + \partial_{\eta_1}(m_1^{(2)}) \partial_{\eta_i}(m_i^{(1)})\big\}\mathbb{E}\Big[\prod_{j\neq1,i}\Upsilon_{j}\Big]+error_2\notag,
\end{align}
where in the last equality, we have used $(1-\beta^{(t)})/\beta^{(t)}= -\iota \partial_\eta(m^{(t)})$ from \eqref{eq:m^t_derivative}.
Now, let us move to evaluate the terms $S_1^{(1)}, S_1^{(2)}.$ As per the definition,
\begin{align*}
    S_1^{(1)}&={\sum_{ab}}'\mathbb{E}\bigg[(G_1^{(1)})_{ab}(G_1^{(1)}A_1^{(1)})_{ab}\prod_{i\neq1} \Upsilon_{i} \bigg]\\
        &= 2n\mathbb{E}\bigg[\langle G_1^{(1)}A_1^{(1)}E(G_1^{(1)})^tE^{\dag}\rangle \prod_{i\neq1} \Upsilon_{i} \bigg].\\
\end{align*}

We observe that $\langle G_1^{(1)}A_1^{(1)}E(G_1^{(1)})^tE^{\dag}\rangle$ is same as the expression in the second equality in \eqref{eq:F^{(t)}}. Therefore, using \eqref{eq:F^{(t)}}, we obtain
\begin{align}\label{eq:S_1^{(1)}}
    S_1^{(1)}\notag\\
    =&{\sum_{ab}}'\mathbb{E}\bigg[(G_1^{(1)})_{ab}(G_1^{(1)}A_1^{(1)})_{ab}\prod_{i\neq1} \Upsilon_{i} \bigg]\notag\\
        =&\bigg\{\frac{in}{2}\partial_{\eta_1}\log\big(1-(u_1^{(1)})^2+2(u_1^{(1)})^3|z_1|^2-(u_1^{(1)})^2(z_1^2+\bar{z_1}^2)\big)\notag\\
        &+ \mathcal{O}_\prec \Big(\frac{1}{|z_1-\bar{z}_1|^2}\frac{1}{|\beta_1^{(1)}|(n\eta_1)^2}\Big) \bigg\}\mathbb{E}\Big[\prod_{i\neq1} \Upsilon_{i} \Big]\notag\\
         =&\bigg\{\frac{in}{2}\partial_{\eta_1}\log\big(1-(u_1^{(1)})^2+2(u_1^{(1)})^3|z_1|^2-(u_1^{(1)})^2(z_1^2+\bar{z_1}^2)\big)\bigg\}\mathbb{E}\Big[\prod_{i\neq1} \Upsilon_{i} \Big]\notag\\
         &+\mathcal{O}_\prec\bigg(\frac{\psi}{|\Im{z_1}|(n\eta_1)^{3/2}}\bigg)\notag\\
        =&:\frac{in}{2}\partial_{\eta_1}\log\big(1-(u_1^{(1)})^2+2(u_1^{(1)})^3|z_1|^2-(u^{(1)})^2(z_1^2+\bar{z_1}^2)\big)\mathbb{E}\Big[\prod_{i\neq1} \Upsilon_{i}\Big]\\
        &+ error_3,\notag
    \end{align}
    
Similarly, we can evaluate $S_1^{(2)}$ as well.
Now, for all the other terms in \eqref{eq:I^{(1)}+I^{(2)}Noataions}, namely $S_5^{(1)},S_5^{(2)},S_6^{(1)},\text{ and }S_6^{(2)}$, we shall use the following bound from Lemma \ref{lemma:overlinebound};

\begin{align}\label{eq:bound for overline terms}
        &\mathbb{E}|\langle \underline{G^{(t)}_1A^{(t)}_1EG_i^{(s)}A_i^{(s)}W^{(s)}G_i^{(s)}E^{\dag}} \rangle|^2\\
        &=\mathcal{O}_{\prec}\bigg(\frac{1}{|\beta_1^{(t)}|^{2}|\beta_i^{(s)}|^{2} n^2\eta_i^4}\bigg),\text{  for } 1\leq t\neq s\leq 2.\notag
    \end{align}
    
Now, using \eqref{eq:bound for overline terms}, similar to the \eqref{eq:error_bound_final}, we have
\begin{align}\label{eq:S_5+S_6}
    &\frac{1}{n^2}\big(S_5^{(1)}+S_5^{(2)}+S_6^{(1)}+S_6^{(2)}\big)\\
    &= \mathcal{O}_\prec \Big( \frac{\psi \sqrt{\eta_1\eta_i}}{(n\eta_i)^2}+\frac{\psi \sqrt{\eta_1\eta_i}}{(n\eta_1)^2}\Big)\notag\\
    &=\mathcal{O}_\prec\Big(\frac{\psi\sqrt{\eta_1\eta_i}}{(n\eta_*^{1i})^2}\Big)=:error_4.\notag
\end{align}
Thus, from \eqref{eq:I^{(1)}+I^{(2)}Noataions}, \eqref{eq:S^{(1)}_2}, \eqref{eq:S^{(2)}_2}, \eqref{S_3^{(1)}+S_3^{(2)}}, \eqref{S_4^{(1)}+S_4^{(2)}}, \eqref{eq:S_1^{(1)}},\eqref{eq:S_5+S_6}, and using the expressions of $\xi_1^{(1)}, \xi_1^{(2)}$ from Lemma \ref{Lemma:Exp or resolvent} we have
\begin{align}\label{eq:I_1+I_2_final}
    &\mathbb{E}\bigg[\prod_{i\in[p]}c\langle -\underline{W^{(1)}G_i^{(1)}}A_i^{(1)}-\xi_i^{(1)}\rangle +d\langle -\underline{W^{(2)}G_i^{(2)}}A_i^{(2)}-\xi_i^{(2)}\rangle\bigg]\notag\\
        &=cT_5^{(1)}+dT_5^{(2)}-c\langle\xi_1^{(1)}\rangle\mathbb{E}\Big[\prod_{i\neq1}\Upsilon_{i}\Big]-d\langle\xi_1^{(2)}\rangle\mathbb{E} \bigg[\prod_{i\neq1} \Upsilon_{i}\bigg]
        +\frac{c}{2n^2}S^{(1)}_1\notag\\
        &\;\;\;\;\; +\frac{d}{2n^2}S^{(2)}_1
        +\frac{c^2}{2n^2}S^{(1)}_2
        +\frac{d^2}{2n^2}S^{(2)}_2+ \frac{cd\gamma}{2n^2}\big(S^{(1)}_3+S^{(2)}_3+S^{(1)}_5
        +S^{(2)}_5\big)
        \notag\\
        &\;\;\;\;\;+\frac{cd\rho}{2n^2}\big(S^{(1)}_4 +S^{(2)}_4+S^{(1)}_6+S^{(2)}_6\big)\notag\\
        &=cT_5^{(1)}+dT_5^{(2)}-c\langle\frac{i}{4n}\partial_{\eta_1}\big(1-(u_1^{(1)})^2+2(u_1^{(1)})^3|z_1|^2\notag\\
        &\;\;\;\;\;-(u_1^{(1)})^2(z_1^2+\bar{z}_1^2)\big) \rangle\mathbb{E}\Big[\prod_{i\neq1}\Upsilon_{i}\Big]
        +c\langle\frac{i (\kappa_4)_{11}}{4n} \partial_{\eta_1}(m_1^{(1)})^4
   \rangle\mathbb{E}\Big[\prod_{i\neq1}\Upsilon_{i}\Big]\notag\\
   &\;\;\;\;\;-d\langle\frac{i}{4n}\partial_{\eta_1}\big(1-(u_1^{(2)})^2+2(u_1^{(2)})^3|z_1|^2-(u_1^{(2)})^2(z_1^2+\bar{z}_1^2)\big) \rangle\mathbb{E}\Big[\prod_{i\neq1}\Upsilon_{i}\Big]\notag\\
        &\;\;\;\;\;
        +d\langle\frac{i (\kappa_4)_{22}}{4n} \partial_{\eta_1}(m_1^{(2)})^4
   \rangle\mathbb{E}\Big[\prod_{i\neq1}\Upsilon_{i}\Big]\notag\\
        &\;\;\;\;\;+\frac{c}{2n^2}\frac{in}{2}\partial_{\eta}\log\big(1-(u_1^{(1)})^2+2(u_1^{(1)})^3|z_1|^2-(u^{(1)})^2(z_1^2+\bar{z_1}^2)\big)\mathbb{E}\Big[\prod_{i\neq1} \Upsilon_{i}\Big]\notag\\
        &\;\;\;\;\;+ error_3+\frac{d}{2n^2}\frac{in}{2}\partial_{\eta}\log\big(1-(u_1^{(2)})^2+2(u_1^{(2)})^3|z_1|^2-(u^{(2)})^2(z_1^2+\bar{z_1}^2)\big)\notag\\
        &\;\;\;\;\;\;\;\;\mathbb{E}\Big[\prod_{i\neq1} \Upsilon_{i}\Big]+ error_3+\sum_{i\neq 1}\frac{c^2\widehat{V_{1i}^{(1)}}}{2n^2}\mathbb{E}\Big[\prod_{j\neq1,i}\Upsilon_{j}\Big]\notag\\
        &\;\;\;\;\;+\sum_{i\neq 1}\frac{d^2\widehat{V_{1i}^{(2)}}}{2n^2}\mathbb{E}\Big[\prod_{j\neq1,i}\Upsilon_{j}\Big]+error_1+\sum_{i\neq 1}\Bigg\{\frac{8cd\gamma m_i^{(2)}m_1^{(1)} u_1^{(1)}u_i^{(2)}\Re{(\overline{z_1}z_i)}}{2n^2\beta_i^{(2)}\beta_1^{(1)}}\notag\\
       &\;\;\;\;\;+\frac{8cd\gamma m_i^{(1)}m_1^{(2)} u_i^{(1)}u_1^{(2)}\Re{(\overline{z_1}z_i)}}{2n^2\beta_i^{(1)}\beta_1^{(2)}}\Bigg\}\mathbb{E}\Big[\prod_{j\neq1,i}\Upsilon_{j}\Big]+error_2\notag\\
        &\;\;\;\;\;+ \frac{cd\rho}{2n^2}\sum_{i\neq1}(-2)\big\{ \partial_{\eta_1}(m_1^{(1)}) \partial_{\eta_i}(m_i^{(2)})
         + \partial_{\eta_1}(m_1^{(2)}) \partial_{\eta_i}(m_i^{(1)})\big\}\mathbb{E}\Big[\prod_{j\neq1,i}\Upsilon_{j}\Big]\notag\\
         &\;\;\;\;\;+error_2+error_4\notag\\
        &=cT_5^{(1)}+dT_5^{(2)}
        +c\frac{\iota (\kappa_4)_{11}}{4n} \partial_{\eta_1}(m_1^{(1)})^4
         \mathbb{E}\Big[\prod_{i\neq1}\Upsilon_{i}\Big]
        +d\frac{\iota (\kappa_4)_{22}}{4n} \partial_{\eta_1}(m_1^{(2)})^4
        \\
        &\;\;\;\;\;\;\;\mathbb{E}\Big[\prod_{i\neq1}\Upsilon_{i}\Big]+ \sum_{i\neq 1}\frac{c^2\widehat{V_{1i}^{(1)}}}{2n^2}\mathbb{E}\Big[\prod_{j\neq1,i}\Upsilon_{j}\Big] +\sum_{i\neq 1}\frac{d^2\widehat{V_{1i}^{(2)}}}{2n^2}\mathbb{E}\Big[\prod_{j\neq1,i}\Upsilon_{j}\Big]\notag\\
       &\;\;\;\;\;+\sum_{i\neq 1}\Bigg\{\frac{8cd\gamma m_i^{(2)}m_1^{(1)} u_1^{(1)}u_i^{(2)}\Re{(\overline{z_1}z_i)}}{2n^2\beta_i^{(2)}\beta_1^{(1)}}+\frac{8cd\gamma m_i^{(1)}m_1^{(2)} u_i^{(1)}u_1^{(2)}\Re{(\overline{z_1}z_i)}}{2n^2\beta_i^{(1)}\beta_1^{(2)}}\Bigg\}\notag\\
        &\;\;\;\;\;\;\;\mathbb{E}\Big[\prod_{j\neq1,i}\Upsilon_{j}\Big]+ \frac{cd\rho}{n^2}\sum_{i\neq1}\Big( -\partial_{\eta_1}(m_1^{(1)}) \partial_{\eta_i}(m_i^{(2)})
         - \partial_{\eta_1}(m_1^{(2)}) \partial_{\eta_i}(m_i^{(1)}) \Big)\notag\\
        &\;\;\;\;\;\;\;
         \mathbb{E}\Big[\prod_{j\neq1,i}\Upsilon_{j}\Big]+error_1+error_2+error_3+error_4\notag.
\end{align}

\subsection{Evaluation of $T_{5}^{(t)}$ terms in \eqref{eq:I^{(1)}+I^{(2)}Noataions}}

We evaluate $T_5^{(1)}$ on the RHS of \eqref{eq:I_1^{(3)}_final} for $k =2$ and $k\geq 3$ separately. Recall $T_5^{(1)}$ from \eqref{eq:I_1^{(3)}_final}
\begin{align}\label{eq:RHSofeq:cummulant_exp_I_1^{(3)}_1}
    &\sum_{ab}\sum_{k\geq 2}\sum_{\alpha\in\{(ab)_1,(ba)_1,(ab)_2,(ba)_2\}^k}\frac{\kappa((ba)_1,\alpha)}{k!}\mathbb{E}\bigg[\partial_\alpha\Big(\langle-\Delta^{ba}G_1^{(1)}A_1^{(1)}\rangle\prod_{i\neq1}\Upsilon_{i}\Big)\bigg].
\end{align}
Now, let us consider the case when $k=2$. In that case, letting $\alpha=(\alpha_{1}, \alpha_{2}),$ we have the following;

\begin{align}\label{eqn: k=2 derivative written as H sum}
    &\partial_{\alpha_1, \alpha_2}\Big(\langle-\Delta^{ba}G_1^{(1)}A_1^{(1)}\rangle \prod_{i\neq1}\Upsilon_{i}\Big)\\
    &=\big(\langle-\Delta^{ba}G_1^{(1)}A_1^{(1)}\rangle\big)^{''}\Big(\prod_{i\neq1}\Upsilon_{i}\Big)\notag\\
    &\;\;\;\;\;+\big(\langle-\Delta^{ba}G_1^{(1)}A_1^{(1)}\rangle\big)^{'} \Big(\prod_{i\neq1}\Upsilon_{i}\Big)^{'}\notag\\
    &\;\;\;\;\;+\langle-\Delta^{ba}G_1^{(1)}A_1^{(1)}\rangle \Big(\prod_{i\neq1} \Upsilon_{i}\Big)^{''}\notag\\
    &=:H_1+H_2+H_3.\notag
\end{align}

Recall from page \pageref{eqn: WGA in terms of S_1, S_2} that $\partial_{\alpha_{i}}$ denotes the derivative with respect to $w_{\alpha_{i}}^{(t)}.$ In general, we use the shorthand notation $\alpha_{i}$ to denote $w_{\alpha_{i}}^{(t)}.$ The same convention is used in the above, where $\partial_{\alpha_{i}}f$ is denoted by $f',$ $\partial_{\alpha_{i}}\partial_{\alpha_{j}}f$ is denoted by $f''$ and so on.

Notice that $ H_1 $ is non-zero only if both $\alpha_{1}, \alpha_{2}$, belong to $ \{(ab)_1, (ba)_1\} $. In that case, we have the following expression;
\begin{align}\label{eq:H_1}
    H_1 &= \big(\langle-\Delta^{ba}G_1^{(1)}A_1^{(1)}\rangle\big)^{''}\prod_{i\neq1}\Upsilon_{i}\\
    &= \big(\partial_{\alpha_1,\alpha_2}\langle-\Delta^{ba}G_1^{(1)}A_1^{(1)}\rangle\big) \prod_{i\neq1}\Upsilon_{i}\notag\\
    &=\langle-\Delta^{ba}G_1^{(1)}(\Delta G_{1}^{(1)})^{2}A_1^{(1)}\rangle \prod_{i\neq1}\Upsilon_{i}\notag,
\end{align}
where $(\Delta G_{1}^{(1)})^{k}=\Delta G_{1}^{(1)}\Delta G_{1}^{(1)}\cdots \Delta G_{1}^{(1)},$ and each $\Delta$ is either $\Delta^{ab}$ or $\Delta^{ba}$ depending on the context.

Similarly, $H_2$ is non-zero when at least one of $ \alpha_1 $ and $ \alpha_2 $ does not belong to $ \{(ab)_2, (ba)_2\} .$ Thus, taking derivatives with respect appropriate variables, we have
\begin{align}\label{eq:H_2}
    &H_2= \big(\langle-\Delta^{ba}G_1^{(1)}A_1^{(1)}\rangle\big)^{'} \Big(\prod_{i\neq1}\Upsilon_{i}\Big)^{'}\notag\\
    &= \langle\Delta^{ba}G_1^{(1)}\Delta G_1^{(1)}A_1^{(1)}\rangle \sum_{j =2}^p\left(c\langle -\underline{W^{(1)} G_j^{(1)}\Delta G_j^{(1)}}A_j^{(1)} + \Delta G_j^{(1)}A_j^{(1)} \rangle  \right)  \prod_{i\neq 1,j}\Upsilon_i \\
    &\;\;\;\;\;+ \langle\Delta^{ba}G_1^{(1)}\Delta G_{1}^{(1)}A_1^{(1)}\rangle \sum_{j =2}^p\left(d\langle  \underline{-W^{(2)}G_j^{(2)}\Delta G_j^{(2)}}A_j^{(2)} + \Delta G_j^{(2)}A_j^{(2)} \rangle  \right)  \prod_{i\neq1,j}\Upsilon_i \notag.
\end{align}

Finally, no particular choice of $\alpha_{1}, \alpha_{2}$ completely nullifies $H_{3}.$ Therefore, taking all possible derivatives,
\begin{align}\label{eq:H_3}
   &H_3= \langle-\Delta^{ba}G_1^{(1)}A_1^{(1)}\rangle \Big(\prod_{i\neq1}\Upsilon_{i}\Big)^{''}\\
   &=\langle-\Delta^{ba}G_1^{(1)}A_1^{(1)}\rangle \sum_{j =2}^p\big(2c\langle \underline{W^{(1)}(G_j^{(1)}\Delta)^2G_j^{(1)}}A_j^{(1)}-\Delta (G_j^{(1)})^2A_j^{(1)}\rangle\big) \prod_{i\neq 1, j}\Upsilon_{i}\notag\\
   &\;\;\;\;\;+\langle-\Delta^{ba}G_1^{(1)}A_1^{(1)}\rangle \sum_{j =2}^p\big(2d\langle \underline{W^{(2)}(G_j^{(2)}\Delta)^2G_j^{(2)}}A_j^{(2)}-(\Delta G_j^{(2)})^2A_j^{(2)}\rangle \big)  \prod_{i\neq 1, j}\Upsilon_{i}\notag\\
   &\;\;\;\;\;+\langle-\Delta^{ba}G_1^{(1)}A_1^{(1)}\rangle \sum_{\substack{j,l = 2 \\ j \neq l}}^p
   \big(c\langle \underline{-W^{(1)}G_j^{(1)}\Delta G_j^{(1)}}A_j^{(1)})+\Delta G_j^{(1)}A_j^{(1)}\rangle \big) \notag\\
   &\;\;\;\;\;\;\;
   \big(d\langle \underline{-W^{(2)}G_l^{(2)}\Delta G_l^{(2)}}A_l^2+\Delta G_l^{(2)}A_l^2)\rangle  \big)\prod_{i\neq 1, j, l}\Upsilon_{i}\notag.
\end{align}

When we perform the $\alpha$-derivatives in \eqref{eq:RHSofeq:cummulant_exp_I_1^{(3)}_1} using the Leibniz rule, we get sum of product of normalized traces. Each product in the sum is a product of $p$ terms. Among those $p$ terms, at least one or more (say, $r$ many) are of the form $\langle \big(\Delta G_i^{(1)}\big)^{k_i} A_i^{(1)} \rangle$, $\langle \underline{W^{(1)}\big(G_i^{(1)} \Delta\big)^{k_i} G_i^{(1)}} A_i^{(1)} \rangle$, or $\langle \underline{W^{(2)}\big(G_i^{(2)} \Delta\big)^{k_i} G_i^{(2)}} A_i^{(2)}\rangle$ with $k_i \geq 0, \sum k_i=k+1.$ Rest (i.e., $p-r$ many) are of the form $\langle \underline{W^{(1)} G_i^{(1)} A_i^{(1)}}+\xi_i^{(1)} \rangle + \langle \underline{W^{(2)} G_i^{(2)} A_i^{(2)}}+\xi_i^{(2)} \rangle$. For example, in the expression of $H_{3},$ we have a term of the form $\langle\Delta^{ba}G_{1}^{(1)}A_{1}^{(1)}\rangle\langle (\Delta G_{j}^{(1)})^{2}A_{j}^{(1)}\rangle\prod_{i\neq 1, j}\Upsilon _{i}.$ For this particular term, $k=2$ and $r=2.$

In the rest of this section, we shall analyze the cases for different values of $k$ and $r$. We shall see that the significant contribution comes only from the case when $k=3.$ Prior to that, we rewrite the self-renormalized terms as

\begin{align}\label{eq:W^r}
&\langle \underline{ W^{(t)}(G^{(t)} \Delta)^k G^{(t)}}A^{(t)}\rangle \notag\\
& =\langle \underline{  G^{(t)}A^{(t)}W^{(t)}(G^{(t)} \Delta)^k}\rangle \notag\\
& =\langle\underline{G^{(t)} A^{(t)} W^{(t)} G^{(t)}} \Delta(G^{(t)} \Delta)^{k-1}\rangle+\sum_{j=1}^{k-1}\langle G^{(t)} A^{(t)} S\left[(G^{(t)} \Delta)^j G^{(t)}\right](G^{(t)} \Delta)^{k-j}\rangle \notag\\
& =\langle\underline{G^{(t)} A^{(t)} W^{(t)} G^{(t)}} \Delta(G^{(t)} \Delta)^{k-1}\rangle+\sum_{j=1}^{k-1}\langle G^{(t)} A^{(t)} E(G^{(t)} \Delta)^{k-j}\rangle\\
&\quad\quad\langle G^{(t)} E^{\dag}(G^{(t)} \Delta)^j\rangle.\notag
\end{align}

Now, we analyze the cases for different values of $k$ and $r.$

\subsubsection{Contribution of \eqref{eq:RHSofeq:cummulant_exp_I_1^{(3)}_1} to $T_{5}^{(1)}$ when $k=2, r=1$} This is essentially $H_{1}$ from \eqref{eq:H_1}. We argue that the term $\langle\Delta G_1^{(1)} \Delta G_1^{(1)} \Delta G_1^{(1)} A_1^{(1)}\rangle$ will have insignificant contribution. Since it is a product of three $\Delta^{\alpha_{i}}G_{1}^{(1)}$s, and $\alpha_{i}\in \{(ab)_{1}, (ba)_{1}\},$ at least one of the term in this product will have an off-diagonal entry of $G_{1}^{(1)}.$ For example, the term $\langle\Delta^{ba} G_1^{(1)} \Delta^{ab} G_1^{(1)} \Delta^{ba} G_1^{(1)} A_1^{(1)}\rangle$ simplifies to $(2n)^{-1}\sum' (G_{1}^{(1)})_{aa}(G_{1}^{(1)})_{bb}(G_{1}^{(1)}A_{1}^{(1)})_{ab},$ which contains an off-diagonal entry. Such a term (along with the prefactor $n^{-3/2}$) can further be estimated by proceeding in the same way as in \eqref{eq:exp_local_law}, as follows;

\begin{align}\label{eq:k=2,r=1}
&n^{-1-3 / 2} \sum_{a \leq n} \sum_{b>n} (G_1^{(1)})_{a a}  (G_1^{(1)})_{bb}(G^{(1)}_1 A^{(1)}_1)_{a b}\\
& =\frac{m^2}{n^{5 / 2}}\langle E_1 \boldsymbol{1}, G^{(1)}_1 A^{(1)}_1 E_2 \boldsymbol{1}\rangle +\mathcal{O}_{\prec}\bigg(\frac{1}{n^{1 / 2}} \frac{1}{|\beta_1^{(1)}|(n \eta_1)^{3 / 2}}\bigg)\notag \\
& =\mathcal{O}_{\prec}\bigg(\frac{1}{|\beta_1^{(1)}|n^{3 / 2}}+\frac{1}{|\beta_1^{(1)}|n^2 \eta_1^{3 / 2}}\bigg)\notag.
\end{align}

\subsubsection{Contribution of \eqref{eq:RHSofeq:cummulant_exp_I_1^{(3)}_1} to $T_{5}^{(1)}$ when $k=2, r=2$} From \eqref{eq:H_2} and \eqref{eq:H_3}, we see that this case consists of the expression of $H_{2}$ and the first two terms of $H_{3},$ which is written as follows;

\begin{align}\label{eqn: H_2 plus first two terms of H_3}
& \langle \Delta^{ba} G_1^{(1)} \Delta G_1^{(1)} A_1^{(1)}\rangle \langle\Delta G_i^{(1)} A_i^{(1)}- \underline{W^{(1)} G_i^{(1)} \Delta G_i^{(1)}} A_i^{(1)}\rangle\\
&\;\;\;+ \langle \Delta^{ba} G_1^{(1)} \Delta G_1^{(1)} A_1^{(1)}\rangle \langle\Delta G_i^{(2)} A_i^{(2)}- \underline{W^{(2)} G_i^{(2)} \Delta G_i^{(2)}} A_i^{(2)}\rangle \notag\\
&\;\;\;+\langle-\Delta^{ba} G_1^{(1)} A_1^{(1)}\rangle(2\langle-\Delta G_i^{(1)} \Delta G_i^{(1)} A_i^{(1)}+\underline{W^{(1)} G_i^{(1)} \Delta G_i^{(1)} \Delta G_i^{(1)}} A_i^{(1)}\rangle)\notag\\
&\;\;\;+\langle-\Delta^{ba} G_1^{(1)} A_1^{(1)}\rangle\big(2\langle-\Delta G_i^{(2)} \Delta G_i^{(2)} A_i^{(2)}+\underline{W^{(2)} G_i^{(2)} \Delta G_i^{(2)}\Delta G_i^{(2)}} A_i^{(2)}\rangle\big).\notag
\end{align}
In the rest of the subsection, we find the most contributing factors of the above expression. The essential tools are the isotropic local law, and accumulation of the diagonal terms from $G_{1}^{(1)}$ or $G_{1}^{(1)}A_{1}^{(1)}$ in the product, and the following estimate \eqref{eq:G^iBW^iG^ibound}
\begin{equation}\label{eq:G^iBW^iG^ibound}
    |\langle \boldsymbol{x}, \underline{G^{(t)}BW^{(t)}G^{(t)}} \boldsymbol{y}\rangle| \prec \frac{\norm{x} \norm{y}\norm{B}}{n^{1/2}\eta^{3/2}},
\end{equation}
which can be obtained by following the techniques in Lemma \ref{lemma:overlinebound}. Let's see how these principles are applied to obtain the estimates. For example, the highest contribution from the first four factors may arise from one of the following scenarios;
\begin{align}\label{eq:k=2,r=2_first four term}
& \frac{1}{n^{7 / 2}} {\sum_{ab}}'\Big((G_1^{(1)})_{a a}(G_1^{(1)} A_1^{(1)})_{b b}\big((G_i^{(1)} A_i^{(1)})_{a b}-(\underline{G_i^{(1)} A_i^{(1)} W^{(1)} G_i^{(1)}})_{a b}\big)\\
&\;\;\;\;\;+(G_1^{(1)})_{a a}(G_1^{(1)} A_1^{(1)})_{b b}\big((G_i^{(2)} A_i^{(2)})_{a b}-(\underline{G_i^{(2)} A_i^{(2)} W^{(2)} G_i^{(2)}})_{a b}\big)\Big)\notag\\
&=\mathcal{O}_\prec\bigg(\frac{1}{n^3 \eta_1 \eta_i^{3 / 2}|\beta_1^{(1)}||\beta_i^{(1)}|}\bigg).\notag
\end{align}
The above is obtained from the following estimates \eqref{eq:GaGAbbGAab} and \eqref{eq:GaaGAbbGAWGab}. The following calculations are similar to \eqref{eq:exp_local_law}.
\begin{align}\label{eq:GaGAbbGAab}
   & \frac{1}{n^{7/2}}{\sum_{ab}}'(G_1^{(1)})_{aa}(G_1^{(1)} A_1^{(1)})_{b b}(G_i^{(1)} A_i^{(1)})_{a b}\\
   =&\frac{1}{n^{7/2}}{\sum_{ab}}'\bigg\{\big(m_1^{(1)}+(G^{(1)}_1-M^{(1)}_1)_{aa}\big)\Big((M_1^{(1)} A_1^{(1)})_{b b}\notag\\
   &+\big((G^{(1)}_1-M^{(1)}_1)A^{(1)}_1\big)_{bb}\Big)\Big((M_i^{(1)} A_i^{(1)})_{ab}+\big((G_i^{(1)}-M_i^{(1)}) A_i^{(1)}\big)_{ab}\Big)\bigg\}\notag\\
     = & \frac{1}{n^{7/2}}{\sum_{ab}}'\bigg\{m_1^{(1)}(M_1^{(1)} A_1^{(1)})_{bb}
     \big((M_i^{(1)} A_i^{(1)})_{ab}+\big((G_i^{(1)}-M_i^{(1)}) A_i^{(1)}\big)_{ab}\big) \notag\\
    & + (G^{(1)}_1-M^{(1)}_1)_{aa}(M_1^{(1)} A_1^{(1)})_{bb}
    \big((M_i^{(1)} A_i^{(1)})_{ab}+\big((G_i^{(1)}-M_i^{(1)}) A_i^{(1)}\big)_{ab}\big) \notag\\
    &+ m_1^{(1)}\big( (G^{(1)}_1-M^{(1)}_1)A^{(1)}_1\big)_{bb}
    \big((M_i^{(1)} A_i^{(1)})_{ab}+\big((G_i^{(1)}-M_i^{(1)}) 
   A_i^{(1)}\big)_{ab}\big) \notag\\
    &+ \big((G^{(1)}_1-M^{(1)}_1)A^{(1)}_1\big)_{bb}(G^{(1)}_1-M^{(1)}_1)_{aa}
    \big((M_i^{(1)} A_i^{(1)})_{ab}\notag\\
    &+\big((G_i^{(1)}-M_i^{(1)}) A_i^{(1)}\big)_{ab}\big)\bigg\}\notag\\
     =& \mathcal{O}_\prec\Bigg(\frac{1}{n^{7/2}}\bigg(\frac{1}{|\beta_1^{(1)}|}\frac{n}{(1+\eta_i)|\beta_i^{(1)}|}+\frac{1}{|\beta_1^{(1)}|}\frac{n^2}{\sqrt{n\eta_i}|\beta_i^{(1)}|}+\frac{1}{|\beta_1^{(1)}||\beta_i^{(1)}|}\frac{n}{\eta_1}\notag\\
    &+\frac{1}{|\beta_1^{(1)}||\beta_i^{(1)}|}\frac{n}{\eta_1}\frac{1}{\sqrt{n\eta_i}}+\frac{1}{|\beta_1^{(1)}||\beta_i^{(1)}|}\frac{n}{\eta_1}+\frac{1}{|\beta_1^{(1)}||\beta_i^{(1)}|}\frac{n}{\eta_1}\frac{1}{\sqrt{n\eta_i}}\notag\\
    &+\frac{1}{|\beta_1^{(1)}||\beta_i^{(1)}|}\frac{n}{\eta_1}\frac{1}{\sqrt{n\eta_1}}+\frac{1}{|\beta_1^{(1)}||\beta_i^{(1)}|}\frac{n}{\eta_1}\frac{1}{\sqrt{n\eta_1}}\frac{1}{\sqrt{n\eta_i}}\bigg)\Bigg)\notag\\
    =&\mathcal{O}_\prec\bigg(\frac{1}{n^{3}\eta_1\eta_i^{1/2}|\beta_1^{(1)}||\beta_i^{(1)}|}\bigg)\notag,
\end{align}
and 
\begin{align}\label{eq:GaaGAbbGAWGab}
     & \frac{1}{n^{7/2}}{\sum_{ab}}'(G_1^{(1)})_{aa}(G_1^{(1)} A_1^{(1)})_{b b}\big(\underline{G_i^{(1)} A_i^{(1)} W^{(1)} G_i^{(1)}}\big)_{a b}\\
      =&  \frac{1}{n^{7/2}}{\sum_{ab}}'\bigg\{m_1^{(1)}\big(M_1^{(1)} A_1^{(1)}\big)_{b b}\big(\underline{G_i^{(1)} A_i^{(1)} W^{(1)} G_i^{(1)}}\big)_{a b}\notag\\
     &+(G^{(1)}_1-M^{(1)}_1)_{aa}\big(M_1^{(1)} A_1^{(1)}\big)_{b b}\big(\underline{G_i^{(1)} A_i^{(1)} W^{(1)} G_i^{(1)}}\big)_{a b}\notag\\
     &+m_1^{(1)}\big(G^{(1)}_1-M^{(1)}_1)A^{(1)}_1\big)_{bb}\big(\underline{G_i^{(1)} A_i^{(1)} W^{(1)} G_i^{(1)}}\big)_{a b}\notag\\
     &+\big(G^{(1)}_1-M^{(1)}_1)A^{(1)}_1\big)_{bb}(G^{(1)}_1-M^{(1)}_1)_{aa}\big(\underline{G_i^{(1)} A_i^{(1)} W^{(1)} G_i^{(1)}}\big)_{a b}\bigg\}\notag\\
     = &\mathcal{O}_\prec \bigg( \frac{1}{n^{7/2}}\frac{1}{|\beta_1^{(1)}|}\frac{n}{n^{1/2}\eta_i^{3/2}|\beta^{(1)}_i|}+\frac{1}{n^{7/2}}\bigg(\frac{1}{\sqrt{n\eta_1}}\frac{1}{|\beta_1^{(1)}|}\frac{n}{n^{1/2}\eta_i^{3/2}|\beta^{(1)}_i|}
     \notag\\
     &+\frac{n^2}{n\eta_1|\beta_1^{(1)}|}\frac{1}{n^{1/2}\eta_i^{3/2}|\beta^{(1)}_i|}+\frac{n^2}{n\eta_1|\beta_1^{(1)}|}\frac{1}{\sqrt{n\eta_1}}\frac{1}{n^{1/2}\eta_i^{3/2}|\beta^{(1)}_i|}\bigg)\bigg)\notag\\
      =&\mathcal{O}_\prec\Bigg(\frac{1}{n^3 \eta_1 \eta_i^{3 / 2}|\beta_1^{(1)}||\beta_i^{(1)}|}\Bigg).\notag
\end{align}
Similarly, by using \eqref{eq:W^r},  \eqref{eq:G^iBW^iG^ibound}, and the estimates from Lemma \ref{Thm: Local_Law}, we obtain the following estimate for the remaining terms of \eqref{eqn: H_2 plus first two terms of H_3};
\begin{align}\label{eq:k=2,r=2 remaing terms}
&\frac{1}{n^{7/2}}\mathbb{E}{\sum_{ab}}'(G_1^{(1)}A_1^{(1)})_{ab}\Big((\underline{G_i^{(1)}A_i^{(1)}W^{(1)}G_i^{(1)}})_{ab}(G_i^{(1)})_{bb}\\
&\quad\quad+\frac{1}{n}(G_i^{(1)}A_i^{(1)}EG_i^{(1)})_{ab}(G_i^{(1)}E^{\dag}G_i^{(1)})_{ab}\Big)+\frac{1}{n^{7/2}}\mathbb{E}{\sum_{ab}}'(G_1^{(1)}A_1^{(1)})_{ab}\notag\\
&\quad\quad\quad\Big((\underline{G_i^{(2)}A_i^{(2)}W^{(2)}G_i^{(2)}})_{ab}(G_i^{(2)})_{bb}+\frac{1}{n}(G_i^{(2)}A_i^{(2)}EG_i^{(2)})_{ab}(G_i^{(2)}E^{\dag}G_i^{(2)})_{ab}\Big)\notag\\
&\quad=\mathcal{O}_{\prec}\left(\frac{1}{n^3 \eta_1^{1 / 2} \eta_i^2|\beta_i^{(1)}||\beta_1^{(1)}|}\right)\notag.
\end{align}

\subsubsection{Contribution of \eqref{eq:RHSofeq:cummulant_exp_I_1^{(3)}_1} to $T_{5}^{(1)}$ when $k=2, r=3$}
This case consists of the remaining terms of $H_{3}$ (see \eqref{eq:H_3}), which were not considered in the previous cases (i.e., $k=2, r\leq 2$). The expression becomes
$$
\begin{aligned}
    &\langle-\Delta^{ba} G_1^{(1)} A_1^{(1)}\rangle \langle\Delta G_i^{(1)} A_i^{(1)}- \underline{W^{(1)} G_i^{(1)} \Delta G_i^{(1)}} A_i^{(1)}\rangle \langle \Delta G_j^{(1)} A_j^{(1)}-\underline{W^{(1)} G_j^{(1)} \Delta G_j^{(1)}} A_j^{(1)}\rangle\\
    &+\langle-\Delta^{ba} G_1^{(1)} A_1^{(1)}\rangle \langle\Delta G_i^{(2)} A_i^{(2)}-\underline{W^{(2)} G_i^{(2)} \Delta G_i^{(2)}} A_i^{(2)}\rangle \langle \Delta G_j^{(2)} A_j^{(2)}-\underline{W^{(2)} G_j^{(2)} \Delta G_j^{(2)}} A_j^{(2)}\rangle\\
    &+\langle-\Delta^{ba} G_1^{(1)} A_1^{(1)}\rangle \langle\Delta G_i^{(1)} A_i^{(1)}-\underline{W^{(1)} G_i^{(1)} \Delta G_i^{(1)}} A_i^{(1)}\rangle \langle \Delta G_j^{(2)} A_j^{(2)}-\underline{W^{(2)} G_j^{(2)} \Delta G_j^{(2)}} A_j^{(2)}\rangle
\end{aligned}
$$

which, using \eqref{eq:G^iBW^iG^ibound}, we estimate by
\begin{align}\label{eq:k=2,r=3}
    & \frac{1}{n^{9 / 2}} {\sum_{ab}}'\Big\{-(G_1^{(1)} A_1^{(1)})_{a b}\big((G_i^{(1)} A_i^{(1)})_{a b}-(\underline{G_i^{(1)} A_i^{(1)} W^{(1)} G_i^{(1)}})_{a b}\big)\\
&\;\;\;\;\;\big((G_j^{(1)} A_j^{(1)})_{a b} -(\underline{G_j^{(1)} A_j^{(1)} W^{(1)} G_j^{(1)}})_{a b}\big)-(G_1^{(1)} A_1^{(1)})_{a b}\big((G_i^{(2)} A_i^{(2)})_{a b}\notag\\
&\;\;\;\;\;-(\underline{G_i^{(2)} A_i^{(2)} W^{(2)} G_i^{(2)}})_{a b}\big)\big((G_j^{(2)} A_j^{(2)})_{a b}-(\underline{G_j^{(2)} A_j^{(2)} W^{(2)} G_j^{(2)}})_{a b}\big)\notag\\
&\;\;\;\;\;-(G_1^{(1)} A_1^{(1)})_{a b}\big((G_i^{(1)} A_i^{(1)})_{a b}-(\underline{G_i^{(1)} A_i^{(1)} W^{(1)} G_i^{(1)}})_{a b}\big)\big((G_j^{(2)} A_j^{(2)})_{a b}\notag\\
&\;\;\;\;\;-(\underline{G_j^{(2)} A_j^{(2)} W^{(2)} G_j^{(2)}})_{a b}\big)\Big\}\notag\\
& \quad=\mathcal{O}_{\prec}\left(\frac{1}{n^4 \eta_1^{1 / 2} \eta_i^{3 / 2} \eta_j^{3 / 2}|\beta_1^{(1)}||\beta_i^{(1)}| |\beta_j^{(2)}|}\right)\notag.
\end{align}

In summary, from the \eqref{eq:k=2,r=1}, \eqref{eq:k=2,r=2_first four term}, \eqref{eq:k=2,r=2 remaing terms}, and \eqref{eq:k=2,r=3}, we have the essential estimates for the terms of \eqref{eq:RHSofeq:cummulant_exp_I_1^{(3)}_1} when $k=2.$ However, none of estimates consider the attached partial product $\prod\Upsilon_{i},$ where some of $\Upsilon_{i}$s are omitted. But this can be bounded by some form of the estimate given in \eqref{eqn: definition of psi error}. Consequently, the updated estimates of \eqref{eq:k=2,r=1}, \eqref{eq:k=2,r=2_first four term}, \eqref{eq:k=2,r=2 remaing terms}, and \eqref{eq:k=2,r=3} will become as
\begin{align*}
    &\mathcal{O}_\prec\Big(\frac{\psi}{n^{3/2}\eta_1}\Big), \mathcal{O}_\prec\Big(\frac{\psi}{n^{2}\eta_i\sqrt{\eta_1}}\Big), \mathcal{O}_\prec\Big(\frac{\psi}{n^2\eta_i^{3/2}}\Big), \text{ and }\mathcal{O}_\prec\Big(\frac{\psi}{n^{5/2}\eta_i\eta_j}\Big)
\end{align*}
respectively. Hence, the final contribution of \eqref{eq:RHSofeq:cummulant_exp_I_1^{(3)}_1} in the case of $k=2$ is $\mathcal{O}_\prec\Big(\frac{\psi}{n^{3/2}\eta_*}\Big).$

\subsubsection{Contribution of \eqref{eq:RHSofeq:cummulant_exp_I_1^{(3)}_1} to $T_{5}^{(1)}$ when $k\geq 3$}
Now, for $k=3$, the RHS term of \eqref{eq:RHSofeq:cummulant_exp_I_1^{(3)}_1} is 
\begin{align}\label{eq:RHSofeq:cummulant_exp_I_1^{(3)}}
    &\sum_{ab}\sum_{\alpha\in\{(ab)_1,(ba)_1,(ab)_2,(ba)_2\}^3}\frac{\kappa((ba)_1,\alpha)}{3!}\mathbb{E}\bigg[\partial_\alpha\Big(\langle-\Delta^{ba}G_1^{(1)}A_1^{(1)}\rangle\prod_{i\neq1}\Upsilon_i\Big)\bigg].
\end{align}
Before considering this case, let us look at the following remark which summarizes the estimate types from previous three subsections.
\begin{rem}\label{rem:leading_contri_discussion}
\begin{enumerate}
    \item If we have the following sum;
$$
\sum_{ab} \prod_{i=1}^l \Delta^{\alpha_i} G_i^{(t)},
$$
where $\alpha_i \in \{ab, ba\}$, then we have the following two cases.
\begin{itemize}
    \item If $l$ is odd, no product term in the sum can consist entirely of diagonal entries; there must be at least one off-diagonal entry in each term.
    \item If $l$ is even, it is possible to have a product term where all the entries are diagonal.
\end{itemize}
\item If we have entries such as $(G_{i}^{(t)})_{ab}$ or $(G_{i}^{(t)}A_{i}^{(t)})_{ab},$ then these are decomposed as $(M_{i}^{(t)})_{ab}+(G_{i}^{(t)}-M_{i}^{(t)})_{ab}$ or $(M_{i}^{(t)}A_{i}^{(t)})_{ab}+((G_{i}^{(t)}-M_{i}^{(t)})A_{i})_{ab}$ respectively. Now, the terms $(M_{i}^{(t)})_{ab}$ or $(M_{i}^{(t)}A_{i}^{(t)})_{ab}$ yield a significant contribution if these are diagonal terms. Therefore if a product term contains a off-diagonal term, it does not give any significant contribution, which is further explained in the following point.
\item If we have off-diagonal factors such as $(G_{i}^{(t)} - M_{i}^{(t)})_{ab}$,  $((G_{i}^{(t)} - M_{i}^{(t)})A_{i}^{(t)})_{ab}$, $(\underline{G_{i}^{(t)}A_{i}^{(t)}W^{(t)}G_{i}^{(t)}})_{ab}$ etc. in the product terms, then using the local law (Theorem \ref{Thm: Local_Law}) or \eqref{eq:G^iBW^iG^ibound}, we argue that such terms do not give any contribution. This was demonstrated in the derivations \eqref{eq:k=2,r=1}, \eqref{eq:k=2,r=2_first four term}, \eqref{eq:k=2,r=2 remaing terms}, and \eqref{eq:k=2,r=3}.
\item In \eqref{eq:RHSofeq:cummulant_exp_I_1^{(3)}_1}, after performing the $\alpha$-derivatives, we obtain a sum of product of terms, where each of the product is a product of $p$-terms. Essentially, the terms in the product can be categorized into two categories (a) the terms in $\prod_{i}\Upsilon_{i}$ which are affected by the derivatives (b) terms in $\prod_{i}\Upsilon_{i}$, which are unaffected by the derivatives. The terms, which are affected by the differentiation, will be a normalized trace of the form $\langle (\Delta G_{i}^{(t)})^{k_{i}} A_{i}\rangle$ or $\langle \underline{W^{(t)}(G_{i}^{(t)}\Delta)^{k_{i}}G_{i}^{(t)}A_{i}^{(t)}}\rangle,$ where $k_{i}\in \mathbb{N}.$ Moreover, the first term $\langle -\Delta^{ba}G_{1}^{(1)}A_{1}^{(1)} \rangle$ is also a normalized trace. Overall if there are $r$ many normalized traces, this will give a pre-factor of $n^{-r-(k+1)/2}$ (including the cumulant estimate \eqref{eqn: cumulant estimate}). This is also mentioned on page \pageref{eq:H_3}. Additionally, $\sum_{i=1}^{r}k_{i}=k+1.$
\item Now, if the product contains all the diagonal entries (similar to \eqref{eq:DeltaG_DeltaG_DeltaG_DeltaGA}), we may get some contribution. In such cases, we still have the pre-factor $n^{-r-(k+1)/2}.$ Giving a naive bound (of constant order) to each of those diagonal entries, the prefactor will get reduced to $n^{-r-(k-3)/2}.$ We loose a factor of $n^{2}$ because the sum $\sum'_{ab}$ contains $2n^{2}$ many terms. From here, it is clear that the leading order contribution comes from the case $k=3.$
\end{enumerate}
\end{rem}
Now, for $k = 3$ proceeding as \eqref{eqn: k=2 derivative written as H sum}, we can identify the terms which contain all diagonal entries. The terms which can possibly contain all diagonal entries are of the following forms; 
$$
\langle  \Delta^{ba} G_1^{(1)} \Delta G_1^{(1)} \Delta G_1^{(1)} \Delta G_1^{(1)} A_1^{(1)}\rangle, \quad\langle  \Delta^{ba} G_1^{(1)} \Delta G_1^{(1)} A_1^{(1)}\rangle \langle  \Delta G_i^{(t)} \Delta G_i^{(t)} A_i^{(t)}\rangle.
$$

For the second types of terms, we notice that $\alpha=\big((ab)_1, (ab)_1, (ba)_1\big)$ yields a terms having product of diagonal entries, which can be evaluated as follows;
\begin{align*}
& \kappa\left((ba)_1,(ab)_1, (ab)_1, (ba)_1\right) {\sum_{ab}}'\langle\Delta^{ba} G_1^{(1)} \Delta^{a b} G_1^{(1)} A_1^{(1)}\rangle\langle\Delta^{ab} G_i^{(1)} \Delta^{ba} G_i^{(1)} A_i^{(1)}\rangle\\
=& \frac{(\kappa_4)_{11}}{4n^4} {\sum_{ab}}'(G_1^{(1)})_{aa}(G_1^{(1)} A_1^{(1)})_{bb}(G_i^{(1)})_{bb} (G_i^{(1)} A_i^{(1)})_{aa}\\
=& \frac{(\kappa_4)_{11}}{4n^4} {\sum_{ab}}'m_1^{(1)} m_i^{(1)}(M_1^{(1)} A_1^{(1)})_{bb} (M_i^{(1)} A_i^{(1)})_{aa}+\mathcal{O}_{\prec}\Big(\frac{1}{n^{5/2}\eta_*^{1/2}|\beta_{1i}^{*}|^2}\Big)\\
=& \frac{(\kappa_4)_{11}}{2n^2} \langle M_1^{(1)} \rangle \langle M_i^{(1)} \rangle \langle M_1^{(1)} A_1^{(1)} \rangle \langle M_i^{(1)} A_i^{(1)} \rangle+\mathcal{O}_{\prec}\Big(\frac{1}{n^{5/2}\eta_*^{1/2}|\beta_{1i}^{*}|^2}\Big).
\end{align*}
The other choices of $\alpha$ which give a significant contribution are $\alpha=\big((ab)_1, (ba)_1, (ab)_1\big)$, $ \big((ab)_1, (ab)_2, (ba)_2\big), \big((ab)_1, (ba)_2, (ab)_2\big).$ Combining all the contributions, we obtain
\begin{align*}
    \sum_{a b} & \sum_\alpha \kappa((ba)_1, \alpha)\langle\Delta^{ba} G_1^{(1)} \Delta^{\alpha_1} G_1^{(1)}\rangle\langle\Delta^{\alpha_2} G_i^{(t)} \Delta^{\alpha_3} G_i^{(t)} A_i^{(t)}\rangle \\
    &=\frac{(\kappa_4)_{11}}{n^2} \langle M_1^{(1)} \rangle \langle M_i^{(1)} \rangle \langle M_1^{(1)} A_1^{(1)} \rangle \langle M_i^{(1)} A_i^{(1)} \rangle\\
    &\quad+ \frac{(\kappa_4)_{12}}{n^2} \langle M_1^{(1)} \rangle \langle M_i^{(2)} \rangle \langle M_1^{(1)} A_1^{(1)} \rangle \langle M_i^{(2)} A_i^{(2)} \rangle\\
   &\quad+\mathcal{O}_{\prec}\Big(\frac{1}{n^{5/2}\eta_*^{1/2}|\beta_{1i}^{*}|^2}\Big)+\mathcal{O}_{\prec}\Big(\frac{1}{n^3\eta_*|\beta_{1i}^{*}|^2}\Big)\\
   &=:\frac{(\kappa_4)_{11}}{2n^2} U_1^{(1)}U_i^{(1)}+ \frac{(\kappa_4)_{12}}{2n^2} U_1^{(1)}U_i^{(2)}
   +\mathcal{O}_{\prec}\Big(\frac{1}{n^{5/2}\eta_*^{1/2}|\beta_{1i}^{*}|^2}\Big),
\end{align*}
where, 
\begin{equation}\label{eq:U_i^s}
     U_i^{(t)}:=-\sqrt{2}\langle M_i^{(t)} \rangle \langle M_i^{(t)} A_i^{(t)} \rangle=\frac{\iota}{\sqrt{2}} \partial_{\eta_i} (m_i^{(t)})^2.
\end{equation}

Similarly, for the first one, only $\alpha = \big((ab)_1 (ba)_1, (ab)_1\big)$ gives significant contribution as follows;
\begin{align}\label{eq:kequalto3contri}
    & \sum_{a b} \sum_\alpha \kappa((ba)_1, \alpha)\langle \Delta^{ba} G_1^{(1)} \Delta^{ab} G_1^{(1)} \Delta^{ba} G_1^{(1)} \Delta^{ab} G_1^{(1)} A_1^{(1)}\rangle\notag \\
& \quad=\frac{(\kappa_4)_{11}}{n}\langle M_1^{(1)}\rangle^3\langle M_1^{(1)}A_1^{(1)}\rangle+\mathcal{O}_{\prec}\Big(\frac{1}{n^{3/2}\eta_1^{1/2}|\beta_1^{(1)}|}\Big)+\mathcal{O}_\prec\Big(\frac{1}{n^2\eta_1|\beta_1^{(1)}|}\Big)\notag\\
& \quad=-\frac{\iota(\kappa_4)_{11}}{4n}\partial_{\eta_1}(m_1^{(1)})^4+\mathcal{O}_\prec\Big(\frac{1}{n^2\eta_1|\beta_1^{(1)}|}\Big).
\end{align}
The above calculation is same as \eqref{eq:DeltaG_DeltaG_DeltaG_DeltaGA}.
Thus, we can conclude for the $k \geq 3$ terms in \eqref{eq:cummulant_exp_I_1^{(3)}} that

\begin{align}\label{eq:T_5^{(1)} for kgeq 3}
    &\sum_{ab}\sum_{k\geq 3}\sum_{\alpha\in\{(ab)_1,(ba)_1,(ab)_2,(ba)_2\}^k}\frac{\kappa((ba)_1,\alpha)}{k!}\mathbb{E}\bigg[\partial_\alpha\Big(\langle-\Delta^{ba}G_1^{(1)}A_1^{(1)}\rangle\prod_{i\neq1}\Upsilon_i\Big)\bigg]\\
    =& -\frac{\iota(\kappa_4)_{11}}{4n}\partial_{\eta_1}(m_1^{(1)})^4\mathbb{E}\Big[\prod_{i\neq1}\Upsilon_i\Big]
    + \sum_{i\neq1}\frac{1}{2n^2}\big (c(\kappa_4)_{11}U_1^{(1)}U_i^{(1)}+d(\kappa_4)_{12}U_1^{(1)}U_i^{(2)}\big)\notag\\
    &\;\;\mathbb{E}\bigg[\prod_{j\neq i,1}\Upsilon_j\bigg]
   +\mathcal{O}_\prec\Big(\frac{1}{n^2\eta_1|\beta_1^{(1)}|}\Big)\mathbb{E}\Big[\prod_{i\neq1}\Upsilon_i\Big]+\mathcal{O}_{\prec}\bigg(\frac{1}{n^{5/2}\eta_*^{1/2}|\beta_{1i}^{*}|^2} \bigg)\mathbb{E}\bigg[\prod_{j\neq i,1}\Upsilon_j\bigg]\notag\\
   =& -\frac{\iota(\kappa_4)_{11}}{4n}\partial_{\eta_1}(m_1^{(1)})^4\mathbb{E}\Big[\prod_{i\neq1}\Upsilon_i\Big]
    + \sum_{i\neq1}\frac{1}{2n^2}\big (c(\kappa_4)_{11}U_1^{(1)}U_i^{(1)}+d(\kappa_4)_{12}U_1^{(1)}U_i^{(2)}\big)\notag\\
    &\;\;\mathbb{E}\bigg[\prod_{j\neq i,1}\Upsilon_j\bigg]
   +\mathcal{O}_{\prec}\bigg(\frac{\psi}{n^{3/2}\eta_1^{1/2}}+\frac{\psi (\eta_*)^{1/2}}{n^2}\bigg)\notag.
\end{align}
Thus, from page \pageref{eq:k=2,r=3} and \eqref{eq:T_5^{(1)} for kgeq 3} we have 
\begin{align}\label{T^{(5)}_1_final}
    &T_5^{(1)}\\
    =&-\frac{\iota(\kappa_4)_{11}}{4n}\partial_{\eta_1}(m_1^{(1)})^4\mathbb{E}\Big[\prod_{i\neq1}\Upsilon_i\Big]
    + \sum_{i\neq1}\frac{1}{2n^2}\big(c(\kappa_4)_{11}U_1^{(1)}U_i^{(1)}\notag\\
    &+d(\kappa_4)_{12}U_1^{(1)}U_i^{(2)}\big)\mathbb{E}\bigg[\prod_{j\neq i,1}\Upsilon_j\bigg]+\mathcal{O}_{\prec}\bigg(\frac{\psi}{n^{3/2}\eta_*^{1/2}}\bigg)\notag.
\end{align}
Similarly, we can evaluate $T_5^{(2)}$ and we have the final expression for $T_5^{(1)}+T_5^{(2)}$ as follows;
\begin{align}\label{eq:T5p{(1)}+T5p{(2)}_final}
    &T_5^{(1)}+T_5^{(2)}\\
    =&-\frac{\iota(\kappa_4)_{11}}{4n}\partial_{\eta_1}(m_1^{(1)})^4\mathbb{E}\Big[\prod_{i\neq1}\Upsilon_i\Big]
    + \sum_{i\neq1}\frac{1}{2n^2}\big (c(\kappa_4)_{11}U_1^{(1)}U_i^{(1)}\notag\\
    &+d(\kappa_4)_{12}U_1^{(1)}U_i^{(2)}\big)\mathbb{E}\Big[\prod_{j\neq i,1}\Upsilon_j\Big]-\frac{\iota(\kappa_4)_{22}}{4n}\partial_{\eta_1}(m_1^{(2)})^4\mathbb{E}\Big[\prod_{i\neq1}\Upsilon_i\Big]
    \notag\\
    &+ \sum_{i\neq1}\frac{1}{2n^2}\big( c(\kappa_4)_{22}U_1^{(2)}U_i^{(2)}+d(\kappa_4)_{21}U_1^{(2)}U_i^{(1)}\big)\mathbb{E}\bigg[\prod_{j\neq i,1}\Upsilon_j\bigg]+\mathcal{O}_{\prec}\bigg(\frac{\psi}{n^{3/2}\eta_*^{1/2}}\bigg)\notag\\
    =&-\frac{\iota(\kappa_4)_{11}}{4n}\partial_{\eta_1}(m_1^{(1)})^4\mathbb{E}\Big[\prod_{i\neq1}\Upsilon_i\Big]
    + \sum_{i\neq1}\frac{1}{2n^2}\big (c(\kappa_4)_{11}U_1^{(1)}U_i^{(1)}\notag\\
    &+d(\kappa_4)_{12}U_1^{(1)}U_i^{(2)}\big)\mathbb{E}\Big[\prod_{j\neq i,1}\Upsilon_j\Big]-\frac{\iota(\kappa_4)_{22}}{4n}\partial_{\eta_1}(m_1^{(2)})^4\mathbb{E}\Big[\prod_{i\neq1}\Upsilon_i\Big]
    \notag\\
    &+ \sum_{i\neq1}\frac{1}{2n^2}\big( c(\kappa_4)_{22}U_1^{(2)}U_i^{(2)}+d(\kappa_4)_{21}U_1^{(2)}U_i^{(1)}\big)\mathbb{E}\bigg[\prod_{j\neq i,1}\Upsilon_j\bigg]+error_5.\notag
\end{align}

\subsection{Concluding the proof of the Theorem \ref{Thm:CLT for resolvents}}

 From \eqref{eq:I_1+I_2_final} and \eqref{eq:T5p{(1)}+T5p{(2)}_final} we have
\begin{align*}
    &\mathbb{E}\Big[\prod_{i\in[p]}\Upsilon_i\Big]\\
    =& c\bigg\{-\frac{\iota(\kappa_4)_{11}}{4n}\partial_{\eta_1}(m_1^{(1)})^4\mathbb{E}\Big[\prod_{i\neq1}\Upsilon_i\Big]
    + \sum_{i\neq1}\frac{1}{2n^2}\big (c(\kappa_4)_{11}U_1^{(1)}U_i^{(1)}\notag\\
    &+d(\kappa_4)_{12}U_1^{(1)}U_i^{(2)}\big)\mathbb{E}\Big[\prod_{j\neq i,1}\Upsilon_j\Big]\bigg\}+d\bigg\{-\frac{\iota(\kappa_4)_{22}}{4n}\partial_{\eta_1}(m_1^{(2)})^4\mathbb{E}\Big[\prod_{i\neq1}\Upsilon_i\Big]\notag\\
    &+ \sum_{i\neq1}\frac{1}{2n^2}\big( c(\kappa_4)_{22}U_1^{(2)}U_i^{(2)}+d(\kappa_4)_{21}U_1^{(2)}U_i^{(1)}\big)\mathbb{E}\bigg[\prod_{j\neq i,1}\Upsilon_j\bigg]\bigg\}\notag\\
        &+c\frac{\iota (\kappa_4)_{11}}{4n} \partial_{\eta_1}(m_1^{(1)})^4
         \mathbb{E}\Big[\prod_{i\neq1}\Upsilon_{i}\Big]
        +d\frac{\iota (\kappa_4)_{22}}{4n} \partial_{\eta_1}(m_1^{(2)})^4
        \mathbb{E}\Big[\prod_{i\neq1}\Upsilon_{i}\Big]\\
        &+ \sum_{i\neq 1}\frac{c^2\widehat{V_{1i}^{(1)}}}{2n^2}\mathbb{E}\Big[\prod_{j\neq1,i}\Upsilon_{j}\Big] +\sum_{i\neq 1}\frac{d^2\widehat{V_{1i}^{(2)}}}{2n^2}\mathbb{E}\Big[\prod_{j\neq1,i}\Upsilon_{j}\Big]\notag\\
       &+\sum_{i\neq 1}\Bigg\{\frac{8cd\gamma m_i^{(2)}m_1^{(1)} u_1^{(1)}u_i^{(2)}\Re{(\overline{z_1}z_i)}}{2n^2\beta_i^{(2)}\beta_1^{(1)}}+\frac{8cd\gamma m_i^{(1)}m_1^{(2)} u_i^{(1)}u_1^{(2)}\Re{(\overline{z_1}z_i)}}{2n^2\beta_i^{(1)}\beta_1^{(2)}}\Bigg\}\notag\\
        &\quad\mathbb{E}\Big[\prod_{j\neq1,i}\Upsilon_{j}\Big]+ \frac{cd\rho}{n^2}\sum_{i\neq1}\Big( -\partial_{\eta_1}(m_1^{(1)}) \partial_{\eta_i}(m_i^{(2)})
         - \partial_{\eta_1}(m_1^{(2)}) \partial_{\eta_i}(m_i^{(1)}) \Big)\notag\\
        &\quad\mathbb{E}\Big[\prod_{j\neq1,i}\Upsilon_{j}\Big]
         +error_1+error_2+error_3+error_4+error_5\notag\\
    =& \sum_{i\neq1}\frac{1}{2n^2}\big(c^2(\kappa_4)_{11}U_1^{(1)}U_i^{(1)}+d^2(\kappa_4)_{22}U_1^{(2)}U_i^{(2)}\big)\mathbb{E}\Big[\prod_{j\neq i,1}\Upsilon_j\Big]\notag\\
     &+ \sum_{i\neq1}\frac{cd(\kappa_4)_{12}}{2n^2}\big(U_1^{(2)}U_i^{(1)}+U_1^{(1)}U_i^{(2)}\big)\mathbb{E}\Big[\prod_{j\neq i,1}\Upsilon_j\Big]+ \sum_{i\neq 1}\frac{c^2\widehat{V_{1i}^{(1)}}}{2n^2}\mathbb{E}\Big[\prod_{j\neq1,i}\Upsilon_{j}\Big]\notag\\
     &+\sum_{i\neq 1}\frac{d^2\widehat{V_{1i}^{(2)}}}{2n^2}\mathbb{E}\Big[\prod_{j\neq1,i}\Upsilon_{j}\Big]+\sum_{i\neq 1}\Bigg\{\frac{8cd\gamma m_i^{(2)}m_1^{(1)} u_1^{(1)}u_i^{(2)}\Re{(\overline{z_1}z_i)}}{2n^2\beta_i^{(2)}\beta_1^{(1)}}\notag\\
        &+\frac{8cd\gamma m_i^{(1)}m_1^{(2)} u_i^{(1)}u_1^{(2)}\Re{(\overline{z_1}z_i)}}{2n^2\beta_i^{(1)}\beta_1^{(2)}}\Bigg\}\mathbb{E}\Big[\prod_{j\neq1,i}\Upsilon_{j}\Big]\notag\\
        &+ \frac{cd\rho}{n^2}\sum_{i\neq1}\Big( -\partial_{\eta_1}(m_1^{(1)}) \partial_{\eta_i}(m_i^{(2)})
        - \partial_{\eta_1}(m_1^{(2)}) \partial_{\eta_i}(m_i^{(1)}) \Big)\mathbb{E}\Big[\prod_{j\neq1,i}\Upsilon_{j}\Big]\notag\\
        &
         +error_1+error_2+error_3+error_4+error_5\notag\\
          =&:\sum_{i\neq 1}\frac{c^2\widehat{V_{1i}^{(1)}}+d^2\widehat{V_{1i}^{(2)}} +8cd\gamma (L_{i1}^{12}+L_{i1}^{21})+2cd\rho  N_{1i} }{2n^2}\mathbb{E}\Big[\prod_{j\neq1,i}\Upsilon_{j}\Big]\notag\\
           &+\sum_{i\neq 1}\frac{c^2(\kappa_4)_{11}U_1^{(1)}U_i^{(1)}+d^2(\kappa_4)_{22}U_1^{(2)}U_i^{(2)}+cd(\kappa_4)_{12} \big(U_1^{(1)}U_i^{(2)}+U_i^{(1)}U_1^{(2)}\big) }{2n^2}\mathbb{E}\Big[\prod_{j\neq1,i}\Upsilon_{j}\Big]\notag\\
         &+error_1+error_2+error_3+error_4+error_5\notag\\
          &=:\frac{1}{2n^2}\sum_{i\neq 1}\Bigg\{\bigg(c^2\widehat{V_{1i}^{(1)}}+d^2\widehat{V_{1i}^{(2)}} +8cd\gamma (L_{i1}^{12}+L_{i1}^{21})+2cd\rho  N_{1i}+ c^2(\kappa_4)_{11}U_1^{(1)}U_i^{(1)}\notag\\
          &+d^2(\kappa_4)_{22}U_1^{(2)}U_i^{(2)}+cd(\kappa_4)_{12} \big(U_1^{(1)}U_i^{(2)}+U_i^{(1)}U_1^{(2)}\big)\bigg)\mathbb{E}\Big[\prod_{j\neq1,i}\Upsilon_{j}\Big]\Bigg\}\notag\\
         &+\mathcal{O}_\prec \bigg(\psi\Big(\frac{1}{(n\eta_*)^2}+\frac{1}{n^2 \eta_*^{1i}|z_1-z_i|^4} + \frac{1}{(n\eta_*^{1i})^3 |z_1-z_i|^4}+\frac{1}{\Im{z_1}(n\eta_1)^{3/2}}\Big)\bigg),\notag
\end{align*}
where
\begin{align}\label{eq:L_{i1}^{12}}
    &L_{i1}^{12} = \frac{ m_i^{(1)}m_1^{(2)} u_i^{(1)}u_1^{(2)}\Re{(\overline{z_1}z_i)}}{\beta_i^{(1)}\beta_1^{(2)}},
    & L_{i1}^{21} = \frac{m_i^{(2)}m_1^{(1)} u_i^{(2)}u_1^{(1)}\Re{(\overline{z_1}z_i)}}{\beta_i^{(2)}\beta_1^{(1)}},
\end{align}

and

\begin{align}\label{eq:N_{1i}}
    & N_{1i}=\partial_{\eta_1}(m_1^{(1)}) \partial_{\eta_i}(m_i^{(2)})
         + \partial_{\eta_1}(m_1^{(2)}) \partial_{\eta_i}(m_i^{(1)}).
\end{align}
Thus, by induction we have
\begin{align}
    &\mathbb{E}\bigg[\prod_{i\in[p]}\Upsilon_i\bigg]\\
    &= \frac{1}{2n^{p}}\sum_{P\in \prod_p} \prod_{\{i,j\}\in P}\Big\{c^2\widehat{V_{ij}^{(1)}}+d^2\widehat{V_{ij}^{(2)}} +8cd\gamma (L_{ji}^{12}+L_{ji}^{21})+2cd\rho  N_{ij}\notag\\
     &\;\;\;\;\;+c^2(\kappa_4)_{11}U_i^{(1)}U_j^{(1)}+d^2(\kappa_4)_{22}U_i^{(2)}U_j^{(2)}+cd(\kappa_4)_{12} \big(U_i^{(1)}U_j^{(2)}+U_j^{(1)}U_i^{(2)}\big)\Big\}\\
     &\;\;\;\;\;+\mathcal{O}_\prec \bigg(\psi\Big(\frac{1}{(n\eta_*)^2}+\frac{1}{n^2 \eta_*|z_i-z_j|^4} + \frac{1}{(n\eta_*)^3 |z_i-z_j|^4}+\frac{1}{\Im{z_i}(n\eta_*)^{3/2}} \Big)\bigg),\notag
     \end{align}
     where $\eta_*= \displaystyle \min_i{\eta_i}.$

Now, we will estimate the bound \eqref{eq:error_bound} in Lemma \ref{lemma:overlinebound}. However, before proceeding, we will state the following lemma;
\begin{lemma}\cite[Lemma 5.8]{cipolloni2023central}\label{lemma:estimate_of_producr_of_G_ii}
    Let $\omega_i$ and $z_i$ be arbitrary spectral parameters with $\eta_i = \Im(\omega_i) > 0$. Define  $G_j^{(t)} = (G^{(t)})^{z_j}(\omega_j),$
and let $G^{(t)}_{j_1, \dots, j_k}$ denote products of resolvents (or their adjoints/transposes) interspersed with bounded deterministic matrices, e.g.,  
$G^{(t)}_{1i1} = A_1^{(t)} G_1^{(t)} A_2^{(t)} G_i^{(t)} A_3^{(t)} G_1^{(t)} A_4^{(t)}.$
Then for $j_{1}, j_{2}, \ldots, j_{k}$ and any $1\leq s<t\leq k,$ we obtain the following bounds; 
    \begin{enumerate}
        \item Isotropic Bound
        $$
          |\langle \boldsymbol{x}, G^{(t)}_{j_1, \dots, j_k} \boldsymbol{y} \rangle| \prec \|\boldsymbol{x}\| \|\boldsymbol{y}\| \sqrt{\eta_{j_1} \eta_{j_k}} \Big(\prod_{i=1}^{k} \eta_{j_i}^{-1}\Big),
        $$
        \item Averaged Bound 
        $$
        |\langle G^{(t)}_{j_1,\dots, j_k} \rangle| \prec \sqrt{\eta_{j_s}\eta_{j_t}}\Big(\prod_{i=1}^{k} \eta_{j_i} \Big)^{-1}.
        $$
    \end{enumerate}
\end{lemma}
\begin{lemma}\label{lemma:overlinebound} Let $G_{i}^{(t)},$ $G^{(t)}_{j_1, \dots, j_k}$ be same as defined in Lemma \ref{lemma:estimate_of_producr_of_G_ii}, and $A_{i}^{(t)}$ be deterministic matrices with $\|A_{i}^{(t)}\|\leq |\beta_{i}^{(t)}|^{-1}$. Then we have the following estimate;
    \begin{align*}
        &\mathbb{E}\big[|\langle \underline{G^{(t)}_1A^{(t)}_1EG_i^{(s)}A_i^{(s)}W^{(s)}G_i^{(s)}E^{\dag}} \rangle|^2\big]\\
        &= \mathcal{O}_{\prec}\Big(\frac{1}{n^2 \eta_i^4|\beta_1^{(t)}|^{2}|\beta_i^{(s)}|^{2}}\Big),\text{  for } 1\leq t\neq s\leq 2.
    \end{align*}
\begin{proof}
Consider
    \begin{align*}
         &\mathbb{E}\big[|\langle \underline{G^{(1)}_1A^{(1)}_1EG_i^{(2)}A_i^{(2)}W^{(2)}G_i^{(2)}E^{\dag}} \rangle|^2\big]\\
         &=\mathbb{E}\big[|\langle \underline{W^{(2)}G_i^{(2)}E^{\dag}G^{(1)}_1A^{(1)}_1EG_i^{(2)}A_i^{(2)}} \rangle|^2\big]\\
         &=\mathbb{E}\big[|\langle \underline{W^{(2)}G_{ii}^{(2)}}\rangle|^2\big].
    \end{align*}
Now, we compute $\mathbb{E}\big[|\langle \underline{W^{(2)}G_{ii}^{(2)}}\rangle|^2\big]$ by using a cumulant expansion \ref{Result:Cumulant Expansion}. The following cumulant expansion is similar to \eqref{eq:I^{(1)}}, and the calculations thereafter.
\begin{align}\label{eq:cumulant}
& \mathbb{E}\big[|\langle \underline{W^{(2)}G_{ii}^{(2)}}\rangle|^2\big] \notag\\
 =& \frac{1}{n}{\sum_{ab}}' \mathbb{E}\big[\langle \Delta^{ba}G^{(2)}_{ii}\rangle \partial_{ab}\big(\langle \underline{W^{(2)}G^{(2)}_{ii}} \rangle\big)\big]+\frac{\gamma_1}{n}{\sum_{ab}}' \mathbb{E}\big[\partial_{ba}\big(\langle \Delta^{ba}G^{(2)}_{ii}\rangle\big) \langle \underline{W^{(2)}G^{(2)}_{ii}} \rangle\big]\notag\\
&+\frac{\gamma_1}{n}{\sum_{ab}}' \mathbb{E}\big[\langle \Delta^{ba}G^{(2)}_{ii}\rangle \partial_{ba}\big(\langle \underline{W^{(2)}G^{(2)}_{ii}} \rangle\big)\big]\notag\\
&+ {\sum_{ab}}'\sum_{k\geq 2}\sum_{\alpha\in\{(ab)_2,(ba)_2\}^k} \frac{\kappa\big((ba)_2,\alpha\big)}{k!}\mathbb{E}\big[\partial_{\alpha}\big(\langle \Delta^{ba}G^{(2)}_{ii} \rangle \langle \underline{W^{(2)}G^{(2)}_{ii}} \rangle\big)\big]\notag\\
=&\mathbb{E} \Big[\widetilde{\mathbb{E}}\big[\langle\widetilde{W}^{(2)} G_{ii}^{(2)}\rangle\big(\langle\widetilde{W}^{(2)} G_{ii}^{(2)}\rangle +\langle\underline{W^{(2)}G_i^{(2)}\widetilde{W}^{(2)}G_{ii}^{(2)}} \rangle+\langle\underline{W^{(2)}G_{ii}^{(2)}\widetilde{W}^{(2)}G_{i}^2} \rangle \big)\big]\Big]\\
&+\frac{\gamma_1}{n}{\sum_{ab}}' \mathbb{E}\big[\langle \Delta^{ba}\partial_{ba}G^{(2)}_{ii}\rangle \langle \underline{W^{(2)}G^{(2)}_{ii}} \rangle\big]+ \frac{\gamma_1}{n}{\sum_{ab}}' \mathbb{E}\big[\langle \Delta^{ba}G^{(2)}_{ii}\rangle \partial_{ba}\big(\langle \underline{W^{(2)}G^{(2)}_{ii}} \rangle\big)\big] \notag\\
& +\sum_{k \geq 2} {\sum_{a b}}^\prime \sum_{k_1+k_2=k-1} \sum_{\alpha_1, \alpha_2}\frac{\kappa((ba)_2, \alpha_1,\alpha_2)}{k!} \mathbb{E}\big[\langle\Delta^{ba} \partial_{\alpha_1} G_{ii}^{(2)}\rangle \langle\Delta^{ba} \partial_{\alpha_2} G_{ii}^{(2)}\rangle\big] \notag\\
& +\sum_{k \geq 2}{\sum_{a b}}^\prime \sum_{k_1+k_2=k} \sum_{\alpha_1, \alpha_2}\frac{\kappa((ba)_2, \alpha_1,\alpha_2)}{k!} \mathbb{E}\big[\langle\Delta^{ba} \partial_{\alpha_1} G_{ii}^{(2)}\rangle\langle \underline{W^{(2)} \partial_{\alpha_2} G_{ii}^{(2)}}\rangle\big], \notag
\end{align}
where $\alpha_i \in\{(ba)_2, (ba)_2\}^{k_i}$. 

First, we bound the terms that do not contain self-renormalization terms, as we do not need to expand them further. So, first consider the term
\begin{align}\label{eq:cumulant..}
    \widetilde{\mathbb{E}}\big[\langle\widetilde{W}^{(2)} G_{ii}^{(2)}\rangle\langle\widetilde{W}^{(2)} G_{ii}^{(2)}\rangle \big]&=  \frac{1}{2n^2}\langle G_{ii}^{(2)}E_1G_{ii}^{(2)}E_2 + G_{ii}^{(2)}E_2G_{ii}^{(2)}E_1 \rangle\\
    &=\frac{\langle G_{iiii}^{(2)} \rangle}{2n^2}= \mathcal{O}_{\prec}\Big(\frac{1}{n^2 \eta_i^3|\beta_1^{(1)}|^2|\beta_i^{(2)}|^2}\Big),\notag
\end{align}
where we use the Lemma \ref{lemma:estimate_of_producr_of_G_ii} in the last step.

Now, we consider the next non self-renormalization term
\begin{align}\label{eq:cumulant2}
    &\sum_{k \geq 2} \frac{\kappa((ba)_2, \alpha_1,\alpha_2)}{k!} {\sum_{a b}}^\prime \sum_{k_1+k_2=k-1} \sum_{\alpha_1, \alpha_2} \mathbb{E}\big[\langle\Delta^{ba} \partial_{\alpha_1} G_{ii}^{(2)}\rangle \langle\Delta^{ba} \partial_{\alpha_2} G_{ii}^{(2)}\rangle \big]\\ 
    &\quad= \mathcal{O}_{\prec}\Big(\frac{1}{n^{3/2} \eta_i^2|\beta_1^{(1)}|^2|\beta_i^{(2)}|^2}\Big),\notag
\end{align}
where we use the estimate $|\kappa((ba)_2, \alpha_1,\alpha_2)|/k! \lesssim n^{-(k+1) / 2}$ of the cumulants from \eqref{eqn: cumulant estimate}. Note that from Lemma \ref{lemma:estimate_of_producr_of_G_ii}, we have the bound $|(G^{(2)}_{ii})_{ab}| \prec 1/\eta_i$. After applying the $\alpha$ derivatives on these terms and applying Lemma \ref{lemma:estimate_of_producr_of_G_ii}, we get back the same bound. This justifies the last inequality in equation \eqref{eq:cumulant2}.

So far, we have examined the non self renormalized terms in the \eqref{eq:cumulant}. To analyze the self renormalized terms, which are listed at the end of the first line, in the second line, and in the last line of \eqref{eq:cumulant}, we do a further cumulant expansion as follows;

\begin{align}\label{eq:cumulant3}
    &\mathbb{E}\big[\langle \Delta^{ba} \partial_{\alpha_1}G_{ii}^{(2)}\rangle \langle \underline{W^{(2)}\partial_{\alpha_2}G_{ii}^{(2)}} \rangle\big]\notag\\
    &= \mathbb{E} \big[\widetilde{\mathbb{E}}\langle \Delta^{ba} \partial_{\alpha_1}(G_i^{(2)}\widetilde{W}^{(2)}G_{ii}^{(2)}+G_{ii}^{(2)}\widetilde{W}^{(2)}G_{i}^{(2)})\rangle \langle \widetilde{W}^{(2)} \partial_{\alpha_2}G_{ii}^{(2)}\rangle \big]\notag\\
    &\;\;\;\;\;+\frac{\gamma_1}{n}{\sum_{cd}}'\mathbb{E}\big[\partial_{dc}\big(\langle \Delta^{dc}\partial_{\alpha_2}G_{ii}^{(2)}\rangle \langle \Delta^{ba}\partial_{\alpha_1}G_{ii}^{(2)} \rangle\big)\big]\notag\\
    &\;\;\;\;\;+\sum_{l\geq 2}{\sum_{cd}}^\prime\sum_{l_{1}+l_{2}=l}\sum_{\vartheta_1,\vartheta_2} \mathbb{E}\big[\langle \Delta^{ba}\partial_{\alpha_1} \partial_{\vartheta_1} G_{ii}^{(2)} \rangle \langle \Delta^{cd}\partial_{\alpha_2} \partial_{\vartheta_2} G_{ii}^{(2)} \rangle\big]\notag\\
    &=\frac{1}{n^2}\mathbb{E}\big[\langle \partial_{\alpha_1}(G_{ii}^{(2)}\Delta^{ba}G_i^{(2)}+G_i^{(2)}\Delta^{ba}G_{ii}^{(2)})\partial_{\alpha_2}(G_{ii}^{(2)})\rangle\big]+\frac{\gamma_1}{n}{\sum_{cd}}'\mathbb{E}\big[\partial_{dc}\big(\langle \Delta^{dc}\\
    &\;\;\;\;\;\partial_{\alpha_2}G_{ii}^{(2)}\rangle\big) \langle \Delta^{ba}\partial_{\alpha_1}G_{ii}^{(2)} \rangle\big]+\frac{\gamma_1}{n}{\sum_{cd}}'\mathbb{E}\big[\langle \Delta^{dc}\partial_{\alpha_2}G_{ii}^{(2)}\rangle \partial_{dc}\big(\langle \Delta^{ba}\partial_{\alpha_1}G_{ii}^{(2)} \rangle\big)\big]\notag\\
    &\;\;\;\;\;+\sum_{l\geq 2}{\sum_{cd}}^\prime\sum_{l_{1}+l_{2}=l}\sum_{\vartheta_1,\vartheta_2} \frac{\kappa(dc, \vartheta_{1}, \vartheta_{2})}{l!}\mathbb{E}\big[\langle \Delta^{ba}\partial_{\alpha_1} \partial_{\vartheta_1} G_{ii}^{(2)} \rangle \langle \Delta^{cd}\partial_{\alpha_2} \partial_{\vartheta_2} G_{ii}^{(2)} \rangle\big],\notag
\end{align}

where $\vartheta_i \in\{(cd)_2, (dc)_2\}^{l_i}.$ After inserting the first line of \eqref{eq:cumulant3} back into \eqref{eq:cumulant}, we obtain an overall factor of $n^{-3-(k+1) / 2}$ as well as the $\sum'_{ab}$ summation over some $\partial_\alpha(\mathcal{G})_{ab}$, where $\mathcal{G}$ is a product of five $G_i^{(2)}$'s. The sum $\sum'_{ab}$ has $\mathcal{O}(n^{2})$ many terms. Therefore, using the bound $\left|\partial_\alpha(\mathcal{G})_{a b}\right| \prec  \eta_i^{-4},$ we can estimate the sum by $n^{2}n^{-9 / 2} \eta_i^{-4}=n^{-5 / 2} \eta_i^{-4}$ since $k \geq 2$.

Now, we turn to the last line of \eqref{eq:cumulant3}. When this is inserted back into \eqref{eq:cumulant}, we obtain a total perfecter of $n^{-(k+l) / 2-3}$, and a summation $\sum'_{a b}\sum'_{c d}$ over
$$
\big(\partial_{\alpha_1} \partial_{\vartheta_1} G_{ii}^{(2)}\big)_{ab}\big(\partial_{\alpha_2} \partial_{\vartheta_2} G_{ii}^{(2)}\big)_{cd}.
$$

Using the fact that $\partial_{ab}G^{(t)}=-G^{(t)}\Delta^{ab}G^{(t)}$, we observe that a single derivative of $G^{(t)}$ yields one $\Delta.$ Therefore, the term $\Delta^{ba}\partial_{\alpha_1} \partial_{\vartheta_1} G_{ii}^{(2)}$ will have $|\alpha_{1}|+|\vartheta_{1}|+1$ many $\Delta$s. Similarly, $\Delta^{cd}\partial_{\alpha_2} \partial_{\vartheta_2} G_{ii}^{(2)} $ will have $|\alpha_{2}|+|\vartheta_{2}|+1$ many $\Delta$s. The indices of $\Delta$s, such as $\Delta^{ab}, \Delta^{cd}$ etc, determine the number of diagonal and off diagonal factors of $G$ in the product. When $k = l = 2$, for $|\alpha_1| + |\vartheta_1| = 3$ or $|\alpha_1| + |\vartheta_1| = 1$, all the entries are diagonal, which gives a bound of $n^{-1} \eta_i^{-2}$. In all other cases, at least two factors are off-diagonal, so the bound is $n^{-1} \eta_i^{-2}$ multiplied by a Ward improvement factor of $(n\eta_i)^{-1}$, which is $\mathcal{O}(n^{-2}\eta_i^{-3})$. Now, when $k+l\geq 5$, the bound are of smaller order.

Therefore, finally we obtain

\begin{align}\label{eq:cumulant4}
& \frac{1}{n^{(k+1) / 2}} \sum_{\substack{k_1+k_2=k, k \geq 2}} \sum_{ab} \sum_{\alpha_1, \alpha_2} \mathbb{E}\big[\langle\Delta^{ba} \partial_{\alpha_1} G_{ii}^{(2)}\rangle \langle \underline{W^{(2)} \partial_{\alpha_2} G_{ii}^{(2)}}\rangle\big]\notag \\
& \quad=\mathcal{O}_{\prec}\Big(\frac{1}{n^2 \eta_i^4|\beta_1^{(1)}|^2|\beta_i^{(2)}|^2}\Big).
\end{align}

Finally, we consider the remaining terms in the first line of \eqref{eq:cumulant}. After applying the cumulant expansion, we obtain the estimate
\begin{equation}\label{eq:cumulant5}
    \frac{1}{n^2} \widetilde{\mathbb{E}}\big[\langle(G_{ii}^{(2)} \widetilde{W}^{(2)} G_i^{(2)}+G_i^{(2)} \widetilde{W}^{(2)} G_{ii}^{(2)})^2\rangle \big]=\mathcal{O}_{\prec}\Big(\frac{1}{n^2 \eta_i^4|\beta_1^{(1)}|^2|\beta_i^{(2)}|^2}\Big).
\end{equation}

The above bound is obtained by piecing the following bounds together; 
$$
\begin{aligned}
|\langle G_{ii}^{(2)} G_{i}^{(2)} \rangle| & \prec \frac{1}{\eta_i^2|\beta_1^{(1)}||\beta_i^{(2)}|},  
\end{aligned}
$$
which are from Lemma \ref{lemma:estimate_of_producr_of_G_ii}.

Similarly, we can estimate the remaining terms. For example, the second line of \eqref{eq:cumulant3} can be estimated in the same way as we estimated the third line of \eqref{eq:cumulant3}, with $|\vartheta_1| \leq 1$ and $|\vartheta_2| = 0$ for the first term, and $|\vartheta_2| \leq 1$ and $|\vartheta_1| = 0$ for the second term. Likewise, the second line of \eqref{eq:cumulant} can be estimated as the last line of \eqref{eq:cumulant} in the cases $|\alpha_1| = 1, |\alpha_2| = 0$ or $|\alpha_1| = 0, |\alpha_2| = 1$. However, the estimates obtained from these will be less than or equal to the previously obtained estimates. Thus, by combining \eqref{eq:cumulant..}-\eqref{eq:cumulant5} we conclude the proof.
\end{proof}    
\end{lemma}
Similarly we can prove the following estimate.
\begin{cor}\cite{cipolloni2021fluctuation}\label{cor:overline_cor}Let $G_{i}^{(t)}$, $G_{i1i}^{(t)},$ $A_{i}^{(t)}$ be same as defined in Lemma \ref{lemma:estimate_of_producr_of_G_ii}. Then we have
    \begin{align*}
    & \mathbb{E}|\langle\underline{ G_1^{(t)}A_1^{(t)}EG_i^{(t)} A_i^{(t)}W^{(t)}G_i^{(t)}E^\dag}\rangle|^2=\mathbb{E}|\langle \underline{G_{i1i}^{(t)}}\rangle|^2\\
        &\;\;\;\; \lesssim \Big(\frac{1}{n \eta_1\eta_i\eta_{*}^{1i}|\beta_1^{(t)}||\beta_i^{(t)}|}\Big)^2\notag,
\end{align*}
where $\eta_*^{1i} = \min \{\eta_1,\eta_i\}.$
\end{cor}

\section*{Declarations}

\subsection*{Acknowledgments} We would like to thank L\'aszl\'o Erd\H{o}s for helpful discussions. We would also like to thank Giorgio Cipolloni for pointing out some technical errors in the first draft.

\subsection*{Funding} Indrajit Jana's research is partially supported by INSPIRE Fellowship\\DST/INSPIRE/04/2019/000015, Dept. of Science and Technology, Govt. of India.\\

Sunita Rani's research is fully supported by the University Grant Commission (UGC), New Delhi.


\bibliographystyle{abbrv} 

\end{document}